\documentclass{sl2art}

\makeatletter
\def\qed@warning{}
\makeatother
\usepackage{tikz-cd}
\usepackage{todonotes}

\usepackage{graphicx,subcaption}
\graphicspath{{fig/}}

\usepackage{enumerate}
\usepackage{mathtools}

\RequirePackage[noabbrev,capitalize]{cleveref}
\crefname{equation}{equation}{equations}
\Crefname{equation}{Equation}{Equations}

\newcommand{\ZZ}{\mathbb{Z}}

\newcommand{\CC}{\mathbb{C}}
\newcommand{\HH}{\mathbb{H}}

\newcommand{\set}[1]{\left\{ \def\given{\ \middle| \ }  #1 \right\}  }

\newcommand{\slg}{\mathrm{SL}_2(\CC)}
\newcommand{\pslg}{\mathrm{PSL}_2(\CC)}
\newcommand{\glg}{\mathrm{GL}_2(\CC)}
\newcommand{\hol}{\mathrm{Hol}}

\newcommand{\catformat}[1]{\mathsf{#1}}
\newcommand{\modc}[1]{{#1}\catformat{-Mod}}

\newcommand{\cs}{\operatorname{CS}}

\NewDocumentCommand{\chiset}{}{\mathsf{X}}

\newcommand{\rvar}[1]{\mathfrak{R}_{#1}}
\NewDocumentCommand{\svar}{m O{}}{\mathfrak{P}_{#1}^{#2}}
\newcommand{\avar}[1]{\mathfrak{A}_{#1}}
\newcommand{\bvar}[1]{\mathfrak{B}_{#1}}

\newcommand{\mer}{\mathfrak{m}}
\newcommand{\lon}{\mathfrak{l}}
\newcommand{\lonb}{\tilde{\mathfrak{l}}}
\newcommand{\comp}[1]{S^3 \setminus #1}
\newcommand{\extr}[1]{M_{#1}}

\NewDocumentCommand{\homol}{D[]{1} m D[]{} }{\operatorname{H}_#1(#2)^{#3}}
\NewDocumentCommand{\cohomol}{D[]{1} m D[]{} }{\operatorname{H}^#1(#2)^{#3}}

\newcommand{\crossrat}[4]{\left[ #1 : #2 : #3 : #4 \right]}

\newcommand{\Rmat}{\mathcal{R}}

\NewDocumentCommand{\qsl}{O{\xi}}{\mathcal{U}_{#1}(\mathfrak{sl}_2)}
\NewDocumentCommand{\weyl}{O{\xi}}{\mathcal{W}_{#1}}
\renewcommand{\hbar}{\hslash}

\NewDocumentCommand{\W}{}{\mathcal{W}}

\newcommand{\defeq}{:=}
\DeclareMathOperator{\tr}{tr}
\DeclareMathOperator{\id}{id}
\DeclareMathOperator{\op}{op}

\newcommand{\fib}[1]{\operatorname{Fib}_{#1}}

\NewDocumentCommand{\kinv}{ O{N}m }{\mathrm{V}_{#1}(#2)}
\NewDocumentCommand{\kinvname}{O{N}}{\mathrm{V}_{#1}}

\declaretheorem[style=theorem]{proposition}
\declaretheorem[style=theorem,sibling=proposition]{theorem}
\declaretheorem[style=theorem,sibling=proposition]{lemma}

\declaretheorem[style=definition,sibling=proposition]{definition}
\declaretheorem[style=definition,sibling=proposition]{remark}
\declaretheorem[style=definition,sibling=proposition]{example}
\declaretheorem[style=definition,numberwithin=,title={Problem}]{problem}

\usepackage{biblatex}
\title{Hyperbolic structures on link complements, octahedral decompositions, and quantum \(\mathfrak{sl}_2\)}
\author{Calvin McPhail-Snyder}
\address{Duke University}
\email{calvin@sl2.site}

\subjclass{57K32, 57K10, 20G42}
\keywords{octahedral decomposition, biquandle, Kac-de Concini quantum group}

\begin{document}

\begin{abstract}
  Hyperbolic structures on link complements (equivalently, representations of the fundamental group into $\operatorname{SL}_2(\mathbb{C})$) can be described algebraically by using the \defemph{octahedral decomposition} determined by a link diagram.
  The decomposition (like any ideal triangulation) gives a set of \defemph{gluing equations} in \defemph{shape parameters} whose solutions are hyperbolic structures.
  We show that these equations can be obtained from Kashaev-Reshetikhin's braiding on the \defemph{Kac-de Concini quantum group} $\mathcal{U}_\xi(\mathfrak{sl}_2)$ at a root of unity $\xi$.
  This braiding gives coordinates on the $\operatorname{SL}_2(\mathbb{C})$ representation variety of a link and our work shows how to interpret these geometrically.
\end{abstract}

\maketitle

\tableofcontents
\section{Introduction}
For \(L\) a link in \(S^3\) a representation \(\rho : \pi_1(S^3 \setminus L) \to \slg\) is a (generalized) hyperbolic structure because the isometry  \(\operatorname{Isom}(\HH^3) = \pslg\) of hyperbolic \(3\)-space is double-covered by \(\slg\).
It is frequently useful to describe the hyperbolic structure by ideally triangulating \(S^{3} \setminus L\) and geometrizing the tetrahedra in terms of shape parameters \cite{ThurstonNotes}.

Given a diagram \(D\) of a link \(L\) one can define an ideal triangulation of \(S^{3} \setminus L\) minus two points called the \defemph{octahedral decomposition} \cite{ThurstonDNotes,Kashaev1995}.
This triangulation is convenient when working with link diagrams, as in the Reshetikhin-Turaev construction in quantum topology.
Understanding its geometry is frequently relevant to the Volume Conjecture.
\textcite{Hikami2015} showed how to determine the shape parameters of the octahedral decomposition in terms of cluster variables, with crossings of the diagram corresponding to cluster mutations.
\textcite{Kim2016,Kim2019} further studied the geometry of the octahedral decomposition and showed how to parametrize its hyperbolic structures using a simpler, more convenient set of variables; in our paper we call these \defemph{\(\chi\)-colorings} of link diagrams.
We can view \(\chi\)-colorings as a coordinate system on the \defemph{representation variety}, the space of representations \(\rho : \pi_{1}(\comp{L}) \to \slg\).

In this paper we show that \(\chi\)-colorings arise naturally in quantum topology.
\citeauthor{Kashaev2004} showed how to define a braiding on \(\qsl\) at \(q = \xi\) a root of unity.
This braiding is really a family of braidings parametrized by central characters of \(\qsl\), so it leads to invariants tangles with a choice of representation \(\rho\) of the complement into \(\slg\) \cite{Kashaev2005}.
Later \citeauthor{Blanchet2018} \cite{Blanchet2018} showed how to use these to define link invariants.
A significant technical problem in this construction is that one must express \(\rho\) in terms of a nonstandard coordinate system on the \(\slg\) representation variety, or in the language of \cite{Blanchet2018} a \defemph{generic biquandle factorization} of (the conjugation quandle of) \(\slg\).

In this paper we show that the generic biquandle factorization is natural in the context of the octahedral decomposition: by expressing \(\qsl\) as a subalgebra of a Weyl algebra (i.e.\ by using cluster-type coordinates) the factorized coordinates are essentially equivalent to the octahedral coordinates of \cite{Kim2016}.
This shows more directly the relationship between the BGPR invariant and hyperbolic geometry.
This perspective was originally discovered by the author and Resheitkhin in the course of our study of the BGPR braiding, and it proved crucial to resolving some normalization ambiguities in the braiding \cite{McPhailSnyderAlgebra}.
The holonomy \(R\)-matrices considered there were subsequently used to a define an (algebraic) quantization of the complex Chern-Simons invariant of a tangle exterior \cite{McPhailSnyderVolume}.
Again, the connection to octahedral decompositions and hyperbolic geometry is key.
Subsequent work has shown how to derive \(\chi\)-colorings directly from the Wirtinger presentation of the fundamental group \cite{McphailSnyder2024octahedralcoordinateswirtingerpresentation}.

The goal of this paper is to explain a connection between hyperbolic knot theory and the algebraic theory of quantum groups.
As such we have included standard background material on both topics to try to make the paper accessible to readers familiar with only one of these areas.

\subsection*{Plan of the paper}
\begin{itemize}
  \item In \Cref{sec:shaped-tangles} we define \(\chi\)-colorings of tangle diagrams and their associated representations.
  \item In \Cref{sec:quantum-groups} we explain the connection to quantum groups and cluster algebras: the colors \(\chi\) can be interpreted as characters on a central Hopf subalgebra of \(\qsl\).
  \item In \Cref{sec:geometry} we explain in detail how \(\chi\)-colorings determine hyperbolic structures on the octahedral decomposition.
    Most of these results were previously shown by \citeauthor{Kim2016} \cite{Kim2016,Kim2019}.
    We re-derive them in our conventions to make the paper self-contained.
  \item In \Cref{sec:gluing} we explain how in practice one can eliminate half the variables in the \(\chi\)-colorings, recovering the segment and region equations of \cite{Kim2016}.
    As an example, we compute all \(\chi\)-colorings of \((2,2n+1)\)-torus knots inducing to irreducible holonomy representations.
\end{itemize}

\subsection*{Acknowledgements}
I would like to thank Ian Agol for several helpful conversations about hyperbolic knot theory and Seokbeom Yoon for informing me about the recurrences in \cref{rem:twist-different-eigvals}, Matthias Goerner for helping me with some subtleties about decorations and the anonymous referee for helpful comments on the structure of the paper and pointing out some technical issues.

Some of the work in this paper was carried out while I was a visiting scholar at UNC Chapel Hill and I would like to thank the Mathematics Department and David E.\ V.\ Rose for their hospitality.

\section{\texorpdfstring{\(\chi\)}{χ}-colorings of link diagrams}
\label{sec:shaped-tangles}

We begin by defining \(\chi\)-colorings of link diagrams and their associated holonomy representations.

\subsection{Basic definitions}
\begin{definition}
  Let \(L\) be a link in \(S^3\) and \(D\) a diagram of \(L\).
  We assume all link diagrams are oriented.
  Thinking of \(D\) as a decorated \(4\)-valent graph \(G\) embedded in \(S^2\), the \defemph{segments} of \(D\) are the edges%
  \note{%
    Usually these are called the ``edges'' of the diagram, but we do not want to confuse them with edges of ideal polyhedra.
  }
  of \(G\).
  A \defemph{region} of a diagram is a connected component of the complement of \(G\), equivalently a vertex of the dual graph of \(G\).
\end{definition}
  For example, \cref{fig:figure-eight-labeled} shows an (oriented) diagram with the segments labeled.
\begin{marginfigure}
  %% Creator: Inkscape 1.2.2 (b0a8486541, 2022-12-01), www.inkscape.org
%% Accompanies image file 'figure-eight-labeled.pdf' (pdf, eps, ps)
\begingroup%
  \makeatletter%
  \providecommand\color[2][]{%
    \errmessage{(Inkscape) Color is used for the text in Inkscape, but the package 'color.sty' is not loaded}%
    \renewcommand\color[2][]{}%
  }%
  \providecommand\transparent[1]{%
    \errmessage{(Inkscape) Transparency is used (non-zero) for the text in Inkscape, but the package 'transparent.sty' is not loaded}%
    \renewcommand\transparent[1]{}%
  }%
  \providecommand\rotatebox[2]{#2}%
  \newcommand*\fsize{\dimexpr\f@size pt\relax}%
  \newcommand*\lineheight[1]{\fontsize{\fsize}{#1\fsize}\selectfont}%
  \ifx\svgwidth\undefined%
    \setlength{\unitlength}{132.93800354bp}%
    \ifx\svgscale\undefined%
      \relax%
    \else%
      \setlength{\unitlength}{\unitlength * \real{\svgscale}}%
    \fi%
  \else%
    \setlength{\unitlength}{\svgwidth}%
  \fi%
  \global\let\svgwidth\undefined%
  \global\let\svgscale\undefined%
  \makeatother%
  \begin{picture}(1,0.73142784)%
    \lineheight{1}%
    \setlength\tabcolsep{0pt}%
    \put(0,0){\includegraphics[width=\unitlength,page=1]{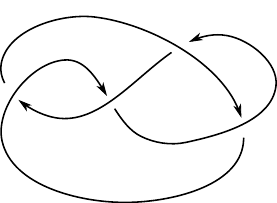}}%
    \put(0.18574584,0.53607107){\makebox(0,0)[lt]{\lineheight{1.25}\smash{\begin{tabular}[t]{l}$1$\end{tabular}}}}%
    \put(0.49378445,0.15159342){\makebox(0,0)[lt]{\lineheight{1.25}\smash{\begin{tabular}[t]{l}$2$\end{tabular}}}}%
    \put(0.86669712,0.6034901){\makebox(0,0)[lt]{\lineheight{1.25}\smash{\begin{tabular}[t]{l}$3$\end{tabular}}}}%
    \put(0.51156609,0.41261917){\makebox(0,0)[lt]{\lineheight{1.25}\smash{\begin{tabular}[t]{l}$4$\end{tabular}}}}%
    \put(0.13125042,0.24642853){\makebox(0,0)[lt]{\lineheight{1.25}\smash{\begin{tabular}[t]{l}$5$\end{tabular}}}}%
    \put(0.28308073,0.69587022){\makebox(0,0)[lt]{\lineheight{1.25}\smash{\begin{tabular}[t]{l}$6$\end{tabular}}}}%
    \put(0.82184746,0.44367152){\makebox(0,0)[lt]{\lineheight{1.25}\smash{\begin{tabular}[t]{l}$7$\end{tabular}}}}%
    \put(0.32516562,0.03348949){\makebox(0,0)[lt]{\lineheight{1.25}\smash{\begin{tabular}[t]{l}$8$\end{tabular}}}}%
  \end{picture}%
\endgroup%

  \caption{A diagram of the figure-eight knot, with the \(8\) segments indexed by \(1, \dots, 8\).}
  \label{fig:figure-eight-labeled}
\end{marginfigure}
In an oriented diagram all crossings are positive or negative, as shown in \cref{fig:crossing-types}.
Our preference is to read crossings left-to-right.
\begin{figure}
  \centering
  %% Creator: Inkscape 1.2.2 (b0a8486541, 2022-12-01), www.inkscape.org
%% Accompanies image file 'crossing-types.pdf' (pdf, eps, ps)
\begingroup%
  \makeatletter%
  \providecommand\color[2][]{%
    \errmessage{(Inkscape) Color is used for the text in Inkscape, but the package 'color.sty' is not loaded}%
    \renewcommand\color[2][]{}%
  }%
  \providecommand\transparent[1]{%
    \errmessage{(Inkscape) Transparency is used (non-zero) for the text in Inkscape, but the package 'transparent.sty' is not loaded}%
    \renewcommand\transparent[1]{}%
  }%
  \providecommand\rotatebox[2]{#2}%
  \newcommand*\fsize{\dimexpr\f@size pt\relax}%
  \newcommand*\lineheight[1]{\fontsize{\fsize}{#1\fsize}\selectfont}%
  \ifx\svgwidth\undefined%
    \setlength{\unitlength}{283.55544662bp}%
    \ifx\svgscale\undefined%
      \relax%
    \else%
      \setlength{\unitlength}{\unitlength * \real{\svgscale}}%
    \fi%
  \else%
    \setlength{\unitlength}{\svgwidth}%
  \fi%
  \global\let\svgwidth\undefined%
  \global\let\svgscale\undefined%
  \makeatother%
  \begin{picture}(1,0.27643014)%
    \lineheight{1}%
    \setlength\tabcolsep{0pt}%
    \put(0,0){\includegraphics[width=\unitlength,page=1]{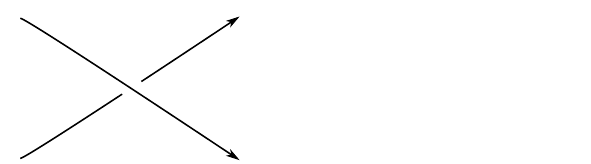}}%
    \put(-0.00351287,0.24426852){\makebox(0,0)[lt]{\lineheight{1.25}\smash{\begin{tabular}[t]{l}$1$\end{tabular}}}}%
    \put(-0.00351287,0.00621985){\makebox(0,0)[lt]{\lineheight{1.25}\smash{\begin{tabular}[t]{l}$2$\end{tabular}}}}%
    \put(0.41174984,0.24426852){\makebox(0,0)[lt]{\lineheight{1.25}\smash{\begin{tabular}[t]{l}$2'$\end{tabular}}}}%
    \put(0.41174984,0.00621985){\makebox(0,0)[lt]{\lineheight{1.25}\smash{\begin{tabular}[t]{l}$1'$\end{tabular}}}}%
    \put(0,0){\includegraphics[width=\unitlength,page=2]{crossing-types.pdf}}%
    \put(0.56563205,0.24427706){\makebox(0,0)[lt]{\lineheight{1.25}\smash{\begin{tabular}[t]{l}$1$\end{tabular}}}}%
    \put(0.56563205,0.00622839){\makebox(0,0)[lt]{\lineheight{1.25}\smash{\begin{tabular}[t]{l}$2$\end{tabular}}}}%
    \put(0.9808948,0.24427706){\makebox(0,0)[lt]{\lineheight{1.25}\smash{\begin{tabular}[t]{l}$2'$\end{tabular}}}}%
    \put(0.9808948,0.00622839){\makebox(0,0)[lt]{\lineheight{1.25}\smash{\begin{tabular}[t]{l}$1'$\end{tabular}}}}%
  \end{picture}%
\endgroup%

  \caption{Positive (left) and negative (right) crossings.}
  \label{fig:crossing-types}
\end{figure}
As shown there, we usually refer to the segments at a given crossing by \(1\), \(2\), \(1'\), and \(2'\).
We similarly refer to the regions touching the crossing as \(N\), \(S\), \(E\), and \(W\).
The labeling conventions are summarized in \cref{fig:crossing-regions}.
\begin{marginfigure}
  %% Creator: Inkscape 1.2.2 (b0a8486541, 2022-12-01), www.inkscape.org
%% Accompanies image file 'crossing-regions.pdf' (pdf, eps, ps)
\begingroup%
  \makeatletter%
  \providecommand\color[2][]{%
    \errmessage{(Inkscape) Color is used for the text in Inkscape, but the package 'color.sty' is not loaded}%
    \renewcommand\color[2][]{}%
  }%
  \providecommand\transparent[1]{%
    \errmessage{(Inkscape) Transparency is used (non-zero) for the text in Inkscape, but the package 'transparent.sty' is not loaded}%
    \renewcommand\transparent[1]{}%
  }%
  \providecommand\rotatebox[2]{#2}%
  \newcommand*\fsize{\dimexpr\f@size pt\relax}%
  \newcommand*\lineheight[1]{\fontsize{\fsize}{#1\fsize}\selectfont}%
  \ifx\svgwidth\undefined%
    \setlength{\unitlength}{113.47111416bp}%
    \ifx\svgscale\undefined%
      \relax%
    \else%
      \setlength{\unitlength}{\unitlength * \real{\svgscale}}%
    \fi%
  \else%
    \setlength{\unitlength}{\svgwidth}%
  \fi%
  \global\let\svgwidth\undefined%
  \global\let\svgscale\undefined%
  \makeatother%
  \begin{picture}(1,0.700154)%
    \lineheight{1}%
    \setlength\tabcolsep{0pt}%
    \put(0,0){\includegraphics[width=\unitlength,page=1]{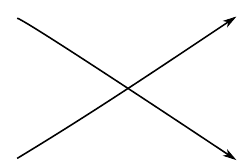}}%
    \put(-0.02199761,0.61980597){\makebox(0,0)[lt]{\lineheight{1.25}\smash{\begin{tabular}[t]{l}$1$\end{tabular}}}}%
    \put(-0.02199761,0.02494095){\makebox(0,0)[lt]{\lineheight{1.25}\smash{\begin{tabular}[t]{l}$2$\end{tabular}}}}%
    \put(1.01571146,0.61980597){\makebox(0,0)[lt]{\lineheight{1.25}\smash{\begin{tabular}[t]{l}$2'$\end{tabular}}}}%
    \put(1.01571146,0.02494095){\makebox(0,0)[lt]{\lineheight{1.25}\smash{\begin{tabular}[t]{l}$1'$\end{tabular}}}}%
    \put(0.48694249,0.45456569){\makebox(0,0)[lt]{\lineheight{1.25}\smash{\begin{tabular}[t]{l}$N$\end{tabular}}}}%
    \put(0.48694249,0.19018121){\makebox(0,0)[lt]{\lineheight{1.25}\smash{\begin{tabular}[t]{l}$S$\end{tabular}}}}%
    \put(0.68523089,0.28932538){\makebox(0,0)[lt]{\lineheight{1.25}\smash{\begin{tabular}[t]{l}$E$\end{tabular}}}}%
    \put(0.28865412,0.28932538){\makebox(0,0)[lt]{\lineheight{1.25}\smash{\begin{tabular}[t]{l}$W$\end{tabular}}}}%
  \end{picture}%
\endgroup%

  \caption{Segments and regions near a crossing.}
  \label{fig:crossing-regions}
\end{marginfigure}

\begin{definition}
  A \defemph{\(\chi\)-color} is a triple of nonzero complex numbers.
  We write \(\chiset = (\CC \setminus \{0\})^{3}\) for the set of \(\chi\)-colors.
  We usually abbreviate $(a, b, m) = \chi \in \chiset$, and when is assigned to a segment $i$ of a tangle diagram we write $\chi_i = (a_i, b_i, m_i)$.
\end{definition}

We can think of a \(\chi\)-color as:
\begin{itemize}
  \item An element of the group \(\slg^*\):
    \begin{align*}
      (a, b, m)
        &=
        \left(
          \begin{bmatrix}
            a & 0 \\
            (a - 1/m)/b & 1
          \end{bmatrix}
          ,
          \begin{bmatrix}
            1 & (a-m)b \\
            0 & a
          \end{bmatrix}
        \right)
        \\
        &\in
        \left\{
          \left(
            \begin{bmatrix}
              \kappa & 0 \\
              \phi & 1
            \end{bmatrix}
            ,
            \begin{bmatrix}
              1 & \epsilon \\
              0 & \kappa
            \end{bmatrix}
          \right)
          \middle |
          \kappa \ne 0
        \right\}
        = \slg^*
        \subseteq
        \operatorname{GL}_2(\CC)
        \times
        \operatorname{GL}_2(\CC)
    \end{align*}
    Here \(\slg^*\) is the \defemph{Poisson dual group} \cite[Section 0.1]{McPhailSnyderThesis} of \(\slg\).
  \item A character on a central subalgebra \(\mathcal{Z}_{0} \subset \qsl\) for \(\xi = e^{\pi i/ N}\) a root of unity; the full center \(\mathcal{Z}\) is a \(N\)-fold cover of \(\mathcal{Z}_{0}\).
  \item Data determining the shapes (complex dihedral angles) of the octahedral decomposition of a link diagram.
    In the language of \textcite{Kim2016} the \(b_i\) and \(m_i\) correspond to segment variables and \(a_i\) and \(m_i\) to (ratios of) region variables.
\end{itemize}

\begin{definition}
  \label{def:braiding}
  The \defemph{braiding} is the birational map \(B : \chiset^{2} \to \chiset^{2}\) given by $B(\chi_1, \chi_2) = (\chi_{2'}, \chi_{1'})$, where
  \begin{gather}
    \label{eq:a-transf-positive}
    \begin{aligned}
      a_{1'}
      &=
      a_1 A^{-1}
      \\
      a_{2'}
      &=
      a_2 A
      \\
      A &= 1 - \frac{m_1 b_1}{b_2} \left(1 - \frac{a_1}{m_1}\right)\left(1 - \frac{1}{m_2 a_2}\right)
    \end{aligned}
    \\
    \label{eq:b-transf-positive}
    \begin{aligned}
      b_{1'}
      &=
      \frac{m_2 b_2}{m_1}
      \left(
        1 - m_2 a_2 \left( 1 - \frac{b_2}{m_1 b_1} \right)
      \right)^{-1}
      \\
      b_{2'}
      &=
      b_1
      \left(
        1 - \frac{m_1}{a_1}\left( 1 - \frac{b_2}{m_1 b_1} \right)
      \right)
    \end{aligned}
    \\
    \label{eq:m-transf-positive}
    \begin{aligned}
      m_{1'}
      &= m_1
      &
      m_{2'}
      &= m_2
    \end{aligned}
  \end{gather}
  We think of $B$ as being associated to a positive crossing with incoming strands $1$ and $2$ and outgoing strands $2'$ and $1' $, as in \cref{fig:crossing-regions}.
  The inverse map $B^{-1}(\chi_1, \chi_2) = (\chi_{2'}, \chi_{1'})$ is given by
  \begin{gather}
    \label{eq:a-transf-negative}
    \begin{aligned}
      a_{1'}
      &=
      a_1 \tilde A^{-1}
      \\
      a_{2'}
      &=
      a_2 \tilde A
      \\
      \tilde A
      &=
      1 - \frac{b_2}{m_1 b_1}\left(1 - m_1 a_1 \right)\left(1 - \frac{m_2}{a_2}\right).
    \end{aligned}
    \\
    \label{eq:b-transf-negative}
    \begin{aligned}
      b_{1'}
      &=
      \frac{m_2 b_2}{m_1}
      \left(
        1 - \frac{a_2}{m_2} \left( 1 - \frac{m_1 b_1}{b_2} \right)
      \right)
      \\
      b_{2'}
      &=
      b_1
      \left(
        1 - \frac{1}{m_1 a_1} \left( 1 - \frac{m_1 b_1}{b_2} \right)
      \right)^{-1}
    \end{aligned}
    \\
    \label{eq:m-transf-negative}
    \begin{aligned}
      m_{1'}
      &= m_1
      &
      m_{2'}
      &= m_2
    \end{aligned}
  \end{gather}
\end{definition}

\(B\) is a birational map satisfying the braid relation
\[
  (B \times \id)
  (\id \times B)
  (B \times \id)
  =
  (\id \times B)
  (B \times \id)
  (\id \times B)
\]
and some related conditions so \((B, \chiset)\) is a \defemph{generically defined biquandle} as defined in \cite[Section 5]{Blanchet2018}.

\begin{definition}
  \label{def:shaping}
  We say that a tangle diagram $D$ is \defemph{\(\chi\)-colored} if its segments are assigned colors $\{\chi_i\}$ so that at each positive crossing (labeled as in \cref{fig:crossing-regions}) we have $B(\chi_1, \chi_2) = (\chi_{2'}, \chi_{1'})$, and similarly for negative crossings and $B^{-1}$.
  We also require that all of the components of \(\chi_{1'}\) and \(\chi_{2}\) lie in \(\CC^{\times}\).
  For example, this means that at a positive crossing we must assign \(\chi_1\) and \(\chi_2\) so that
  \[
    A = 
      1 - \frac{b_2}{m_1 b_1}\left(1 - m_1 a_1 \right)\left(1 - \frac{m_2}{a_2}\right)
  \]
  is not \(0\) or \(\infty\).
\end{definition}

\begin{marginfigure}
  \centering
  %% Creator: Inkscape 1.2.2 (b0a8486541, 2022-12-01), www.inkscape.org
%% Accompanies image file 'trefoil.pdf' (pdf, eps, ps)
\begingroup%
  \makeatletter%
  \providecommand\color[2][]{%
    \errmessage{(Inkscape) Color is used for the text in Inkscape, but the package 'color.sty' is not loaded}%
    \renewcommand\color[2][]{}%
  }%
  \providecommand\transparent[1]{%
    \errmessage{(Inkscape) Transparency is used (non-zero) for the text in Inkscape, but the package 'transparent.sty' is not loaded}%
    \renewcommand\transparent[1]{}%
  }%
  \providecommand\rotatebox[2]{#2}%
  \newcommand*\fsize{\dimexpr\f@size pt\relax}%
  \newcommand*\lineheight[1]{\fontsize{\fsize}{#1\fsize}\selectfont}%
  \ifx\svgwidth\undefined%
    \setlength{\unitlength}{145.25805473bp}%
    \ifx\svgscale\undefined%
      \relax%
    \else%
      \setlength{\unitlength}{\unitlength * \real{\svgscale}}%
    \fi%
  \else%
    \setlength{\unitlength}{\svgwidth}%
  \fi%
  \global\let\svgwidth\undefined%
  \global\let\svgscale\undefined%
  \makeatother%
  \begin{picture}(1,0.79602985)%
    \lineheight{1}%
    \setlength\tabcolsep{0pt}%
    \put(0,0){\includegraphics[width=\unitlength,page=1]{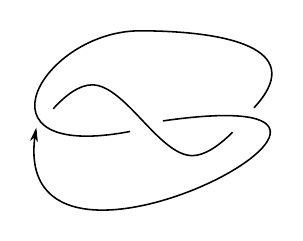}}%
    \put(0.03566389,0.17030626){\makebox(0,0)[lt]{\lineheight{1.25}\smash{\begin{tabular}[t]{l}$2$\end{tabular}}}}%
    \put(0.24367637,0.25541956){\makebox(0,0)[lt]{\lineheight{1.25}\smash{\begin{tabular}[t]{l}$4$\end{tabular}}}}%
    \put(0.56897119,0.21192431){\makebox(0,0)[lt]{\lineheight{1.25}\smash{\begin{tabular}[t]{l}$6$\end{tabular}}}}%
    \put(0.60584459,0.44884311){\makebox(0,0)[lt]{\lineheight{1.25}\smash{\begin{tabular}[t]{l}$5$\end{tabular}}}}%
    \put(0.04266923,0.54604185){\makebox(0,0)[lt]{\lineheight{1.25}\smash{\begin{tabular}[t]{l}$1$\end{tabular}}}}%
    \put(0.32519262,0.5285186){\makebox(0,0)[lt]{\lineheight{1.25}\smash{\begin{tabular}[t]{l}$3$\end{tabular}}}}%
  \end{picture}%
\endgroup%

  \caption{A diagram of the trefoil knot with labeled segments.}
  \label{fig:trefoil}
\end{marginfigure}

\begin{example}
  Consider the diagram of the trefoil in \cref{fig:trefoil} with the segments labeled by \(1, \dots, 6\). 
  A coloring \(\chi_i = (a_i, b_i, m_i), i = 1, \dots, 6\) of the diagram is valid if
  \[
    B(\chi_1, \chi_2) = (\chi_3, \chi_4)
    \text{, }
    B(\chi_3, \chi_4) = (\chi_5, \chi_6)
    \text{, and }
    B(\chi_5, \chi_6) = (\chi_1, \chi_2)
    .
  \]
  This immediately implies that \(m_1 = m_2 = \cdots = m_6 = m\); in general there is only one variable \(m\) for each component of the link.\note{Geometrically \(m\) is an eigenvalue of the holonomy of a meridian, and there is one conjugacy class of meridian per link component.}
  A family of solutions is given by
  \begin{align*}
    \chi_1 &= 
\left(\frac{b_{1} m - b_{2}}{b_{1} - b_{3}}, b_{1}, m\right)
    \\
    \chi_2 &=
\left(-\frac{b_{2}^{2} m^{2} - b_{2} b_{3} m + b_{1} b_{3}}{{\left(b_{1} m - b_{2}\right)} {\left(b_{2} m - b_{3}\right)}}, b_{2}, m\right)
    \\
    \chi_3 &=
\left(-\frac{b_{2} b_{3} m^{3} + b_{1} b_{3} m^{2} - b_{3}^{2} m^{2} - b_{1} b_{2} m + b_{1} b_{3}}{{\left(b_{2} m - b_{3}\right)} {\left(b_{1} - b_{3}\right)} m}, b_{3}, m\right)
    \\
    \chi_4 &=
\left(\frac{{\left(b_{2}^{2} m^{2} - b_{2} b_{3} m + b_{1} b_{3}\right)} m}{b_{2} b_{3} m^{3} + b_{1} b_{3} m^{2} - b_{3}^{2} m^{2} - b_{1} b_{2} m + b_{1} b_{3}}, \frac{{\left(b_{2} m - b_{3}\right)} b_{1}}{{\left(b_{2} m + b_{1} - b_{3}\right)} m}, m\right)
    \\
    \chi_5 &=
\left(-\frac{b_{2}^{2} m^{4} - b_{2} b_{3} m^{3} + b_{2} b_{3} m + b_{1} b_{3} - b_{3}^{2}}{{\left(b_{2} m - b_{3}\right)} {\left(b_{1} - b_{3}\right)} m}, \frac{b_{1} b_{2} m}{b_{2} m + b_{1} - b_{3}}, m\right)
    \\
    \chi_6 &=
\left(\frac{{\left(b_{2}^{2} m^{2} - b_{2} b_{3} m + b_{1} b_{3}\right)} m}{b_{2}^{2} m^{4} - b_{2} b_{3} m^{3} + b_{2} b_{3} m + b_{1} b_{3} - b_{3}^{2}}, -\frac{b_{1} b_{3}}{{\left(b_{2} m - b_{3}\right)} m}, m\right)
  \end{align*}
  where \(b_1, b_2, b_3\) can be freely chosen as long as none of the \(a_i\) or \(b_i\) are \(0\) or \(\infty\).
  It turns out that the choice of \(b_1, b_2, b_3\) does not affect the conjugacy class of the representation induced by the \(\chi\)-coloring.
  We show how to compute these solutions in \cref{sec:twist-region}.
\end{example}

At first glance the equations for all the \(a_i\) and \(b_i\) are difficult to solve.
We can simplify them by either eliminating the \(b_i\) and solving them in terms of the \(a_i\) or vice-versa.
For example, the solutions in the previous example were determined by first solving for the \(b_i\), then using them to determine the \(a_i\).
We discuss this in detail in \cref{sec:gluing}.
Alternatively a method to determine solutions directly from the Wirtinger presentation is given in \cite{McphailSnyder2024octahedralcoordinateswirtingerpresentation}.

\subsection{The holonomy of a \texorpdfstring{\(\chi\)}{χ}-colored diagram}

We now explain how a \(\chi\)-coloring determines a \(\slg\) representation of the link complement.
This uses a nonstandard presentation of the knot group as a groupoid.
In \cref{sec:quantum-groups} we explain how this groupoid arises from the representation theory of \(\qsl\), and in \cref{sec:geometry} we show it is natural when working with the face-pairing maps of the octahedral decomposition.

\begin{definition}
  Let \(D\) be a diagram of \(L\).
  The \defemph{fundamental groupoid} \(\Pi_1(D)\) of \(D\) has one object for each region of \(D\) and two generating morphisms \(x_j^{\pm}\) for each segment \(j\).
  These represent paths above and below the segment, as in \cref{fig:tangle-groupoid-generators}.
  Morphisms of \(\Pi_{1}(D)\) are formal composites of the generating morphisms subject to the following relations for each crossing
  \begin{equation}
    \label{eq:groupoid-relations}
    x_1^\pm x_2^\pm = x_{2'}^\pm x_{1'}^\pm
    \text{ and }
    \begin{cases}
      x_1^{-} x_{2}^{+}
      =
      x_{2'}^{+} x_{1'}^{-}
      &
      \text{ for a positive crossing, or}
      \\
      x_1^{+} x_{2}^{-}
      =
      x_{2'}^{-} x_{1'}^{+}
      &
      \text{ for a negative crossing.}
    \end{cases}
  \end{equation}
\end{definition}
\begin{marginfigure}
  %% Creator: Inkscape 1.2.2 (b0a8486541, 2022-12-01), www.inkscape.org
%% Accompanies image file 'tangle-groupoid-generators.pdf' (pdf, eps, ps)
\begingroup%
  \makeatletter%
  \providecommand\color[2][]{%
    \errmessage{(Inkscape) Color is used for the text in Inkscape, but the package 'color.sty' is not loaded}%
    \renewcommand\color[2][]{}%
  }%
  \providecommand\transparent[1]{%
    \errmessage{(Inkscape) Transparency is used (non-zero) for the text in Inkscape, but the package 'transparent.sty' is not loaded}%
    \renewcommand\transparent[1]{}%
  }%
  \providecommand\rotatebox[2]{#2}%
  \newcommand*\fsize{\dimexpr\f@size pt\relax}%
  \newcommand*\lineheight[1]{\fontsize{\fsize}{#1\fsize}\selectfont}%
  \ifx\svgwidth\undefined%
    \setlength{\unitlength}{140.23033333bp}%
    \ifx\svgscale\undefined%
      \relax%
    \else%
      \setlength{\unitlength}{\unitlength * \real{\svgscale}}%
    \fi%
  \else%
    \setlength{\unitlength}{\svgwidth}%
  \fi%
  \global\let\svgwidth\undefined%
  \global\let\svgscale\undefined%
  \makeatother%
  \begin{picture}(1,0.37438403)%
    \lineheight{1}%
    \setlength\tabcolsep{0pt}%
    \put(0,0){\includegraphics[width=\unitlength,page=1]{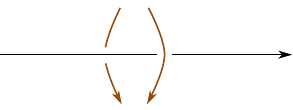}}%
    \put(0.57762114,0.26741715){\makebox(0,0)[lt]{\lineheight{1.25}\smash{\begin{tabular}[t]{l}$x_i^+$\end{tabular}}}}%
    \put(0.25672045,0.26741715){\makebox(0,0)[lt]{\lineheight{1.25}\smash{\begin{tabular}[t]{l}$x_i^-$\end{tabular}}}}%
    \put(0,0){\includegraphics[width=\unitlength,page=2]{tangle-groupoid-generators.pdf}}%
    \put(0.64180094,0.10696684){\makebox(0,0)[lt]{\lineheight{1.25}\smash{\begin{tabular}[t]{l}$i$\end{tabular}}}}%
  \end{picture}%
\endgroup%

  \caption{Generators of the fundamental groupoid \(\Pi_1(D)\) of a tangle diagram.}
  \label{fig:tangle-groupoid-generators}
\end{marginfigure}
\begin{marginfigure}
  %% Creator: Inkscape 1.2.2 (b0a8486541, 2022-12-01), www.inkscape.org
%% Accompanies image file 'tangle-groupoid-relations.pdf' (pdf, eps, ps)
\begingroup%
  \makeatletter%
  \providecommand\color[2][]{%
    \errmessage{(Inkscape) Color is used for the text in Inkscape, but the package 'color.sty' is not loaded}%
    \renewcommand\color[2][]{}%
  }%
  \providecommand\transparent[1]{%
    \errmessage{(Inkscape) Transparency is used (non-zero) for the text in Inkscape, but the package 'transparent.sty' is not loaded}%
    \renewcommand\transparent[1]{}%
  }%
  \providecommand\rotatebox[2]{#2}%
  \newcommand*\fsize{\dimexpr\f@size pt\relax}%
  \newcommand*\lineheight[1]{\fontsize{\fsize}{#1\fsize}\selectfont}%
  \ifx\svgwidth\undefined%
    \setlength{\unitlength}{144bp}%
    \ifx\svgscale\undefined%
      \relax%
    \else%
      \setlength{\unitlength}{\unitlength * \real{\svgscale}}%
    \fi%
  \else%
    \setlength{\unitlength}{\svgwidth}%
  \fi%
  \global\let\svgwidth\undefined%
  \global\let\svgscale\undefined%
  \makeatother%
  \begin{picture}(1,0.67307788)%
    \lineheight{1}%
    \setlength\tabcolsep{0pt}%
    \put(0,0){\includegraphics[width=\unitlength,page=1]{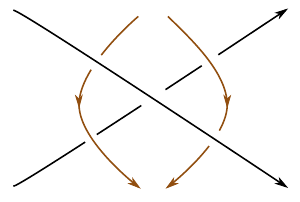}}%
    \put(-0.0032803,0.63811369){\makebox(0,0)[lt]{\lineheight{1.25}\smash{\begin{tabular}[t]{l}$1$\end{tabular}}}}%
    \put(-0.0032803,0.03553828){\makebox(0,0)[lt]{\lineheight{1.25}\smash{\begin{tabular}[t]{l}$2$\end{tabular}}}}%
    \put(0.97038872,0.04534311){\makebox(0,0)[lt]{\lineheight{1.25}\smash{\begin{tabular}[t]{l}$1'$\end{tabular}}}}%
    \put(0.97038872,0.64305345){\makebox(0,0)[lt]{\lineheight{1.25}\smash{\begin{tabular}[t]{l}$2'$\end{tabular}}}}%
    \put(0.07148831,0.32196937){\makebox(0,0)[lt]{\lineheight{1.25}\smash{\begin{tabular}[t]{l}$x_1^- x_2^+$\end{tabular}}}}%
    \put(0.79194551,0.32196937){\makebox(0,0)[lt]{\lineheight{1.25}\smash{\begin{tabular}[t]{l}$x_{2'}^+ x_{1'}^-$\end{tabular}}}}%
  \end{picture}%
\endgroup%

  \caption{Deriving the middle relation at a crossing.}
  \label{fig:tangle-groupoid-relations}
\end{marginfigure}

The standard way to study \(\pi_1(M_L)\) using a diagram of \(L\) is the Wirtinger presentation, which has one generator for each {arc} (arcs don't break at overcrossings, unlike segments) and one relation for each crossing.
We can think of \(\Pi_1(D)\) as using a greater number of more local generators.
This turns out to be more convenient when we discuss face pairings in \cref{sec:geometry}.

\begin{proposition}
  \label{prop:group-groupoid}
  For any diagram \(D\) of a link \(L\) the groupoid \(\Pi_1(D)\) is equivalent to the fundamental group \(\pi_1(M_L)\) of the link complement.
\end{proposition}
To use more abstract language, the claim is that the group \(\pi_1(M_L)\) is a skeleton of the groupoid \(\Pi_1(D)\).
A detailed proof is given in \cite[Section 3]{Blanchet2018}.
It is instructive to consider an example:
In \cref{fig:path-factorization}, we have expressed a path \(w_3 \in \pi_1(M_L)\) representing a generator of the Wirtinger presentation of \(\pi_1(M_L)\) in terms of elements of \(\Pi_1(D)\).
\begin{figure}
  \centering
  %% Creator: Inkscape 1.2.2 (b0a8486541, 2022-12-01), www.inkscape.org
%% Accompanies image file 'path-factorization.pdf' (pdf, eps, ps)
\begingroup%
  \makeatletter%
  \providecommand\color[2][]{%
    \errmessage{(Inkscape) Color is used for the text in Inkscape, but the package 'color.sty' is not loaded}%
    \renewcommand\color[2][]{}%
  }%
  \providecommand\transparent[1]{%
    \errmessage{(Inkscape) Transparency is used (non-zero) for the text in Inkscape, but the package 'transparent.sty' is not loaded}%
    \renewcommand\transparent[1]{}%
  }%
  \providecommand\rotatebox[2]{#2}%
  \newcommand*\fsize{\dimexpr\f@size pt\relax}%
  \newcommand*\lineheight[1]{\fontsize{\fsize}{#1\fsize}\selectfont}%
  \ifx\svgwidth\undefined%
    \setlength{\unitlength}{388.34960175bp}%
    \ifx\svgscale\undefined%
      \relax%
    \else%
      \setlength{\unitlength}{\unitlength * \real{\svgscale}}%
    \fi%
  \else%
    \setlength{\unitlength}{\svgwidth}%
  \fi%
  \global\let\svgwidth\undefined%
  \global\let\svgscale\undefined%
  \makeatother%
  \begin{picture}(1,0.42270363)%
    \lineheight{1}%
    \setlength\tabcolsep{0pt}%
    \put(0,0){\includegraphics[width=\unitlength,page=1]{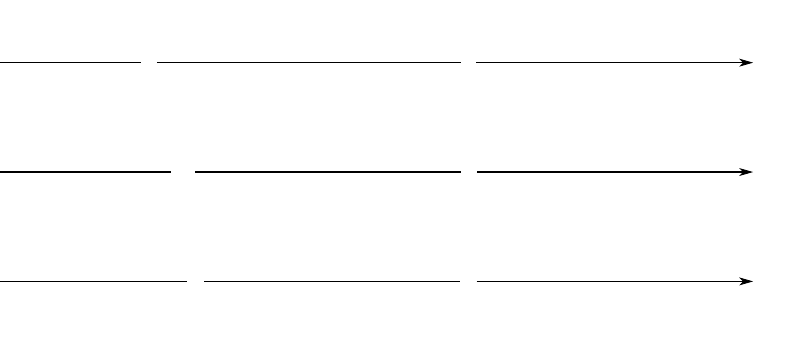}}%
    \put(0.9435563,0.34781885){\makebox(0,0)[lt]{\lineheight{1.25}\smash{\begin{tabular}[t]{l}$1$\end{tabular}}}}%
    \put(0.94355329,0.21006863){\makebox(0,0)[lt]{\lineheight{1.25}\smash{\begin{tabular}[t]{l}$2$\end{tabular}}}}%
    \put(0.9435563,0.07488117){\makebox(0,0)[lt]{\lineheight{1.25}\smash{\begin{tabular}[t]{l}$3$\end{tabular}}}}%
    \put(0,0){\includegraphics[width=\unitlength,page=2]{path-factorization.pdf}}%
    \put(0.16415622,0.00728742){\makebox(0,0)[lt]{\lineheight{1.25}\smash{\begin{tabular}[t]{l}$w_3$\end{tabular}}}}%
    \put(0.58903112,0.30663111){\makebox(0,0)[lt]{\lineheight{1.25}\smash{\begin{tabular}[t]{l}$x_1^+$\end{tabular}}}}%
    \put(0.59868735,0.17144363){\makebox(0,0)[lt]{\lineheight{1.25}\smash{\begin{tabular}[t]{l}$x_2^+$\end{tabular}}}}%
    \put(0.59868735,0.03625617){\makebox(0,0)[lt]{\lineheight{1.25}\smash{\begin{tabular}[t]{l}$x_3^+$\end{tabular}}}}%
    \put(0.42210433,0.11471753){\makebox(0,0)[lt]{\lineheight{1.25}\smash{\begin{tabular}[t]{l}$\left(x_3^-\right)^{-1}$\end{tabular}}}}%
  \end{picture}%
\endgroup%

  \caption{The path $w_3$ in $\pi_1(M_L)$ and the path $x_1^+ x_2^+ x_3^+ \left(x_3^- \right)^{-1} \left(x_2^+\right)^{-1} \left(x_1^+\right)^{-1}$ in $\Pi_1(D)$ are equivalent.}
  \label{fig:path-factorization}
\end{figure}

\begin{definition}
  \label{def:holonomy-representation-diagram}
  Let \(D\) be a \(\chi\)-colored link diagram.
  The \defemph{holonomy representation} of \(D\) is the representation
  \[
    \rho : \Pi_1(D) \to \glg
  \]
  given by
  \begin{align}
    \label{eq:holonomy-representation-diagram}
    \rho\left(x^+\right)
    &=
    \begin{bmatrix}
      a & 0 \\
      (a - 1/m)/b & 1
    \end{bmatrix},
    &
    \rho\left(x^-\right)
    &=
    \begin{bmatrix}
      1 & (a-m)b \\
      0 & a
    \end{bmatrix},
  \end{align}
  where the generators \(x^{\pm}\) are associated to a strand of \(D\) with color \(\chi = (a, b, m)\).
\end{definition}

\begin{remark}
  A link group \(\pi_1(M_L)\) has certain distinguished elements called meridians which correspond to paths around a single strand of \(L\).
  The generators of the Wirtinger presentation are meridians, and all meridians of the same component of \(L\) are conjugate.
  In terms of \(\Pi_1(D)\), the meridian around a strand with color \(\chi\) is conjugate to the matrix
  \begin{equation}
    \label{eq:meridian-factorization}
    \rho\left(x^+(x^-)^{-1}\right)
    =
    \begin{bmatrix}
      a & -(a- m)b \\
      (a-1/m)/b & m + m^{-1} - a
    \end{bmatrix}
  \end{equation}
  which has trace \(m + m^{-1}\).
  In general, the meridian is \emph{not} equal to \(\rho\left(x^{+} (x^{-})^{-1}\right)\), as shown in \cref{fig:path-factorization}.
\end{remark}

\begin{theorem}
  \label{lemma:holonomy-is-defined}
  The holonomy representation of \(\Pi_1(D)\) is well-defined and gives a representation \(\pi_1(M_L) \to \slg\) which we also denote \(\rho\).
\end{theorem}
\begin{proof}
  To make sure \(\rho\) is well-defined, we need to check that the braiding rules on the \(\chi_i\) are compatible with the relations \eqref{eq:groupoid-relations}, which is straightforward.
  Then both claims follow from \cref{prop:group-groupoid}.
\end{proof}

Conversely we say a representation \(\rho\) is \defemph{detected} by \(D\) if there is a \(\chi\)-coloring of \(D\) with holonomy representation \(\rho\).
It is natural to ask if every \(\rho\) is detected by \(D\).
The issue is that the set \(U\) of matrices of the form \eqref{eq:meridian-factorization} is a proper subset of \(\slg\) so one needs to ensure the appropriate conjugate of the image \(\rho(w_{i})\) of each meridian lies in \(U\).
It is easy to find examples where this fails, but they can always be avoided by a global conjugation of \(\rho\).

\begin{theorem}
  \label{thm:existence-blanchet}
  Every \(\slg\)-structure on \(L\) is conjugate to one detected by \(D\).
\end{theorem}

This follows from \cite[Theorem 2]{McphailSnyder2024octahedralcoordinateswirtingerpresentation}, which also gives an explicit condition for whether \(\rho\) is detected and a method to directly determine the \(\chi\)-coloring (in the language of \cite{McphailSnyder2024octahedralcoordinateswirtingerpresentation}, an \defemph{octahedral coloring}) from \(\rho\) and some auxiliary data called a shadow coloring.

\textcite{Blanchet2018} proved a similar theorem \cite[Theorem 5.5]{Blanchet2018} for a closely related representation of \(\Pi_{1}(D)\) given by 
\begin{align}
  \label{eq:holonomy-representation-diagram-BGPR}
  \rho\left(x^+\right)
  &=
  \begin{bmatrix}
    \kappa & 0 \\
    \phi & 1
  \end{bmatrix},
  &
  \rho\left(x^-\right)
  &=
  \begin{bmatrix}
    1 & \epsilon \\
    0 & \kappa
  \end{bmatrix}
\end{align}
for \(\kappa, \epsilon, \phi \in \mathbb{C}\) and \(\kappa \ne 0\).
These include the matrices of \eqref{eq:holonomy-representation-diagram} but are slightly more general; despite this difference their proof still works for \(\chi\)-colorings.
The key fact is that
\[
  U'
  =
  \set{
    \begin{bmatrix}
      \kappa & - \epsilon \\
      \phi & (1 - \epsilon \phi) / \kappa
    \end{bmatrix}
    \given
    \kappa \ne 0
  }
  \text{ and }
  U
  =
  \set{
    \begin{bmatrix}
      a & -(a- m)b \\
      (a-1/m)/b & m + m^{-1} - a
    \end{bmatrix}
    \given
    a, b, m \ne 0
  }
\]
are both Zariski dense subsets of \(\slg\).
The proof of \cite[Theorem 2]{McphailSnyder2024octahedralcoordinateswirtingerpresentation} uses a similar density argument.

\section{A connection to quantum groups}
\label{sec:quantum-groups}
Here we explain how the \(\chi\)-coordinates arise naturally in the representation theory of  \(\qsl\) at \(q = \xi = e^{\pi i / N}\) a root of unity.

\subsection{Quantum groups at roots of unity}
For any simple Lie algebra \(\mathfrak{g}\) the \defemph{quantum group} \(\mathcal{U}_{q}(\mathfrak{g})\) is a \(q\)-analogue of the universal enveloping algebra of \(\mathfrak{g}\).
Quantum groups have a number of interesting algebraic properties including the existence of an element \(\mathbf{R}\) of an appropriately completed tensor product \(\mathcal{U}_{q}(\mathfrak{g}) \mathbin{\widehat{\otimes}} \mathcal{U}_{q}(\mathfrak{g})\) called the \defemph{universal \(R\)-matrix}.
It satisfies nontrivial braid relations that are the key ingredient for constructing quantum invariants of knots and links \cite{Reshetikhin1990,Ohtsuki2001}.
This construction is universal in the sense that any choice of \(\mathcal{U}_{q}(\mathfrak{g})\)-module \(V\) determines a link invariant.
For example, if we choose \(\mathfrak{g} = \mathfrak{sl}_2\) and \(V\) the irreducible \(N\)-dimensional representation of \(\qsl[q]\) we get the \(N\)th colored Jones polynomial.

When \(q\) is not a root of unity, the representation theory of \(\mathcal{U}_q(\mathfrak{g})\) is quite similar to the classical representation theory of \(\mathfrak{g}\).
However, at a root of unity things become much more complicated and depend on exactly which form of the quantum group we use.
For the \defemph{Kac-de Concini form} \cite{de1991representations} of the quantum group, this is because the center gets much larger.

We describe this concretely for \(\qsl[q]\), which is the algebra over $\mathbb{C}[q, q^{-1}]$ with generators $K^{\pm 1}, E, F$ and relations
\[
  KK^{-1} = 1, \quad KE = q^2 EK, \quad KF = q^{-2} FK, \quad EF - FE = (q - q^{-1})(K - K^{-1}).
\]
It is a Hopf algebra with coproduct
\[
  \Delta(K) = K \otimes K, \quad \Delta(E) = E \otimes K + 1 \otimes E, \quad \Delta(F) = F \otimes 1 + K^{-1} \otimes F,
\]
counit
\[
\epsilon(K) = 1, \quad \epsilon(E) = \epsilon(F) = 0, \]
and antipode
\[
  S(E) = - EK^{-1}, \quad S(F) = - KF, \quad S(K) = K^{-1}.
\]
For \(q\) not a root of unity, the center \(Z(\qsl[q])\) is generated by the \defemph{Casimir}
\[
  \Omega = EF + q^{-1}K + q K^{-1} = FE + qK + q^{-1}K^{-1}.
\]

Let \(N \ge 2\) be an integer and set \(q = \xi = e^{\pi i / N}\).
There is now a central Hopf subalgebra \(\mathcal{Z}_{0}\) generated by \(K^{\pm N}\), \(E^N\), and \(F^N\), which we can interpret as the algebra of functions on an algebraic group.
Set
\[
  \slg^* = 
  \left\{
    \left(
      \begin{bmatrix}
        \kappa & 0 \\
        \phi & 1
      \end{bmatrix}
      ,
      \begin{bmatrix}
        1 & \epsilon \\
        0 & \kappa
      \end{bmatrix}
    \right)
    \middle |
    \kappa \ne 0
  \right\}
  \subset
  \glg \times \glg.
\]
Characters \(\chi : \mathcal{Z}_0 \to \CC\) correspond to points of \(\slg^*\) via
\[
  \chi \mapsto
  \left(
    \begin{bmatrix}
      \chi(K^N) & 0 \\
      \chi(K^N F^N) & 1
    \end{bmatrix},
    \begin{bmatrix}
      1 & \chi(E^N) \\
      0 & \chi(K^N)
    \end{bmatrix}
  \right) \in \slg^*
\]
The product is given by
\[
  (\chi_1 \cdot \chi_2)(x) \defeq (\chi_1 \otimes \chi_2)(\Delta(x))
\]
where \(\Delta\) is the coproduct of \(\qsl\).

\begin{theorem}
  The correspondence above gives an isomorphism \(\operatorname{Spec}(\mathcal{Z}_0) \to \slg^*\) of algebraic groups.
  The algebra \(\mathcal{Z}_0\) is large in the sense that \(\qsl/ \ker \chi\) has dimension \(N^3\) for any \(\chi\), and the whole center
  \[
    Z(\qsl) = \mathcal{Z}_0[\Omega]/(\text{polynomial relation})
  \]
  is generated by \(\mathcal{Z}_0\) and the Casimir element \(\Omega\), modulo a degree \(N\) polynomial relation given by a Chebyshev polynomial.
\end{theorem}
\begin{proof}
  See \cite[Chapter 0]{McPhailSnyderThesis} or \cite[Section 6]{Blanchet2018}.
  These results are due to work of \citeauthor{DeConcini1992} \cite{de1991representations,DeConcini1992}.
\end{proof}

This means that to construct quantum invariants from \(\qsl\) we must understand the group \(\slg^*\).
To see why, use Schur's Lemma: if \(V\) is any simple \(\qsl\)-module, the action of the central subalgebra \(\mathcal{Z}_0\) factors through some character \(\chi : \mathcal{Z}_0 \to \CC\).
Furthermore if \(V_1, V_2\) are two modules with characters \(\chi_1, \chi_2\), then their tensor product \(V_1 \otimes V_2\) will have central character \(\chi_1 \chi_2\).
One way to say this is that \(\modc{\qsl}\) is a \defemph{\(\slg^*\)-graded category}.\note{Strictly speaking, this is the case for the category of finite-dimensional \(\qsl\)-modules on which \(\mathcal{Z}_0\) acts diagonizably.}

One approach to dealing with this grading is to mostly eliminate it: if we take the quotient of \(\qsl\) by the relations \(K^{2N} = 1, E^N = F^N = 0\) we obtain the \defemph{small quantum group} \(\overline{\mathcal{U}}_\xi\).
This corresponds to considering only representations whose \(\mathcal{Z}_0\)-character is plus or minus the identity element of \(\slg^*\).
The category \(\modc{\overline{\mathcal{U}}_\xi}\) is not semisimple but by killing the so-called negligible morphisms (those with quantum trace \(0\)) we can obtain a \defemph{modular} category which has the necessary algebraic properties to construct a surgery TQFT \cite{Reshetikhin1991,Turaev2016}.
The corresponding link invariants are colored Jones polynomials of dimension \(1, 2, 3, \dots, N -1\) evaluated at the \(N\)th root of unity \(\xi^2\).

We want to go in a different direction and take full advantage of the \(\slg^*\)-grading.
This idea leads to \defemph{quantum holonomy invariants} \cite{Kashaev2005,Blanchet2018,McPhailSnyder2020,McPhailSnyderThesis,McPhailSnyderAlgebra,McPhailSnyderVolume}.
In the usual construction, picking a single \(\qsl\)-module \(V\) gives a link invariant; for example, choosing \(V\) to be the simple \(N\)-dimensional representation of \(\qsl\) gives the \(N\)th colored Jones polynomial at the root of unity \(\xi^2\) \cite{Murakami2001}.
For quantum holonomy invariants, we instead pick a \emph{family}%
\note{This is what \textcite{Blanchet2018} call a \defemph{representation} of a biquandle in a pivotal category.}
\(V_{\chi}\) of \(\qsl\)-modules indexed by points of \(\slg^*\).
We think of \(V_\chi\) as a deformation of \(V\) by the character \(\chi \in \slg^*\), and \(V\) is the case where \(\chi\) is the identity element.
The input to our construction is no longer a link diagram, but a link diagram with segments colored by elements of \(\slg^*\).
These must satisfy a braiding relation at the crossings, which is derived from the universal \(R\)-matrix but in a significantly different way than for generic \(q\).

\subsection{The braiding at a root of unity}

It is frequently said that \(\qsl[q]\) is a \defemph{quasitriangular} Hopf algebra, but strictly speaking this is false.
Instead this is true for a version \(\qsl[\hbar]\) defined over formal power series in \(\hbar\), where \(q = e^{\hbar}\).
Saying that the Hopf algebra \(\qsl[\hbar]\) is quasitriangular means in particular that it has a \defemph{universal \(R\)-matrix}
\begin{equation}
  \label{eq:universal-R-matrix}
  \mathbf{R} = q^{H \otimes H /2} \sum_{n=0}^{\infty} \frac{q^{n(n-1)/2}}{\{n\}!} (E \otimes F)^n \in \qsl[\hbar] \mathbin{\widehat{\otimes}} \qsl[\hbar]
\end{equation}
where \(\{n\} \defeq q^n - q^{-n}\), \(\{n\}! \defeq \{n\}\{n-1\} \cdots \{1\}\), and the tensor product is appropriately completed.
The key properties of \(\mathbf{R}\) are that it intertwines the coproduct and opposite coproduct
\[
  \mathbf{R} \Delta = \Delta^{\operatorname{op}} \mathbf{R}
\]
and satisfies the \defemph{Yang-Baxter relation}
\[
  \mathbf{R}_{12}  \mathbf{R}_{13}  \mathbf{R}_{23} = \mathbf{R}_{23}  \mathbf{R}_{13}  \mathbf{R}_{12}
\]
which is a version of the braid relation.
(Here \(\mathbf{R}_{12} = \mathbf{R} \otimes 1\) and so on.)
To make it look more like a braid relation, let \(V\) be a \(\qsl[\hbar]\)-module, write \(R\) for the action of \(\mathbf{R}\) on \(V \otimes V\), and set \(\tau(x \otimes y) = y \otimes x\).
Then \(c = \tau R\) is a map \(V \otimes V \to V \otimes V\) of \(\qsl[\hbar]\)-modules satisfying the braid relation
 \[
   (c \otimes \id)
   (\id \otimes c)
   (c \otimes \id)
   =
   (\id \otimes c)
   (c \otimes \id)
   (\id \otimes c).
\]
The map \(c\) is the braiding used to define quantum link invariants.

Usually we work with \(\qsl[q]\) even though the element \(\mathbf{R}\) involves power series in \(\hbar\).
For any finite-dimensional \(\qsl[q]\)-module \(V\) the action of \(\mathbf{R}\) converges and when suitably normalized can be written in terms of \(q\) only, which gives the \(R\)-matrices defining the colored Jones polynomials.
This works because when \(q\) is not a root of unity the elements \(E\) and \(F\) act nilpotently on any finite-dimensional representation.

Even when \(q = \xi\) is a root of unity one can choose modules for which \(E\) and \(F\) act nilpotently.
This leads to Kashaev's invariant \cite{Murakami2001} and to the ADO invariants \cite{Akutsu1992,Blanchet2016}.
However, because at least one of the off-diagonal entries in the holonomy is \(0\) these modules correspond to \(\slg\)-structures with reducible or abelian image; to capture geometrically interesting \(\slg\)-structures we need to allow \(E\) and \(F\) to act invertibly, that is to consider \defemph{cyclic} \(\qsl\)-modules.
Unfortunately the action of the \(R\)-matrix \eqref{eq:universal-R-matrix} on a tensor product of two cyclic modules diverges.
We can work around this by instead considering its conjugation action.
\textcite{Kashaev2004} showed that this still makes sense when specializing to a root of unity:

\begin{proposition}
  \label{thm:outer R-matrix}
  Consider the automorphism \(\Rmat\) of \(\qsl[\hbar]^{\otimes 2}\) defined by
  \[
    \Rmat(x) \defeq \mathbf{R} x \mathbf{R}^{-1}.
  \]
  Set \(W = 1 - K^{-N} E^N \otimes F^N K^N \in \qsl^{\otimes 2}\).
  Then \(\Rmat\) induces an algebra homomorphism
  \[
    \Rmat : \qsl^{\otimes 2} \to \qsl^{\otimes 2}[W^{-1}]
  \]
  characterized uniquely by
  \begin{align*}
    \Rmat(1 \otimes K)
    &=
    (1 \otimes K) (1 - \xi^{-1} K^{-1} E \otimes F K)
    \\
    \Rmat(E \otimes 1)
    &=
    E \otimes K
    \\
    \Rmat(1 \otimes F)
    &=
    K^{-1} \otimes F
  \end{align*}
  and \(
    \Rmat(\Delta(u))
    =
    \Delta^{\op}(u), \quad u \in \qsl.
  \)
\end{proposition}

For a tensor product \(V_1 \otimes V_2\) of \(\qsl\)-modules the \(R\)-matrix defining the braiding is no longer given by the action of \(\mathbf{R}\).
Instead we say a linear map \(R = R(V_{1}, V_{2}, V_{1'}, V_{2'})\)
\begin{equation}
  \label{eq:holonomy-R-matrix}
  R : V_1 \otimes V_2 \to V_{1'} \otimes V_{2'}
\end{equation}
is a \defemph{holonomy \(R\)-matrix} if it intertwines \(\Rmat\) in the sense that
\[
  R(x \cdot v) = \Rmat(x) \cdot R(v) \text{ for every } v \in V_1 \otimes V_2, x \in \qsl^{\otimes 2}.
\]
Unlike for an ordinary braiding \(V_{1'}\) will not be isomorphic to \(V_{1}\) and similarly for \(V_{2}\) and \(V_{2'}\)
To see this, consider the \(\mathcal{Z}_0\)-characters: if \(V_1, V_2\) have characters \(\chi_1, \chi_2\), then the action of any \(z_1 \otimes z_2 \in \mathcal{Z}_0^{\otimes 2}\) will satisfy
\begin{align*}
  (\chi_{1'} \otimes \chi_{2'})(z_1 \otimes z_2) R(v_1 \otimes v_2)
  &=
  z_1 \otimes z_2 \cdot R(v_1 \otimes v_2)
  \\
  &=
  R( \Rmat^{-1}(z_1 \otimes z_2) v_1 \cdot \otimes v_2)
  \\
  &=
  (\chi_1 \otimes \chi_2)\left(\Rmat^{-1}(z_1 \otimes z_2)\right) R(  v_1 \otimes v_2)
\end{align*}
so the characters \(\chi_{1'}, \chi_{2'}\) of the image have
\[
  \chi_{1'} \otimes \chi_{2'} = (\chi_1 \otimes \chi_2) \Rmat^{-1}.
\]
In fact, using the defining relations of \(\Rmat\) this uniquely defines \(\chi_{1'}\) and \(\chi_{2'}\), and the map
\[
  B(\chi_1,  \chi_2) = (\chi_{2'}, \chi_{1'})
\]
is braiding of \cite[Section 6]{Blanchet2018}.
It is birational because of the localization at \(W\).
As for \(\chi\)-colorings we can consider colorings of diagrams by \(\mathcal{Z}_{0}\)-characters related by \(B\) at the crossings.
These similarly define a holonomy representation via the matrices of \cref{eq:holonomy-representation-diagram-BGPR} and when trying to define quantum invariants using \(\Rmat\) one is lead to the following problem:

\begin{problem}
  \label{problem:BGPR}
  Given a representation \(\rho : \pi_{1}(\comp{L}) \to \slg\) and a diagram \(D\) of \(L\), find a coloring of the segments of \(D\) by \(\mathcal{Z}_{0}\)-characters inducing \(\rho\).
\end{problem}

Below we show that by restricting to \(\mathcal{Z}_{0}\)-characters arising from a Weyl algebra this problem is equivalent to finding \(\chi\)-colorings, hence to computing geometric structures on octahedral decompositions.

\subsection{Weyl algebras and \texorpdfstring{\(\chi\)}{χ}-colors}
\begin{definition}
  The \defemph{extended Weyl algebra} is the algebra \(\weyl[q]\) generated over \(\CC[q,q^{-1}]\) by a central invertible element \(z\) and invertible \(x,y\) subject to the relation
  \[
    xy = q^{2} yx.
  \]
\end{definition}
\begin{proposition}
  \label{prop:weyl-embedding}
  The map \(\phi : \weyl[q] \to \qsl[q]\) given by
  \begin{align*}
    K &\mapsto x 
      &
    E &\mapsto qy (z - x)
      &
    F &\mapsto y^{-1}(1 - z^{-1} x^{-1} )
  \end{align*}
  is an algebra homomorphism.
  It acts on the Casimir by
  \[
    \Omega \mapsto qz + (qz)^{-1}.
  \]
  At a \(2N\)th root of unity $q = \xi$ the center of $\weyl$ is generated by $x^N$, $y^N$, and $z$.
  The automorphism $\phi$ takes the center of $\qsl$ to the center of $\weyl$.
  Explicitly,
  \[
    \begin{aligned}
      \phi(K^N)
      &=
      x^N
      \\
      \phi(E^N)
      &=
      y^N(x^N - z^N)
      \\
      \phi(F^N)
      &=
      y^{-N}(1 - z^{-N} x^{-N}).
    \end{aligned} 
    \qedhere
  \]
\end{proposition}

\begin{remark}
  This map was obtained from one given by \citeauthor{Faddeev2000} \cite{Faddeev2000} in terms of a quantum cluster algebra with generators \(w_1, w_2, w_3, w_4\).
  These generators \(q^2\)-commute according to a certain quiver \cite[Figure 4]{Schrader2019} associated to the triangulation of a punctured disc given in \cref{fig:triangulated-braiding-0}.
  It is known \cite{Faddeev2000,Schrader2019} that this presentation explains the factorization of the \(R\)-matrix of \(\qsl[q]\) into four terms.
\end{remark}

As for \(\qsl\) consider the central subalgebra \(\mathcal{Z}_{0}^{\W}\) generated by \(x^{N}\), \(y^{N}\), and \(z^{N}\).
A character \(\chi : \mathcal{Z}_{0}^{\W} \to \mathbb{C}\) can be identified with the \(\chi\)-color
\[
  (a, b, m) =
  (
    \chi(x^{N}),
    \chi(y^{N}),
    \chi(z^{N})
  )
\]
As before central characters of \(\weyl\) are an \(N\)-fold cover of \(\mathcal{Z}_{0}^{\W}\) determined by a choice of \(N\)th root \(\chi(z)\) of \(m\).
The map \(\phi\) sends central characters of \(\weyl\) to central characters of \(\qsl\) via \(\chi \mapsto \chi \phi\).
Under this identification the color \(\chi = (a, b, m)\) corresponds to the element
\[
  \left(
    \begin{bmatrix}
      a & 0 \\
      (a - 1/m)/b & 1
    \end{bmatrix}
    ,
    \begin{bmatrix}
      1 & (a-m)b \\
      0 & a
    \end{bmatrix}
  \right)
\]
of \(\slg^{*}\) that we previously identified as the holonomy of \(\chi = (a,b,m)\).
We can similarly pull back the map \(\Rmat\) along \(\phi\) to give an automorphism \(\Rmat^{\W}\) of \(\weyl[q]^{\otimes 2}\) characterized by
\begin{align*}
  \Rmat^{\W}(x_1)&=x_1g, \\
  \Rmat^{\W}(x_2)&=g^{-1}x_2, \\
  \Rmat^{\W}(y_1^{-1})&=y_2^{-1}+(y_1^{-1}-z_2^{-1}y_2^{-1})x_2^{-1}, \\
  \Rmat^{\W}(y_2)&=\frac{z_1}{z_2}y_1+(y_2-{z_2}^{-1}y_1)x_1, \\
  \Rmat^{\W}(z_1) & = z_1\\
  \Rmat^{\W}(z_2) & = z_2
  \intertext{where}
  g&=1-x_1^{-1}y_1(z_1-x_1)y_2^{-1}(x_2-z_2^{-1}).
\end{align*}
In \cite{McPhailSnyderAlgebra} the map \(\Rmat^{\W}\) is used to compute \(R\)-matrix coefficients that are difficult to determine directly from \(\Rmat\).
This computation leads to the use of quantum dilogarithms and the connection with the Chern-Simons invariant discussed in \cite{McPhailSnyderVolume}.
Here we describe the connection to \(\chi\)-colorings and octahedral decompositions.

\begin{theorem}
  Identifying the set of central characters with \(\chiset\) as above the braiding map \(B\) of \cref{def:braiding} is characterized by 
  \[
    B(\chi_1, \chi_2) = (\chi_{2'}, \chi_{1'})
    \text{ exactly when }
    (\chi_{1'} \otimes \chi_{2'}) = (\chi_1 \otimes \chi_2) \Rmat^{-1}.
    \qedhere
  \]
\end{theorem}
\begin{proof}
  It is not hard to work out that the action of \(\Rmat^{W}\) on \(\mathcal{Z}_{0}^{W}\) is
  \begin{align*}
    \Rmat^W(x_1^N)&=x_1^NG , \\
    \Rmat^W(x_2^N)&=x_2^NG^{-1}, \\
    \Rmat^W(y_1^{-N})&= y_2^{-N}+\left(y_1^{-N} - \frac{y_2^{-N}}{z_2^{N}}\right)x_2^{-N}, \\
    \Rmat^W(y_2^N)&=\frac{z_1^N}{z_2^{N}}y_1^{N}+ \left(y_2^N - \frac{y_1^N}{z_2^N}\right)x_1^N, \\
    \intertext{where}
    G &= 1 + x_1^{-N} \frac{y_1^N}{y_2^N}(x_1^N - z_1^N)(x_2^N - z_2^{-N})
  \end{align*}
  We can then compute
  \begin{align*}
    b_{1'}^{-1} 
    &= \chi_{1'}(y^{-N})
    \\
    &=
    (\chi_1 \otimes \chi_2)
    \left(
      \frac{z_1^N}{z_2^N}y_2^{-N}+(y_1^{-N} - z_1^N y_2^{-N})x_2^{N}
    \right)
    \\
    &=
    \frac{m_1}{b_2 m_2}+\left(\frac{1}{b_1} - \frac{m_1}{b_2}\right)a_2
  \end{align*}
  which after some algebraic manipulation gives the expression for \(b_{1'}\) in \eqref{eq:b-transf-positive}.
  The other variables follow similarly.
\end{proof}

\section{The octahedral decomposition and hyperbolic geometry}
\label{sec:geometry}

Here we describe the geometric interpretation of \(\chi\)-colorings.
In particular, we recover the parametrization of hyperbolic structures on octahedral decompositions worked out by \textcite{Kim2016,Kim2019}.
We re-derive the relevant results in our conventions to make out paper more self-contained; the alternative of explaining how to translate back and forth would take as long and be less clear.

\subsection{Ideal octahedra and their shapes}
\label{sec:octahedra}
\citeauthor{ThurstonNotes} \cite{ThurstonNotes} introduced a way to combinatorially describe hyperbolic structures on link complements (more generally, on cusped \(3\)-manifolds) by using ideal triangulations.
One reference with more detail is the book written by \citeauthor{Purcell2020} \cite{Purcell2020}, which focuses on link complements .
The idea is to triangulate \(S^3 \setminus L\) with all the \(0\)-vertices ``at infinity'', that is lying on \(L\).

\begin{definition}
  An \defemph{ideal tetrahedron} is a tetrahedron \(\Delta \subseteq \HH^3\) whose \defemph{ideal vertices} lie on the boundary at infinity of \(\HH^3\).
  An \defemph{ideal triangulation} of \(M_L\) is a triangulation of \(M_L\) by ideal tetrahedra such that the ideal vertices all lie on \(L\).
\end{definition}

The hyperbolic structure on an ideal tetrahedron is summarized by a \defemph{shape parameter} \(z \in \CC\setminus \{0,1\}\) whose argument is the dihedral angle at a particular edge of the tetrahedron.
The shape parameters at the other edges are \(1/(1-z)\) and \(1 - 1/z\).
If edges \(e_i\) with shape parameters \(z_i\) are glued together in \(\mathcal{T}\), then the \defemph{gluing equation} for that edge is
\[
  \prod_{i} z_i^{k_{i}} = 1
\]
where \(z^{k_{i}}\) is one of \(z\), \(1/(1-z)\), or \(1-1/z\) depending on the combinatorics of the triangulation.
If the gluing equations for every edge are satisfied then we get a hyperbolic structure on the glued manifold.
There is an additional \defemph{completeness equation}  that ensures the restriction of \(\rho\) to the boundary tori of \(M_L\) has the right eigenvalues, or more geometrically that the induced structure on the noncompact manifold \(S^{3} \setminus L\) is complete.

One method to construct ideal triangulations of link complements systematically from diagrams is the \defemph{octahedral decomposition} \cite{ThurstonDNotes, Kashaev1995}, which decomposes the link complement into ideal octahedra.
We briefly summarize the treatment of \citeauthor{Kim2016} \cite{Kim2016}.
\begin{marginfigure}
  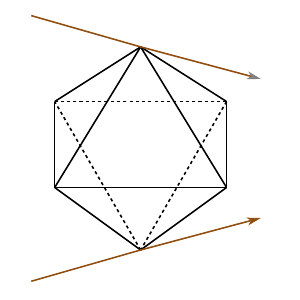
  \caption{An ideal octahedron at a positive crossing, viewed from the side. The ideal vertices \(P_+\) and \(P_+'\) are identified by pulling them above the diagram, as indicated by the grey curves.
  The strands of the link are shown in {\color{accent} gold}.}
  \label{fig:octahedron-positive}
\end{marginfigure}
Fix a diagram \(D\) of \(L\).
We put an ideal octahedron at each crossing of \(D\) with its top and bottom ideal vertices on the strands of the link, labeled as \(P_1\) and \(P_2\) in \cref{fig:octahedron-positive}.
There are four extra ideal vertices \(P_+\), \(P_-\), \(P_+'\), and \(P_-'\), which we pull above and below the diagram, in the process identifying \(P_+\) with \(P_+'\) and \(P_-\) with \(P_-'\).
The resulting simplicial complex is called a \defemph{twisted octahedron}.

The twisted octahedra have two types of edges to glue, which we call \defemph{vertical} and \defemph{horizontal} edges as in \cref{fig:octahedron-positive-edge}.
\begin{marginfigure}
  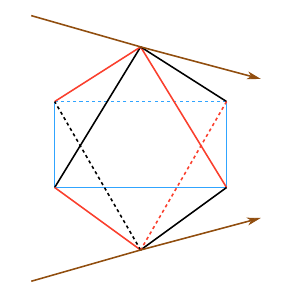
  \caption{
    Red {\color{slred} vertical} and blue {\color{slblue} horizontal} edges of an octahedron at a positive crossing.
    (Recall that the edges \(P_1 P_+\) and \(P_1 P_+'\) are glued in the twisted octahedron, and similarly for \(P_2 P_-\) and \(P_2 P_-'\).)
  }
  \label{fig:octahedron-positive-edge}
\end{marginfigure}
We can determine the gluing patterns of the horizontal edges by looking at the regions of the link diagram \(D\), while the gluing patterns of the vertical edges come from the arcs of \(D\); see \cite[Section 4]{Kim2016} for details.
The result is a decomposition of \(S^3 \setminus (L \cup \{P^+, P_-\})\) into ideal octahedra, where \(P_{\pm}\) are the two extra ideal points above and below the diagram.
These extra ideal points are not a problem in practice as their neighborhoods are balls that can be capped off canonically.
Hyperbolic structures on manifolds with extra ideal points can be understood using the \defemph{pseudo-developing maps} of \cite[Section 2]{Kim2016}.

\begin{figure}
  \centering
  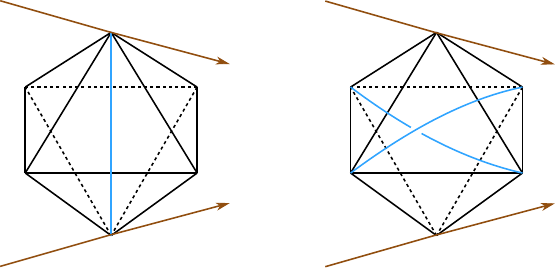
  \caption{The four-term and five-term decompositions of an octahedron.}
  \label{fig:octahedron-four-v-five}
\end{figure}

To assign hyperbolic structures to our octahedral decomposition we subdivide them into tetrahedra; there are two standard ways to do this called the \defemph{four-term} and \defemph{five-term} decompositions shown in \cref{fig:octahedron-four-v-five}.
We discuss the details of these in \cref{sec:four-term,sec:five-term}.
First we show that the shapes they assign to the octahedra glue together to give a well-defined hyperbolic structure.
Consider a twisted octahedron \(O\) at a crossing of \(D\) of sign \(\epsilon \in \set{1, -1}\).
We will choose shapes for the internal tetrahedra (in both decompositions) in such a way that the vertical and horizontal edges of \(O\) have shape parameters
\begin{align}
  \label{eq:edge-shapes-vertical}
  o_1 &= \frac{a_1}{m_1^{\epsilon}}
             &
  o_2 &= \frac{1}{m_2^{\epsilon} a_2}
             &
  o_{1'} &= \frac{m_1^{\epsilon}}{a_{1'}}
              &
  o_{2'} &= {m_2^{\epsilon} a_{2'}}
  \\
  \label{eq:edge-shapes-horizontal}
  o_{N} &= \frac{b_{2'}}{b_1}
             &
  o_{W} &= \frac{m_1 b_1}{b_2}
             &
  o_{S} &= \frac{m_2 b_{2}}{m_1 b_{1'}}
             &
  o_{E} &= \frac{m_2 b_{2'}}{b_{1'}}
\end{align}
Here by \(o_j\) we mean the shape of the vertical edge immediately below or above segment \(j\), and by \(o_{k}\) we mean the shape of the horizontal edge near region \(k\), as in \cref{fig:crossing-shapes-positive}.

\begin{figure}
  \centering
  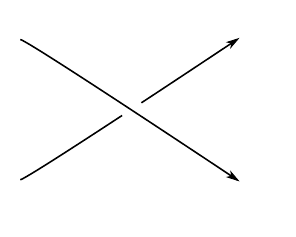
  \caption{
    Shapes of edges at a positive crossing.
    There are four horizontal edges at the four corners and four vertical edges below and above the four segments.
  }
  \label{fig:crossing-shapes-positive}
\end{figure}

\begin{marginfigure}
  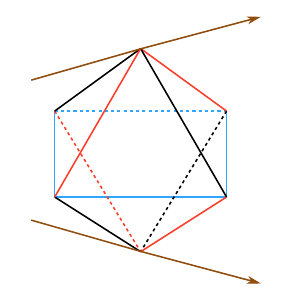
  \caption{
    Red {\color{slred} vertical} and blue {\color{slblue} horizontal} edges of an octahedron at a negative crossing.
    Notice that the vertical edges are indexed slightly differently than in the positive case.
  }
  \label{fig:octahedron-negative-edge}
\end{marginfigure}

\begin{theorem}
  \label{thm:gluing-holds}
  The shape assignments of \cref{eq:edge-shapes-vertical,eq:edge-shapes-horizontal} satisfy the gluing equations of the octahedral decomposition.
\end{theorem}
\begin{proof}
  We refer to \cite[Section 3.2]{Kim2016} for the derivation of these equations.
  There are two types to check: {region} equations and {segment} equations.
  We also need to check the behavior around the cusps (that is, the strands of the link) to make sure it matches the eigenvalues \(m_j\); doing this carefully requires considering the triangulation of the boundary induced by truncating our tetrahedra.

  The region equations \cite[eq.\@ 7]{Kim2016} say that the product of horizontal edge shapes around any region of the diagram must be \(1\).
  It is straightforward to see that this always holds, because the horizontal edge shapes are ratios of parameters assigned to the segments of the diagram, so checking this becomes a combinatorial fact about oriented planar graphs.
  We give two examples in \cref{fig:region-gluing-example}.

  The segment equations are more complicated because the vertical edges of the octahedra are glued together along the over-arcs and under-arcs of the diagram.
  This involves multiple segments at different crossings, so there is something nonlocal to check.
  A key observation of \citeauthor{Kim2016} is that the vertical edge gluing equations follow from a stronger \emph{local} condition at each segment, which they call the \defemph{\(m\)-hyperbolicity equations} \cite[eq.\@ 10]{Kim2016}.
  We give them in our conventions in \cref{def:m-hyperbolicity}.
  It is easy to check that the vertical edges at each segment satisfy the corresponding \(m\)-hyperbolicity equation and this implies that the gluing equations for the vertical edges of all the octahedra hold.

  As suggested by its name the \(m\)-hyperbolicity equation asserts that a segment with color \(\chi = (a, b, m)\) really has \(m\) as an eigenvalue of its holonomy; we discuss this further in \cref{sec:decorations}.
\end{proof}

\begin{figure}
  \centering
  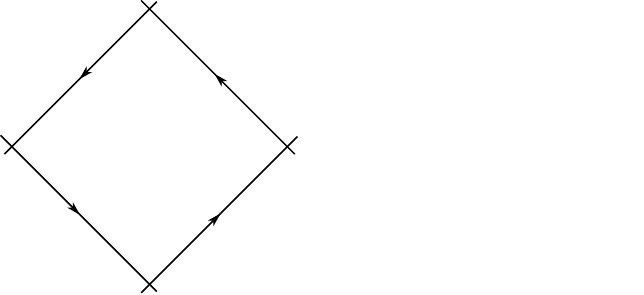
  \caption{Two examples of the region gluing equations. The product of all the parameters is \(1\) regardless of the orientation of the boundary segments.}
  \label{fig:region-gluing-example}
\end{figure}

\subsection{Conventions on ideal tetrahedra}
\label{sec:tetrahedra}
We think of the vertices of an idea tetrahedron \(\tau\) as lying on the Riemann sphere, which is the boundary at infinity of hyperbolic \(3\)-space \(\HH^3\).
By using the upper half-plane model we can identify the boundary of \(\mathbb H^3\) at infinity with the Riemann sphere \(\widehat \CC = \CC \cup \{\infty\}\).
(This is one way to compute \(\operatorname{Isom}(\HH^3) = \pslg\): isometries of \(\HH^{3}\) correspond to isometries of the boundary, which are given by the group \(\pslg\) of fractional linear transformations.)
The geometry of an ideal tetrahedron \(\tau\) is determined by the locations \(p_i \in \widehat \CC,i = 0, 1, 2, 3\) of the points.
The shape parameter of \(\tau\) is their cross-ratio.

\begin{definition}
  For \(p_0, p_1, p_2, p_3 \in \widehat \CC = \CC \cup \{\infty\}\), the \defemph{cross-ratio} is
  \[
    \crossrat{p_0}{p_1}{p_2}{p_3}
    =
    \frac{(p_0 - p_3)(p_1 - p_2)}{(p_0 - p_2)(p_1 - p_3)}.
  \]
  It is well known that \(\crossrat{p_0}{p_1}{p_2}{p_3}\) is invariant under the action of \(\pslg\) by fractional linear transformations.
\end{definition}

We follow \cite[Definition 2.6]{Cho2014} and use a slightly nonstandard convention on edges and shape parameters that is more convenient for our purposes.
Our tetrahedra have signs, and the relationship between the shape parameters and the cross-ratio depends on the sign.

\begin{definition}
  An ideal tetrahedron is \defemph{labeled} if its vertices are totally ordered by labeling them with the set \(\{0, 1, 2, 3\}\) and it is assigned a \defemph{sign} \(\epsilon \in \{1, -1\}\).
  If the vertices of a labeled tetrahedron \(\tau\) are at points \(p_0, p_1, p_2, p_3 \in \widehat \CC\), then we assign the edges \(01\) and \(23\) the shape parameter \(z^0\) given by the cross-ratio
  \[
    z^0 \defeq \crossrat{p_0}{p_1}{p_2}{p_3}^{\epsilon}.
  \]
  We assign the edges \(12\) and \(03\) the shape \(z^{1}\) and the edges \(02\) and \(13\) the shape \(z^{2}\) given by
  \[
    (z^{1})^{\epsilon} = \frac{1}{1- (z^0)^{\epsilon}}
    \text{ and }
    (z^{2})^{\epsilon} = 1- \frac{1}{(z^0)^{\epsilon}}.
  \]
  A tetrahedron is \defemph{degenerate} if one (hence all of) its shape parameters is \(0\), \(1\), or \(\infty\).
\end{definition}
\begin{marginfigure}
  \centering
  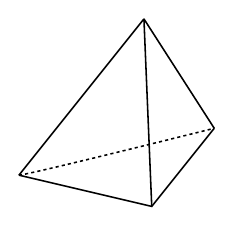
  \caption{Shape parameters assigned to the edges of a labeled tetrahedron.}
  \label{fig:tetrahedron-conventions-small}
\end{marginfigure}

This means that
\begin{align}
  z^{1} &= \frac{1}{1-z^0}
        & 
  z^{2} &= 1 - \frac{1}{z^0}
        &
  \text{for } \epsilon &= 1
  \\
  z^{1} &= 1-\frac{1}{z^0}
        & 
  z^{2} &= \frac{1}{1 - z^0}
        &
  \text{for } \epsilon &= -1
\end{align}
In general, if we index tetrahedra by a symbol \(j\), we write \(z_j^{k}\) for the \(k\)th shape parameter of tetrahedron \(j\) and \(\epsilon_j\) for its sign.
It is frequently useful to use the identity
\[
  z^{k+1}_j =
  \begin{dcases}
    \frac{1}{1-z^k_j} & \epsilon_j = 1
    \\
    1- \frac{1}{z^k_j} & \epsilon_j = -1
  \end{dcases}
\]
The index \(k\) is modulo \(3\), so \(z^3_j = z^0_j\) regardless of \(\epsilon_j\).

\subsection{The four-term decomposition}
\label{sec:four-term}
\begin{figure}
  \centering
  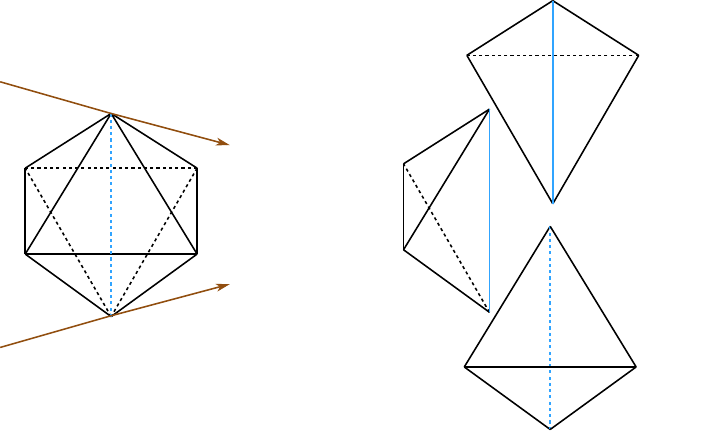
  \caption{The four-term decomposition of an ideal octahedron at a positive crossing.}
  \label{fig:four-term}
\end{figure}
We can now describe the four-term decomposition, which we will use to show that the holonomy representation of a diagram given in \cref{def:holonomy-representation-diagram} agrees with the shapes of the ideal octahedra from \cref{sec:octahedra}.
The idea is to draw a single vertical edge from \(P_1\) to \(P_2\), as in \cref{fig:four-term}.
This divides each octahedron into four tetrahedra, each of which lies between two segments of the diagram, and we label them \(N,S,E,W\) as with the regions near a crossing (see \cref{fig:crossing-regions}).

To describe shapes for the tetrahedra we label them and identify their vertices with points of \(\widehat \CC\).
At a positive crossing, our convention is that the vertices are always ordered \(P_2 P_1 P_- P_+\), that  \(\tau_{N}\) and \(\tau_{S}\) are positive, and that \(\tau_{W}\) and \(\tau_{E}\) are negative.
(Recall that \(P_+, P_+'\) and \(P_-, P_-'\) are identified.)
Geometrically \(P_-\) is located at \(0\), \(P_+\) is located at \(\infty\), and we vary the locations of \(P_1\) and \(P_2\) to give the correct shapes.
\begin{table}
  \[
    \begin{array}{c|c|ccc|c}
      & \text{vertices} & \text{sign } \epsilon & P_1 & P_2  & \text{shape } z^{0} \\
      \hline
      \tau_{N} & P_1P_2P_-'P_+ & 1 & -1/b_1 & -1/b_{2'} & b_{2'}/b_1
      \\
      \tau_{W} & P_1P_2P_-P_+ & -1 &-1/m_1 b_1 & -1/b_2 & m_1 b_1/b_2
      \\
      \tau_{S} & P_1P_2P_-P_+' & 1 &  -1/m_1 b_{1'} & -1/m_2 b_2 & m_2 b_2/ m_1 b_{1'}
      \\
      \tau_{E} & P_1P_2P_-'P_+' & -1 & -1/b_{1'} & -1/m_2 b_{2'} & b_{1'}/m_2 b_{2'}
    \end{array}
  \]
  \caption{Geometric data associated to the four-term decomposition at a positive crossing.}
  \label{table:positive-crossing-data-four}
\end{table}
At a negative crossing we flip the signs of the tetrahedra.
Both sets of conventions are summarized in \cref{table:positive-crossing-data-four,table:negative-crossing-data-four}.

\begin{definition}
  \label{def:pinched-crossing}
  A crossing (labeled as in \cref{fig:crossing-regions}) is \defemph{pinched} if if any of the equations
  \[
    b_2 = m_1 b_1
    ,\quad
    m_2 b_2 = m_1 b_{1'}
    ,\quad
    b_{2'} = b_{1}
    ,\quad
    m_2 b_{2'} = b_{1'}
  \]
  hold, in which case all of them do.
\end{definition}

\begin{table}
  \[
    \begin{array}{c|c|ccc|c}
      & \text{vertices} & \text{sign } \epsilon & P_1 & P_2  & \text{shape } z^{0} \\
      \hline
      \tau_{N} & P_2P_1P_-'P_+ & -1 & -1/b_1 & -1/b_{2'} & b_{2'}/b_1
      \\
      \tau_{W} & P_2P_1P_-P_+ & 1 &-1/m_1 b_1 & -1/b_2 & m_1 b_1/b_2
      \\
      \tau_{S} & P_2P_1P_-P_+' & -1 &  -1/m_1 b_{1'} & -1/m_2 b_2 & m_2 b_2/ m_1 b_{1'}
      \\
      \tau_{E} & P_2P_1P_-'P_+' & 1 & -1/b_{1'} & -1/m_2 b_{2'} & b_{1'}/m_2 b_{2'}
    \end{array}
  \]
  \caption{
    Geometric data associated to the four-term decomposition at a negative crossing.
    The only difference from the positive case is that all the tetrahedra have the opposite sign.
  }
  \label{table:negative-crossing-data-four}
\end{table}

\begin{theorem}
  At any non-pinched crossing the shaped tetrahedra of \cref{table:positive-crossing-data-four,table:negative-crossing-data-four} are geometrically non-degenerate.
  They glue together to give a well-defined hyperbolic structure on an ideal octahedron with edge shapes given by \cref{eq:edge-shapes-vertical,eq:edge-shapes-horizontal}.
\end{theorem}
\begin{proof}
  The non-degeneracy claim is obvious from \cref{def:pinched-crossing}, so consider the claim about gluing.
  The horizontal edges are automatic.
  For example, at a positive crossing the edge \(P_+ P_-'\) is the \(12\) edge of  \(\tau_{S}\), so it is assigned the shape
  \[
    z_{S}^{0} = \frac{m_2 b_2}{m_1 b_{1'}} = o_{S}
  \]
  as it should be.

  The significant part is checking the vertical edges.
  Again, we compute a representative case.
  Consider the edge \(P_+ P_2\) at a positive crossing.
  Its shape has contributions from \(\tau_{W}\) and \(\tau_{S}\); from \cref{table:positive-crossing-data-four} they are
  \[
    z_W^{1} z_S^{1}
    =
    \left(1 - \frac{b_2}{m_1 b_1} \right)
    \left(1 - \frac{m_2 b_2}{m_1 b_{1'}}\right)^{-1}
    =
    \frac{b_{1'}}{b_1}
    \frac{b_2 - m_1 b_1}{m_2 b_2 - m_1 b_{1'}}
  \]
  and this is equal to \(1/m_2 a_2 = o_2\) by \cref{prop:b-gluing-relations} below.
\end{proof}

\begin{lemma}
  \label{prop:b-gluing-relations}
  At a non-pinched positive crossing,
  \begin{equation}
    \label{eq:positive-a-relations}
    \begin{aligned}
      \frac{a_1}{m_1}
      &=
      \frac{1 - b_2/m_1 b_1 }{1 - b_{2'}/b_1}
      &
      \frac{a_{1'}}{m_1}
      &=
      \frac{1 - m_2 b_2 / m_1 b_{1'}}{1 - m_2 b_{2'} / b_{1'}}
      \\
      m_2 a_2
      &=
      \frac{1 - m_2 b_2 / m_1 b_{1'}}{1 - b_2 / m_1 b_1}
      &
      m_2 a_{2'}
      &=
      \frac{1 - m_2 b_{2'} / b_{1'}}{1 - b_{2'} / b_1}
    \end{aligned}
  \end{equation}
  while at a non-pinched negative crossing
  \begin{equation}
    \label{eq:negative-a-relations}
    \begin{aligned}
      m_1 a_1
      &=
      \frac{1 - m_1 b_1 / b_2 }{1 - b_{1} / b_{2'} }
      &
      m_1 a_{1'}
      &=
      \frac{1 - m_1 b_{1'} / m_2 b_2 }{1 - b_{1'} / m_2 b_{2'} }
      \\
      \frac{a_2}{m_2}
      &=
      \frac{1 - m_1 b_{1'} / m_2 b_2 }{1 - m_1 b_1 / b_2 }
      &
      \frac{a_{2'}}{m_2}
      &=
      \frac{1 - b_{1'} / m_2b_{2'} }{1 - b_1 / b_{2'} }
    \end{aligned}
  \end{equation}
\end{lemma}
\begin{proof}
  Once we know \eqref{eq:positive-a-relations} and \eqref{eq:negative-a-relations} it is easy to check them against (\ref{eq:a-transf-positive}--\ref{eq:m-transf-negative}).
\end{proof}

In \cref{sec:shaped-tangles} we used a choice of \(\chi\)-coloring to define a holonomy representation \(\rho : \Pi_1(D) \to \slg\).
Here we show that this agrees with the geometric holonomy representation induced by the four-term decomposition, justifying our name.

We determined the geometry of the tetrahedra in an ideal triangulation by choosing where on \(\partial \HH^3 = \widehat \CC\) their ideal vertices lie.
When we glue two tetrahedra together along a face they will in general disagree about where the vertices of that face are.
The \defemph{face map} \(g \in \pslg = \operatorname{Isom}(\widehat \CC)\) sends the vertices of one face to the other by a fractional linear transformation.
Together all the face maps give a representation \(\pi_1(S^{3} \setminus L ) \to \pslg\).%
\note{
  More precisely, they give a representation of the fundamental groupoid of \(S^{3} \setminus L\) with one basepoint for every tetrahedron.
}
Below we choose the locations of the points \(P_1\) and \(P_2\) so that the face-pairing maps at a crossing exactly correspond to the matrices in \cref{def:holonomy-representation-diagram}.
More precisely:

\begin{theorem}
  \label{thm:holonomies-agree}
  The holonomy representation of a \(\chi\)-colored diagram agrees with the holonomy representation generated by the face maps of the associated four-term decomposition. 
\end{theorem}
To prove the theorem it is helpful to consider a slightly different description of the ideal octahedron at a positive crossing.
This description is related to ideal triangulations of punctured discs.\note{This has something to do with cluster algebras, as discussed in \cref{sec:quantum-groups}.}
Every link $L$ can be represented as the closure of a braid $\beta$.
If we view $\beta$ as an element of the mapping class group of the $n$-punctured disc $D_n$, then the complement $M_L$ of $L$ is the mapping torus\note{If $f : \Sigma \to \Sigma$ is a homeomorphism, them mapping torus of $f$ is the space $\Sigma \times [0,1]$ modulo the relation $(x,0) \sim (f(x),1)$.} of $\beta$.
If we ideally triangulate $D_n$ and interpret the action of $\beta$ in terms of this triangulation, we can get an ideal triangulation of the mapping torus of $\beta$, that is of $M_L$.
\begin{figure}
  \centering
  \subcaptionbox{The initial triangulation.\label{fig:triangulated-braiding-0}}{ \def\svgwidth{2.5in} %% Creator: Inkscape 1.2.2 (b0a8486541, 2022-12-01), www.inkscape.org
%% Accompanies image file 'triangulated-braiding-0.pdf' (pdf, eps, ps)
\begingroup%
  \makeatletter%
  \providecommand\color[2][]{%
    \errmessage{(Inkscape) Color is used for the text in Inkscape, but the package 'color.sty' is not loaded}%
    \renewcommand\color[2][]{}%
  }%
  \providecommand\transparent[1]{%
    \errmessage{(Inkscape) Transparency is used (non-zero) for the text in Inkscape, but the package 'transparent.sty' is not loaded}%
    \renewcommand\transparent[1]{}%
  }%
  \providecommand\rotatebox[2]{#2}%
  \newcommand*\fsize{\dimexpr\f@size pt\relax}%
  \newcommand*\lineheight[1]{\fontsize{\fsize}{#1\fsize}\selectfont}%
  \ifx\svgwidth\undefined%
    \setlength{\unitlength}{215.81824493bp}%
    \ifx\svgscale\undefined%
      \relax%
    \else%
      \setlength{\unitlength}{\unitlength * \real{\svgscale}}%
    \fi%
  \else%
    \setlength{\unitlength}{\svgwidth}%
  \fi%
  \global\let\svgwidth\undefined%
  \global\let\svgscale\undefined%
  \makeatother%
  \begin{picture}(1,1.03093509)%
    \lineheight{1}%
    \setlength\tabcolsep{0pt}%
    \put(0.366521,0.92532356){\color[rgb]{0,0,0}\makebox(0,0)[lt]{\lineheight{1.25}\smash{\begin{tabular}[t]{l}$P_+$\end{tabular}}}}%
    \put(0.38424276,0.03278304){\color[rgb]{0,0,0}\makebox(0,0)[lt]{\lineheight{1.25}\smash{\begin{tabular}[t]{l}$P_-$\end{tabular}}}}%
    \put(0.1486229,0.51022397){\color[rgb]{0,0,0}\makebox(0,0)[lt]{\lineheight{1.25}\smash{\begin{tabular}[t]{l}$P_2$\end{tabular}}}}%
    \put(0.59268972,0.51326773){\color[rgb]{0,0,0}\makebox(0,0)[lt]{\lineheight{1.25}\smash{\begin{tabular}[t]{l}$P_1$\end{tabular}}}}%
    \put(0,0){\includegraphics[width=\unitlength,page=1]{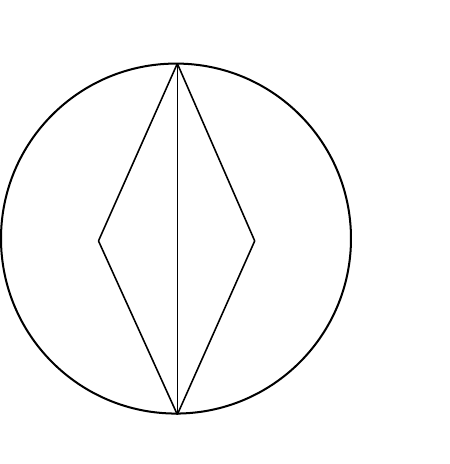}}%
  \end{picture}%
\endgroup%
 }%
  \hfill
  \subcaptionbox{Building a tetrahedron on top of a quadrilateral.\label{fig:triangulated-braiding-1}}{ \def\svgwidth{2.5in} %% Creator: Inkscape 1.2.2 (b0a8486541, 2022-12-01), www.inkscape.org
%% Accompanies image file 'triangulated-braiding-1.pdf' (pdf, eps, ps)
\begingroup%
  \makeatletter%
  \providecommand\color[2][]{%
    \errmessage{(Inkscape) Color is used for the text in Inkscape, but the package 'color.sty' is not loaded}%
    \renewcommand\color[2][]{}%
  }%
  \providecommand\transparent[1]{%
    \errmessage{(Inkscape) Transparency is used (non-zero) for the text in Inkscape, but the package 'transparent.sty' is not loaded}%
    \renewcommand\transparent[1]{}%
  }%
  \providecommand\rotatebox[2]{#2}%
  \newcommand*\fsize{\dimexpr\f@size pt\relax}%
  \newcommand*\lineheight[1]{\fontsize{\fsize}{#1\fsize}\selectfont}%
  \ifx\svgwidth\undefined%
    \setlength{\unitlength}{215.81824493bp}%
    \ifx\svgscale\undefined%
      \relax%
    \else%
      \setlength{\unitlength}{\unitlength * \real{\svgscale}}%
    \fi%
  \else%
    \setlength{\unitlength}{\svgwidth}%
  \fi%
  \global\let\svgwidth\undefined%
  \global\let\svgscale\undefined%
  \makeatother%
  \begin{picture}(1,1.03093509)%
    \lineheight{1}%
    \setlength\tabcolsep{0pt}%
    \put(0.366521,0.92532356){\color[rgb]{0,0,0}\makebox(0,0)[lt]{\lineheight{1.25}\smash{\begin{tabular}[t]{l}$P_+$\end{tabular}}}}%
    \put(0.38424276,0.03278304){\color[rgb]{0,0,0}\makebox(0,0)[lt]{\lineheight{1.25}\smash{\begin{tabular}[t]{l}$P_-$\end{tabular}}}}%
    \put(0.1486229,0.51022397){\color[rgb]{0,0,0}\makebox(0,0)[lt]{\lineheight{1.25}\smash{\begin{tabular}[t]{l}$P_2$\end{tabular}}}}%
    \put(0.59268972,0.51326773){\color[rgb]{0,0,0}\makebox(0,0)[lt]{\lineheight{1.25}\smash{\begin{tabular}[t]{l}$P_1$\end{tabular}}}}%
    \put(0,0){\includegraphics[width=\unitlength,page=1]{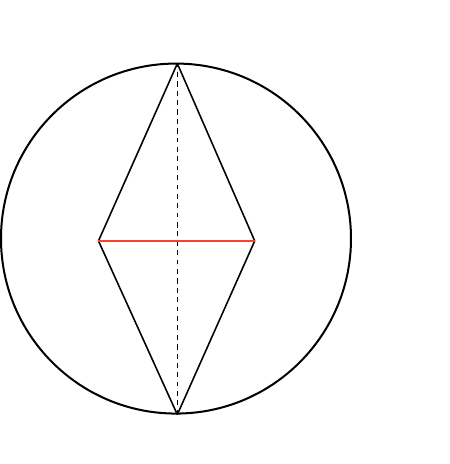}}%
  \end{picture}%
\endgroup%
 }%
  \hfill
  \subcaptionbox{Adding two more tetrahedra.\label{fig:triangulated-braiding-2}}{ \def\svgwidth{2.5in} %% Creator: Inkscape 1.2.2 (b0a8486541, 2022-12-01), www.inkscape.org
%% Accompanies image file 'triangulated-braiding-2.pdf' (pdf, eps, ps)
\begingroup%
  \makeatletter%
  \providecommand\color[2][]{%
    \errmessage{(Inkscape) Color is used for the text in Inkscape, but the package 'color.sty' is not loaded}%
    \renewcommand\color[2][]{}%
  }%
  \providecommand\transparent[1]{%
    \errmessage{(Inkscape) Transparency is used (non-zero) for the text in Inkscape, but the package 'transparent.sty' is not loaded}%
    \renewcommand\transparent[1]{}%
  }%
  \providecommand\rotatebox[2]{#2}%
  \newcommand*\fsize{\dimexpr\f@size pt\relax}%
  \newcommand*\lineheight[1]{\fontsize{\fsize}{#1\fsize}\selectfont}%
  \ifx\svgwidth\undefined%
    \setlength{\unitlength}{215.81824493bp}%
    \ifx\svgscale\undefined%
      \relax%
    \else%
      \setlength{\unitlength}{\unitlength * \real{\svgscale}}%
    \fi%
  \else%
    \setlength{\unitlength}{\svgwidth}%
  \fi%
  \global\let\svgwidth\undefined%
  \global\let\svgscale\undefined%
  \makeatother%
  \begin{picture}(1,1.03093509)%
    \lineheight{1}%
    \setlength\tabcolsep{0pt}%
    \put(0.366521,0.92532356){\color[rgb]{0,0,0}\makebox(0,0)[lt]{\lineheight{1.25}\smash{\begin{tabular}[t]{l}$P_+$\end{tabular}}}}%
    \put(0.38424276,0.03278304){\color[rgb]{0,0,0}\makebox(0,0)[lt]{\lineheight{1.25}\smash{\begin{tabular}[t]{l}$P_-$\end{tabular}}}}%
    \put(0.1486229,0.51022397){\color[rgb]{0,0,0}\makebox(0,0)[lt]{\lineheight{1.25}\smash{\begin{tabular}[t]{l}$P_2$\end{tabular}}}}%
    \put(0.59268972,0.51326773){\color[rgb]{0,0,0}\makebox(0,0)[lt]{\lineheight{1.25}\smash{\begin{tabular}[t]{l}$P_1$\end{tabular}}}}%
    \put(0,0){\includegraphics[width=\unitlength,page=1]{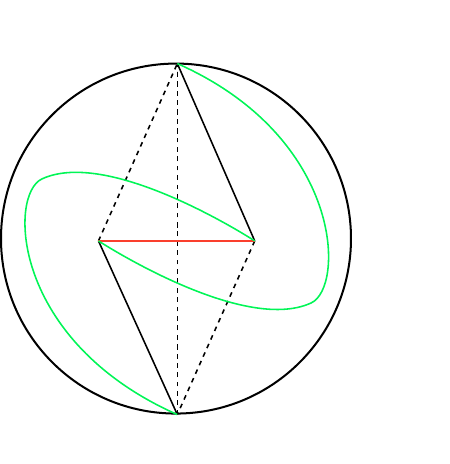}}%
  \end{picture}%
\endgroup%
 }%
  \hfill
  \subcaptionbox{The final result.\label{fig:triangulated-braiding-total}}{ \def\svgwidth{2.5in} %% Creator: Inkscape 1.2.2 (b0a8486541, 2022-12-01), www.inkscape.org
%% Accompanies image file 'triangulated-braiding-total.pdf' (pdf, eps, ps)
\begingroup%
  \makeatletter%
  \providecommand\color[2][]{%
    \errmessage{(Inkscape) Color is used for the text in Inkscape, but the package 'color.sty' is not loaded}%
    \renewcommand\color[2][]{}%
  }%
  \providecommand\transparent[1]{%
    \errmessage{(Inkscape) Transparency is used (non-zero) for the text in Inkscape, but the package 'transparent.sty' is not loaded}%
    \renewcommand\transparent[1]{}%
  }%
  \providecommand\rotatebox[2]{#2}%
  \newcommand*\fsize{\dimexpr\f@size pt\relax}%
  \newcommand*\lineheight[1]{\fontsize{\fsize}{#1\fsize}\selectfont}%
  \ifx\svgwidth\undefined%
    \setlength{\unitlength}{215.81824493bp}%
    \ifx\svgscale\undefined%
      \relax%
    \else%
      \setlength{\unitlength}{\unitlength * \real{\svgscale}}%
    \fi%
  \else%
    \setlength{\unitlength}{\svgwidth}%
  \fi%
  \global\let\svgwidth\undefined%
  \global\let\svgscale\undefined%
  \makeatother%
  \begin{picture}(1,1.03093509)%
    \lineheight{1}%
    \setlength\tabcolsep{0pt}%
    \put(0.366521,0.92532356){\color[rgb]{0,0,0}\makebox(0,0)[lt]{\lineheight{1.25}\smash{\begin{tabular}[t]{l}$P_+$\end{tabular}}}}%
    \put(0.38424276,0.03278304){\color[rgb]{0,0,0}\makebox(0,0)[lt]{\lineheight{1.25}\smash{\begin{tabular}[t]{l}$P_-$\end{tabular}}}}%
    \put(0.1486229,0.51022397){\color[rgb]{0,0,0}\makebox(0,0)[lt]{\lineheight{1.25}\smash{\begin{tabular}[t]{l}$P_2$\end{tabular}}}}%
    \put(0.59268972,0.51326773){\color[rgb]{0,0,0}\makebox(0,0)[lt]{\lineheight{1.25}\smash{\begin{tabular}[t]{l}$P_1$\end{tabular}}}}%
    \put(0,0){\includegraphics[width=\unitlength,page=1]{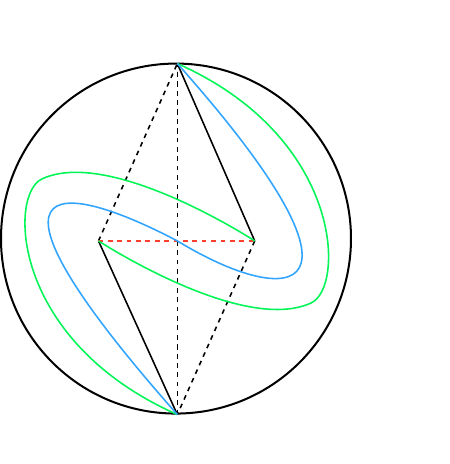}}%
  \end{picture}%
\endgroup%
 }%
  \caption{Building an ideal octahedron.}
  \label{fig:build-ideal-octahedron}
\end{figure}
We describe this process in \cref{fig:build-ideal-octahedron}.
(To visualize it, it may help to examine \cref{fig:octahedron-positive-four-term-colored}.)
\begin{marginfigure}
  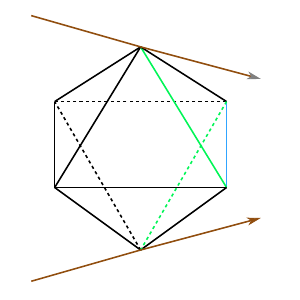
  \caption{The edges of this octahedron are colored to match the edges in \cref{fig:build-ideal-octahedron}.}
  \label{fig:octahedron-positive-four-term-colored}
\end{marginfigure}

We start with the triangulation in \cref{fig:triangulated-braiding-0}.
For simplicity we consider a single crossing at a time, so we only need to consider two punctures $P_1$ and $P_2$ (plus two auxiliary punctures $P_+$ at the top and $P_-$ at the bottom.)
We think of these punctures as corresponding to strands oriented out of the page.%
\note{It's straightforward to extend this picture to any number of interior punctures by gluing copies of \cref{fig:triangulated-braiding-0} along the vertical edges.}

We can modify ideal triangulations by flipping the diagonal of a quadrilateral.
From a $3$-dimensional perspective, we are attaching the final edge of a tetrahedron above its base.
In \cref{fig:triangulated-braiding-1} we add a {\color{slred} red} edge to build an ideal tetrahedron $P_2 P_- P_+ P_1$.
We then add two {\color{slgreen} green} edges, building two more tetrahedra.
Finally, we add the {\color{slblue} blue} edge to finish.

Ignoring the interior dashed edges, which are now below the tetrahedra we have added, we have a new, twisted copy of the triangulation in \cref{fig:triangulated-braiding-0}.
By rotating $P_1$ above $P_2$, we pull the green edges taut and obtain our original picture, but with the points $P_1$ and $P_2$ swapped.
In the process, we have braided the point $P_1$ over the point $P_2$.
This corresponds to a positive braiding in our conventions, assuming that the strands are oriented out of the page in \cref{fig:build-ideal-octahedron}.
At the same time we have built a twisted octahedron at the crossing as required.

\begin{definition}
  \label{def:fractional-linear-action}
  To match our convention that words in \(\Pi_1(D)\) are read left-to-right, elements of \(\pslg\) act on \(\widehat \CC\) on the \emph{right} by
  \[
    z \cdot
    \begin{bmatrix}
      a & b \\
      c & d
    \end{bmatrix}
    =
    \frac{az + c}{bz + d}.
  \]
\end{definition}

\begin{lemma}
  \label{lemma:positive-face-maps-agree}
  The face maps of the octahedron at a positive crossing agree as elements of $\pslg$ with the holonomies assigned to the diagram complement by the shape parameters.
\end{lemma}
\begin{marginfigure}
  {\def\svgwidth{2.5in}
  %% Creator: Inkscape 1.2.2 (b0a8486541, 2022-12-01), www.inkscape.org
%% Accompanies image file 'triangulated-braiding-top-face-map.pdf' (pdf, eps, ps)
\begingroup%
  \makeatletter%
  \providecommand\color[2][]{%
    \errmessage{(Inkscape) Color is used for the text in Inkscape, but the package 'color.sty' is not loaded}%
    \renewcommand\color[2][]{}%
  }%
  \providecommand\transparent[1]{%
    \errmessage{(Inkscape) Transparency is used (non-zero) for the text in Inkscape, but the package 'transparent.sty' is not loaded}%
    \renewcommand\transparent[1]{}%
  }%
  \providecommand\rotatebox[2]{#2}%
  \newcommand*\fsize{\dimexpr\f@size pt\relax}%
  \newcommand*\lineheight[1]{\fontsize{\fsize}{#1\fsize}\selectfont}%
  \ifx\svgwidth\undefined%
    \setlength{\unitlength}{215.81824493bp}%
    \ifx\svgscale\undefined%
      \relax%
    \else%
      \setlength{\unitlength}{\unitlength * \real{\svgscale}}%
    \fi%
  \else%
    \setlength{\unitlength}{\svgwidth}%
  \fi%
  \global\let\svgwidth\undefined%
  \global\let\svgscale\undefined%
  \makeatother%
  \begin{picture}(1,1.03093509)%
    \lineheight{1}%
    \setlength\tabcolsep{0pt}%
    \put(0.366521,0.92532356){\color[rgb]{0,0,0}\makebox(0,0)[lt]{\lineheight{1.25}\smash{\begin{tabular}[t]{l}$P_+$\end{tabular}}}}%
    \put(0.38424276,0.03278304){\color[rgb]{0,0,0}\makebox(0,0)[lt]{\lineheight{1.25}\smash{\begin{tabular}[t]{l}$P_-$\end{tabular}}}}%
    \put(0.1486229,0.51022397){\color[rgb]{0,0,0}\makebox(0,0)[lt]{\lineheight{1.25}\smash{\begin{tabular}[t]{l}$P_2$\end{tabular}}}}%
    \put(0.59268972,0.51326773){\color[rgb]{0,0,0}\makebox(0,0)[lt]{\lineheight{1.25}\smash{\begin{tabular}[t]{l}$P_1$\end{tabular}}}}%
    \put(0,0){\includegraphics[width=\unitlength,page=1]{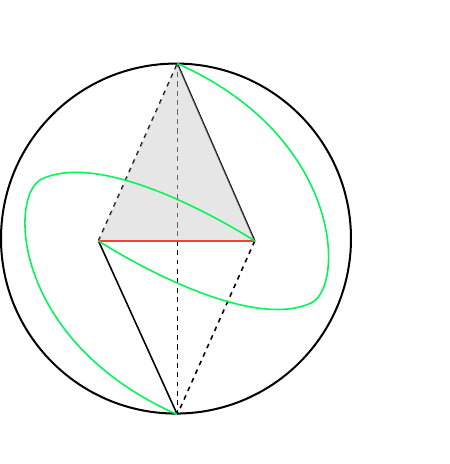}}%
  \end{picture}%
\endgroup%
}
  \caption{The face corresponding to $\chi_2^{+}$.}
  \label{fig:triangulated-braiding-top-face-map}
\end{marginfigure}
\begin{proof}
  If we think of the face map in \cref{fig:triangulated-braiding-top-face-map} as going from $\tau_W$ to $\tau_S$, then it represents the holonomy from travelling above strand $2$, which should be mapped to $\chi_2^+$.
  Observe that for any $z \in \widehat \CC$,
  \[
    z \cdot \chi_2^+ =
    z \cdot
      \begin{bmatrix}
        a_2 & 0 \\
        (a_2 - 1/m_2)/b_2 & 1
      \end{bmatrix}
      =
    a_2 z + \frac{a_2}{b_2} - \frac{1}{m_2 b_2}.
  \]
  In particular, we see that $\chi_2^+$ fixes $\infty$, maps $-1/b_2$ to $-1/m_2 b_2$, and maps $-1/m_1 b_1$ to 
\[
  (-1/m_1 b_1) \cdot \chi_2^+ = a_2 \left( \frac{1}{b_2} -\frac{1}{m_1 b_1} \right)   - \frac{1}{m_2 b_2} = -\frac{1}{m_1 b_1'}.
\]
Because fractional linear transformations are totally determined by their action on three points of $\hat \CC$, we conclude that the face map agrees with $\chi_2^+$.

The negative holonomy of strand $2$ does not correspond directly to a face map, but the face map going from $\tau_W$ to $\tau_N$ similarly corresponds to the \emph{inverse} negative holonomy of $\chi_1$.
We see that the transformation
\[
  z \cdot (\chi_1^-)^{-1}
  = z \cdot
  \begin{bmatrix}
    1 & -(1 + m_1/a_1) b_1 \\
    0 & 1/a_1
  \end{bmatrix}
  \left(
    -b_1 - \frac{m_1 b_1}{a_1} + \frac{1}{za_1}
  \right)^{-1}
\]
preserves $0$, maps $1/m_1 b_1$ to $-1/b_1$, and maps $-1/m_2 b_2$ to 
\[
  (-1/m_2 b_2) \cdot g^-(\chi_1)^{-1}
  =
  \left(
    -b_1 - \frac{m_1 b_1}{a_1} - \frac{b_2}{a_1}
  \right)^{-1}
  =
  - \frac{1}{b_2'}.
\]

There is a parallel characterization of the holonomies on the other side of the crossing.
For example, $\chi_{2'}^+$ corresponds to the gluing map between $\tau_N$ and $\tau_E$, and correspondingly acts on the vertices of $\tau_N$ by
\begin{align*}
  \infty \cdot \chi_{2'}^+
  &=
  \infty
  \\
  (-1/b_{2'}) \cdot \chi_{2'}^+
  &=
  -1/m_2 b_{2'}
  \\
  (-1/b_1) \cdot \chi_{2'}^+
  &=
  -\frac{a_{2'}}{b_1} + \frac{a_{2'}}{b_{2'}} - \frac{1}{m_2 b_{2'}}
  =
  -1/b_{1'}
\end{align*}
and similarly the face map gluing $\tau_S$ to $\tau_E$ is $(\chi_{1'}^-)^{-1}$.
\end{proof}

\begin{proof}[Proof of \cref{thm:holonomies-agree}.]
 At a positive crossing, we have shown the matrices \(\chi_{2}^{+}\), \(\chi_{1}^{-}\), \(\chi_{2'}^+\), and \(\chi_{1'}^{-}\) agree with the corresponding face maps.
 \(\chi_{2}^{-}\) is now the \emph{unique} matrix of the form
 \[
   \begin{bmatrix}
     1 & * \\
     0 & *
   \end{bmatrix}
 \]
 such that
 \[
   \tr \chi_{2}^{+} (\chi_{2}^{-})^{-1} = m_2 + m_2^{-1}
   \text{ and }
   \det \chi_{2}^{+} (\chi_{2}^{-})^{-1} = 1,
 \]
 and similarly for the other strands.
 This shows that we have agreement at any positive crossing.
 By repeating the computation in \cref{lemma:positive-face-maps-agree} for negative crossings we obtain the theorem.
\end{proof}

\subsection{The five-term decomposition}
\label{sec:five-term}
Alternatively we can divide the octahedron at a crossing into five tetrahedra as in \cref{fig:five-simplex-blownup}.
We think of this decomposition as being associated to the \(a\)-variables.
\begin{figure}
  \centering
  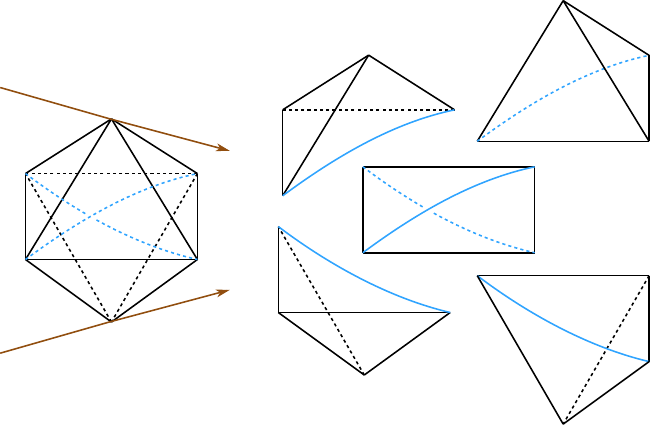
  \caption{Decomposition of the octahedron at a positive crossing into five tetrahedra.}
  \label{fig:five-simplex-blownup}
\end{figure}

\begin{table}
  \centering
  \[
    \begin{array}{c|c|ccc}
      & \text{vertices} & \text{sign } \epsilon & \text{shape } z^{0}
      \\
      \hline
      \tau_1 & P_1 P_-  P_{+}  P_{+} ' & 1 & a_1/m_1
      \\
      \tau_2 & P_2  P_+  P_{-}  P_{-}' & 1 & 1/m_2 a_2
      \\
      \tau_{1'} &  P_1 P_{-}' P_{+}  P_{+}'  & -1 & m_1/a_{1'}
      \\
      \tau_{2'} & P_2  P_{+}' P_-  P_{-}' & -1 & m_2 a_{2'}
      \\
      \tau_m & P_- P_-' P_+  P_+' & 1 & a_{1'}/a_1
    \end{array}
  \]
  \caption{Geometric data for the five-term decomposition at a positive crossing.}
  \label{table:five-simplex-positive}
\end{table}
\begin{table}
  \centering
  \[
    \begin{array}{c|c|ccc}
      & \text{vertices} & \text{sign } \epsilon & \text{shape } z^{0}
      \\
      \hline
      \tau_1 & P_1 P_+  P_{-}  P_{-} ' & -1 & 1/a_1m_1
      \\
      \tau_2 & P_2  P_-  P_{+}  P_{+}' & -1 & a_2/m_2
      \\
      \tau_{1'} &  P_1 P_{+}' P_{-}  P_{-}'  & 1 & m_1 a_{1'}
      \\
      \tau_{2'} & P_2  P_{-}' P_+  P_{+}' & 1 & m_2/a_{2'}
      \\
      \tau_m & P_- P_-' P_+  P_+' & -1 & a_{1}/a_{1'}
    \end{array}
  \]
  \caption{Geometric data for the five-term decomposition at a negative crossing.}
  \label{table:five-simplex-negative}
\end{table}

\begin{definition}
  \label{def:degenerate-crossing}
  A crossing (labeled as in \cref{fig:crossing-regions}) is \defemph{degenerate} if either of the equations
  \[
    a_1 = a_{1'}
    \text{ or }
    a_{2} = a_{2'}
  \]
  hold, in which case both do.
  A degenerate crossing is necessarily pinched but a pinched crossing can be non-degenerate.
\end{definition}

\begin{proposition}
  At any non-degenerate crossing the shaped tetrahedra of \cref{table:five-simplex-positive,table:five-simplex-negative} are non-degenerate and glue together to give an octahedron matching \cref{eq:edge-shapes-vertical,eq:edge-shapes-horizontal}.
\end{proposition}
\begin{proof}
  A crossing is non-degenerate if \(a_1/a_1' = a_2 /a_{2'}\) is not equal to \(1\), which is the same as saying that \(\tau_m\) is geometrically non-degenerate.
  Suppose the crossing is positive.
  Then since
  \[
    \frac{a_1}{a_{1'}} = 
      1 - \frac{m_1 b_1}{b_2} \left(1 - \frac{a_1}{m_1}\right)\left(1 - \frac{1}{m_2 a_2}\right) \ne 1
  \]
  we cannot have \(a_1 = m_1\) or \(a_2 = 1/m_2\), which says that \(\tau_1\) and \(\tau_2\) are geometrically non-degenerate.
  There is a similar expression for \(a_{1}/a_{1'}\) in terms of \(a_{1'}\), \(a_{2'}\), \(b_{1'}\), and \(b_{2'}\) which comes from inverting the map \(B\), and it shows that \(a_1/a_{1'} \ne 1\) implies \(a_{1'} \ne m_1\) and \(a_{2'} \ne 1/m_2\), so \(\tau_{1'}\) and \(\tau_{2'}\) are geometrically non-degenerate.
  If the crossing is negative, a similar argument shows that \(a_1/a_{1'} \ne 1\) implies \(a_1, a_{1'} \ne 1/m_1\) and \(a_{2}, a_{2'} \ne m_2\).

  Next we check the gluing equation.
  The vertical edges are automatic.
  For example, at a positive crossing the total shape of \(P_-P_1\) should be \(o_1 = a_1/m_1\), and the only contributing tetrahedron is \(\tau_1\):
  \[
    z_1^{0} = \frac{a_1}{m_1} = o_1.
  \]

  The horizontal edges require using some identities on \(a_1/a_{1'} = a_{2'}/a_{2}\).
  We compute some representative examples.
  Consider the edge \(P_- P_+\) at a positive crossing, which should have shape \(o_{12} = b_2/m_1 b_1\).
  It has contributions from \(\tau_1\), \(\tau_{2}\), and \(\tau_m\), which give a shape
  \begin{align*}
    z_1^{1} z_2^{1} z_m^{2}
    &=
    \left(1 - \frac{a_1}{m_1} \right)^{-1}
    \left( 1- \frac{1}{m_2 a_1} \right)^{-1}
    \left( 1 - \frac{a_1}{a_{1'}}\right)
    \\
    &=
    \frac{m_1 b_1}{b_2} = o_{W}
  \end{align*}
  using \cref{eq:positive-b-relations} below.
  Similarly, at a positive crossing the edge \(P_-' P_+'\) has contributions from \(\tau_{1'}, \tau_{2'}\), and \(\tau_m\), and again by \cref{eq:positive-a-relations} 
  \[
    z_{1'}^{1} z_{2'}^{1} z_m^{2}
    =
    \left( 1 - \frac{m_1}{a_{1'}} \right)^{-1} \left( 1 - {m_2 a_{2'}} \right)^{-1} \left(1 - \frac{a_1}{a_{1'}}\right)
    =
    o_{E}.\qedhere
  \]
\end{proof}

\begin{lemma}
  \label{lemma:a-gluing-relations}
  At any non-degenerate positive crossing,
  \begin{equation}
    \label{eq:positive-b-relations}
    \begin{aligned}
      \frac{b_2}{b_1 m_1} &= \frac{(1 - a_1/m_1)(1 - 1/m_2 a_2)}{1 - a_{1}/a_{1'}}
      \\
      \frac{m_2 b_{2'}}{b_{1'}} &= \frac{(1 - m_1/a_{1'})(1 - m_2 a_{2'})}{1 - a_{1}/a_{1'}}
      \\
      \frac{b_{1}}{b_{2'}} &= \frac{1 - a_{1'}/a_{1}}{(1- m_1/a_{1})(1- 1/m_2a_{2'})}
      \\
      \frac{m_1 b_{1'}}{m_2 b_{2}} &= \frac{1 - a_{1'}/a_{1}}{(1 - a_{1'}/m_1)(1 - m_2 a_{2})}
    \end{aligned}
  \end{equation}
  while at a non-degenerate negative crossing,
  \begin{equation}
    \label{eq:negative-b-relations}
    \begin{aligned}
      \frac{b_2}{b_1 m_1} &= \frac{1 - a_{1}/a_{1'}}{(1 - a_1 m_1)(1 - m_2/ a_2)}
      \\
      \frac{m_2 b_{2'}}{b_{1'}} &= \frac{1 - a_{1}/a_{1'}}{(1 - 1/m_1 a_{1'})(1 - a_{2'}/m_2)}
      \\
      \frac{b_{1}}{b_{2'}} &= \frac{(1- 1/m_1 a_{1})(1- m_2/ a_{2'})}{1 - a_{1'}/a_{1}}
      \\
      \frac{m_1 b_{1'}}{m_2 b_{2}} &= \frac{(1 - m_1 a_{1'})(1 - a_{2}/m_2)}{1 - a_{1'}/a_{1}}
    \end{aligned}
  \end{equation}
\end{lemma}
\begin{proof}
  As for \cref{prop:b-gluing-relations} once we know the right equations to check this is a straightforward verification.
\end{proof}

The five-term decomposition uses more tetrahedra than the four-term decomposition, but \citeauthor{Cho2016} \cite{Cho2016} and \citeauthor{Yoon2021} \cite{Yoon2021} showed it has some nice nondegeneracy properties.
The analogous theorem is not true for the four-term decomposition: if a crossing is pinched then any conjugate will also be pinched.
We discuss this further in \cite[Section 5]{McphailSnyder2024octahedralcoordinateswirtingerpresentation}.

\begin{theorem}[\protect{\cite[Theorem 1.2]{Yoon2021}}]
  \label{thm:nondegenerate-existence}
  Let \(L\) be a link and \(\rho : \pi_{1}(\comp{L}) \to \slg\) be a representation.
  We say \(\rho\) is \defemph{meridian-trivial} if it sends any meridian of \(L\) to plus or minus the identity matrix.
  If \(\rho\) is meridian-nontrivial, then for any diagram \(D\) of \(L\) there is a \(\chi\)-coloring of \(D\) such that:
  \begin{enumerate}
    \item the holonomy representation of \(D\) is conjugate to \(\rho\), and
    \item all the tetrahedra in the five-term decomposition associated to  \(D\) are geometrically nondegenerate.\qedhere
  \end{enumerate}
\end{theorem}
\begin{proof}
  As discussed in \cref{sec:region-equations} we can solve for non-degenerate \(\chi\)-colorings in terms of only the variables \(a_i\) and \(m_i\).
  Using the region variables instead of the \(a_i\), such solutions are exactly the ``non-degenerate points'' of \cite{Yoon2021}, so the claim is Theorem 1.2 of \cite{Yoon2021}.
\end{proof}

\subsection{Decorations}
\label{sec:decorations}
A \(\chi\)-coloring of a diagram determines both a holonomy representation \(\rho : \pi_1(\comp L) \to \slg\) and an additional choice called a \defemph{decoration}.
These show up naturally in a few contexts: some indeterminacies in the \(A\)-polynomial come from a choice of decoration,%
  \note{%
    The \(A\)-polynomial is usually described as a Laurent polynomial in two variables \(m\) and \(\ell\) that is well-defined up to simultaneous inversion \(m \mapsto m^{-1}, \ell \mapsto \ell^{-1}\).
    These variables represent the meridian and longitude eigenvalues determined by the decorated representation and the inversion comes from changing the choice of decoration.
  }
  and Ptolemy coordinates \cite{Zickert2016} naturally parametrize decorated representations, not representations.
The \(\slg\) Chern-Simons invariant (also known as the complex volume) is most naturally computed using a choice of decoration and similarly the quantized \(\slg\) Chern-Simons invariant of \cite{McPhailSnyderVolume} depends on this choice.

In this section we define decorations and show how a \(\chi\)-coloring of a diagram gives a decoration of its holonomy representation.
We then give a simple, explicit formula for the longitude eigenvalues of the decoration.

\begin{definition}
  A \defemph{decoration} of a representation \(\rho : \pi_{1}(\comp L) \to \slg\) is a choice of invariant line for each meridian of \(L\).
  More formally, each component \(j\) of the oriented tangle has a conjugacy class of meridians \([\mer_j] \subset \pi(T)\).
  For a representative \(\mer_j \in \pi(T)\) we choose a line \(L_j \subset \mathbb{C}^2\) (thought of as a set of \emph{row} vectors) with
  \[
    L_j \rho(\mer_j) = L_j.
  \]
\end{definition}

This definition does not depend on the choice of representative meridian: if \(\mer_j' = y^{-1} \mer_j y\) is any other representative of the conjugacy class we assign it the line \(L_j \rho(y)\), since
\[
  L_j \rho(y) \rho(y^{-1} \mer_j y) = L_j \rho(y).
\]
This choice is called a decoration of component \(j\), and a decoration of \(\rho\) is a decoration of each of the components of \(L\).
A decoration of \(\rho\) induces an equivalent decoration of any conjugate \(g^{-1} \rho g\) in a similar way.

Generically a knot \(K\) has two decorations because a diagonalizable element of \(\slg\) has two eigenspaces.
When \(\tr \rho(x) = \pm 2\) but \(\rho \ne \pm 1\) is nontrivial (i.e. when \(\rho\) is boundary-parabolic) there is only one decoration, and when \(\rho(x) = \pm 1\) is trivial there are infinitely many.
Similar statements apply to links and tangles.

One can also view decorations in terms of peripheral subgroups.
The \defemph{exterior} \(\extr L \defeq S^3 \setminus \nu(L)\) of \(L\) is the complement of an open regular neighborhood of \(L\).
It is a compact manifold with boundary \(\partial \extr L = T_1 \amalg \cdots \amalg T_n\) a disjoint union of tori, one for each component of \(L\).
We call the image \(H_j \subset \pi_1(\comp L)\) of \(\pi_1(T_j)\) the \defemph{peripheral subgroup} associated to the component \(L_j\); as with meridians these are unique up to conjugation.
Each \(H_j\) is isomorphic to \(\ZZ^2\), generated by the meridian \(\mer_j\) and a \defemph{longitude} \(\lon_j\) as in \cref{fig:figure-eight-ML}.
Because \(\ZZ^{2}\) is abelian the choice of decoration gives a basis where the meridian and longitude are lower-triangular:
\begin{equation}
  \label{eq:lower triangular conjugate}
  \rho(\mer_j)
  \sim
  \begin{pmatrix}
    m_j & 0 \\
    * & m_j^{-1}
  \end{pmatrix}
  , \quad
  \rho(\lon_j)
  \sim
  \begin{pmatrix}
    \ell_j & 0 \\
    * & \ell_j^{-1}
  \end{pmatrix}
\end{equation}
A decoration is sometimes defined as an identification of \(\rho(H_{j})\) with the group of upper-triangular matrices \cite[Proposition 4.6]{Garoufalidis2015}; here we use lower-triangular matrices to match our conventions on face maps.

\begin{marginfigure}
  %% Creator: Inkscape 1.2.2 (b0a8486541, 2022-12-01), www.inkscape.org
%% Accompanies image file 'figure-eight-ML.pdf' (pdf, eps, ps)
\begingroup%
  \makeatletter%
  \providecommand\color[2][]{%
    \errmessage{(Inkscape) Color is used for the text in Inkscape, but the package 'color.sty' is not loaded}%
    \renewcommand\color[2][]{}%
  }%
  \providecommand\transparent[1]{%
    \errmessage{(Inkscape) Transparency is used (non-zero) for the text in Inkscape, but the package 'transparent.sty' is not loaded}%
    \renewcommand\transparent[1]{}%
  }%
  \providecommand\rotatebox[2]{#2}%
  \newcommand*\fsize{\dimexpr\f@size pt\relax}%
  \newcommand*\lineheight[1]{\fontsize{\fsize}{#1\fsize}\selectfont}%
  \ifx\svgwidth\undefined%
    \setlength{\unitlength}{139.94114399bp}%
    \ifx\svgscale\undefined%
      \relax%
    \else%
      \setlength{\unitlength}{\unitlength * \real{\svgscale}}%
    \fi%
  \else%
    \setlength{\unitlength}{\svgwidth}%
  \fi%
  \global\let\svgwidth\undefined%
  \global\let\svgscale\undefined%
  \makeatother%
  \begin{picture}(1,0.8768547)%
    \lineheight{1}%
    \setlength\tabcolsep{0pt}%
    \put(0,0){\includegraphics[width=\unitlength,page=1]{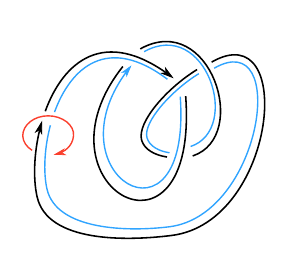}}%
    \put(0.02297167,0.45108763){\color[rgb]{0.98039216,0.24705882,0.17647059}\makebox(0,0)[lt]{\lineheight{1.25}\smash{\begin{tabular}[t]{l}$\mathfrak{m}$\end{tabular}}}}%
    \put(0.31327358,0.13301971){\color[rgb]{0.18823529,0.64313725,1}\makebox(0,0)[lt]{\lineheight{1.25}\smash{\begin{tabular}[t]{l}$\mathfrak{l}$\end{tabular}}}}%
  \end{picture}%
\endgroup%

  \caption{A meridian \(\mer\) (in {\color{slred} red}) and longitude \(\lon\) (in {\color{slblue} blue}) for the figure-eight knot.}
  \label{fig:figure-eight-ML}
\end{marginfigure}

In particular a decoration determines preferred eigenvalues \(m_j\) of each \(\rho(\mer_j)\) and \(\ell_j\) of each \(\rho(\lon_j)\).
We can characterize this in terms of the homomorphism 
\[
  \delta : \homol{\partial \extr L; \ZZ} \to \CC^{\times}
\]
with
\[
  \rho(x) \text{ conjugate to }
  \begin{pmatrix}
    \delta(x) & 0 \\
    * & \delta(x)^{-1}
  \end{pmatrix}
  .
\]
Here we identify the union of the peripheral subgroups \( \bigoplus_{j} \pi_{1}(T_j) \) with the homology \(\homol{\partial \extr L; \ZZ}\) of the boundary of the link exterior.

\begin{theorem}
  \label{thm:decoration-octahedral}
  Let \(D\) be a diagram of a link \(L\).
  For each component \(j\) of \(D\) let \(\mer_{j}\) be the meridian determined by the orientation and \(\lonb_{j}\) the blackboard-framed longitude.
  A \(\chi\)-coloring of \(D\) determines an decoration of its holonomy representation \(\rho\) with distinguished eigenvalues
  \begin{align}
    \label{eq:meridian-eigenvalue}
    \delta(\mer_j) &= m_j
    \\
    \label{eq:longitude-eigenvalue}
    \delta(\lonb_j) &= \prod_{k} b_k^{\eta_k}
  \end{align}
  where the product is over all segments in component \(j\) and
  \[
    \eta_k
    \defeq
    \begin{cases}
      1 & \text{if segment \(k\) is over-under,}
      \\
      -1 & \text{if it is  under-over, and}
      \\
      0 & \text{otherwise.}
    \end{cases}
  \]
  The zero-framed longitude \(\lon_{j}\) has
  \(
    \delta(\lon_j) = m^{-w_j}\delta(\lonb_j)
  \)
  where \(w_{j}\) is the writhe of component \(j\).
\end{theorem}

\begin{proof}
  A segment with color \(\chi = (a, b, m)\) has meridian conjugate to the matrix
  \[
    \begin{bmatrix}
      a & -(a- m)b \\
      (a-1/m)/b & m + m^{-1} - a
    \end{bmatrix}
  \]
  of \cref{eq:meridian-factorization} and its left \(m\)-eigenspace is spanned by \((1/m - a, (a-m)b)\).
  This shows that the \(\chi\)-coloring parameters determine a decoration of their holonomy representation with the claimed meridian eigenvalues.
  We prove the claim about the longitudes in \cref{thm:longitude-eigenvalue-lemma}.
\end{proof}

\begin{example}
  In \cref{fig:trefoil-longitude} the blue curve is \(\lonb\), so it is given by
  \[
    \lonb = \beta_1^- \beta_2^+ \beta_3^- \beta_1^+ \beta_2^- \beta_3^+.
  \]
  Notice that a crossing can appear twice in the product in \cref{eq:blackboard-framed-longitude}, and for knots they always do.
  The zero-framed longitude is \(\lon = \lonb - 3\mer\) because this diagram has writhe \(3\).
\end{example}

\begin{marginfigure}
  \centering
  %% Creator: Inkscape 1.2.2 (b0a8486541, 2022-12-01), www.inkscape.org
%% Accompanies image file 'trefoil-longitude.pdf' (pdf, eps, ps)
\begingroup%
  \makeatletter%
  \providecommand\color[2][]{%
    \errmessage{(Inkscape) Color is used for the text in Inkscape, but the package 'color.sty' is not loaded}%
    \renewcommand\color[2][]{}%
  }%
  \providecommand\transparent[1]{%
    \errmessage{(Inkscape) Transparency is used (non-zero) for the text in Inkscape, but the package 'transparent.sty' is not loaded}%
    \renewcommand\transparent[1]{}%
  }%
  \providecommand\rotatebox[2]{#2}%
  \newcommand*\fsize{\dimexpr\f@size pt\relax}%
  \newcommand*\lineheight[1]{\fontsize{\fsize}{#1\fsize}\selectfont}%
  \ifx\svgwidth\undefined%
    \setlength{\unitlength}{146.04369736bp}%
    \ifx\svgscale\undefined%
      \relax%
    \else%
      \setlength{\unitlength}{\unitlength * \real{\svgscale}}%
    \fi%
  \else%
    \setlength{\unitlength}{\svgwidth}%
  \fi%
  \global\let\svgwidth\undefined%
  \global\let\svgscale\undefined%
  \makeatother%
  \begin{picture}(1,0.59454636)%
    \lineheight{1}%
    \setlength\tabcolsep{0pt}%
    \put(0,0){\includegraphics[width=\unitlength,page=1]{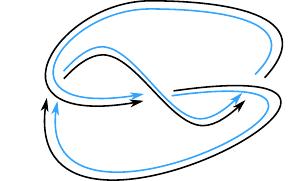}}%
    \put(0.50406201,0.32576861){\makebox(0,0)[lt]{\lineheight{1.25}\smash{\begin{tabular}[t]{l}$2$\end{tabular}}}}%
    \put(0.0424798,0.25839334){\makebox(0,0)[lt]{\lineheight{1.25}\smash{\begin{tabular}[t]{l}$1$\end{tabular}}}}%
    \put(0.80068852,0.2106517){\makebox(0,0)[lt]{\lineheight{1.25}\smash{\begin{tabular}[t]{l}$3$\end{tabular}}}}%
  \end{picture}%
\endgroup%

  \caption{
    Here the blue curve is the blackboard-framed longitude \(\lonb\).
  }
  \label{fig:trefoil-longitude}
\end{marginfigure}

Below we re-derive the meridian eigenvalues geometrically; this is both for completeness and to justify our claim about \(m\)-hyperbolicity equations in the proof of \cref{thm:gluing-holds}.
We then use similar methods to derive the longitude formula \eqref{eq:longitude-eigenvalue}.

Thinking of the boundary torus \(T_j\) of \(\extr L\) as the boundary of a cusp it has an affine structure locally modeled on the Euclidean plane \(\CC\).
The holonomy acts by affine transformations, and the eigenvalues \(m_j\), \(\ell_j\) are related to the scaling factors of these transformations.
This perspective lets us compute the holonomies directly from an ideal triangulation of \(\comp L\).
By truncating our tetrahedra we get a triangulation of the cusps and we can read off the eigenvalues in terms of the shapes.
We refer to \cite[Section 4.3]{Purcell2020} for a general discussion and \cite[Section 4]{Kim2016} for more details in the context of the octahedral decomposition.

\begin{definition}
  \label{def:cusp-holonomy}
  Let \(\gamma\) be an oriented simple curve in the boundary \(T\) of a cusp of \(\extr L\), which we triangulate by truncating an ideal triangulation of \(\comp L\).
  Isotope \(\gamma\) so it intersects only edges of the triangulation transversely and cuts a single corner off of each triangle.
  This corner is associated to the edge of an ideal tetrahedron, and we assign it the shape parameter \(z_k\) of that edge.
  The \defemph{holonomy} of \(\gamma\) is
  \[
    \hol(\gamma) \defeq
    \prod_{k=1}^{n} z_k^{\epsilon_k}
  \]
  where the product is over all the triangles \(\gamma\) passes through,
  \[
    \epsilon_k
    =
    \begin{cases}
      +1 & \text{if the corner is right of \(\gamma\),}
      \\
      -1 & \text{if the corner is left of \(\gamma\),}
    \end{cases}
  \]
  and we view the boundary triangles from outside \(\comp L\).\note{
    This is the opposite of the usual convention, which corresponds to our choice in \cref{def:fractional-linear-action}.
  }
We give an example in \Cref{fig:cusp-triangulation-holonomy-example}.
\end{definition}
\begin{marginfigure}
  \centering
  %% Creator: Inkscape 1.2.2 (b0a8486541, 2022-12-01), www.inkscape.org
%% Accompanies image file 'cusp-triangulation-holonomy-example.pdf' (pdf, eps, ps)
\begingroup%
  \makeatletter%
  \providecommand\color[2][]{%
    \errmessage{(Inkscape) Color is used for the text in Inkscape, but the package 'color.sty' is not loaded}%
    \renewcommand\color[2][]{}%
  }%
  \providecommand\transparent[1]{%
    \errmessage{(Inkscape) Transparency is used (non-zero) for the text in Inkscape, but the package 'transparent.sty' is not loaded}%
    \renewcommand\transparent[1]{}%
  }%
  \providecommand\rotatebox[2]{#2}%
  \newcommand*\fsize{\dimexpr\f@size pt\relax}%
  \newcommand*\lineheight[1]{\fontsize{\fsize}{#1\fsize}\selectfont}%
  \ifx\svgwidth\undefined%
    \setlength{\unitlength}{94.49648952bp}%
    \ifx\svgscale\undefined%
      \relax%
    \else%
      \setlength{\unitlength}{\unitlength * \real{\svgscale}}%
    \fi%
  \else%
    \setlength{\unitlength}{\svgwidth}%
  \fi%
  \global\let\svgwidth\undefined%
  \global\let\svgscale\undefined%
  \makeatother%
  \begin{picture}(1,0.88791976)%
    \lineheight{1}%
    \setlength\tabcolsep{0pt}%
    \put(0,0){\includegraphics[width=\unitlength,page=1]{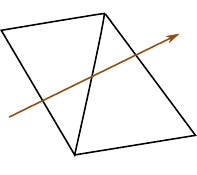}}%
    \put(0.21882778,0.02994896){\makebox(0,0)[lt]{\lineheight{1.25}\smash{\begin{tabular}[t]{l}$z_1$\end{tabular}}}}%
    \put(0.62152997,0.79143808){\makebox(0,0)[lt]{\lineheight{1.25}\smash{\begin{tabular}[t]{l}$z_2$\end{tabular}}}}%
    \put(0.87828525,0.51609974){\color[rgb]{0.56862745,0.30588235,0.05882353}\makebox(0,0)[lt]{\lineheight{1.25}\smash{\begin{tabular}[t]{l}$\gamma$\end{tabular}}}}%
  \end{picture}%
\endgroup%

  \caption{A curve \(\gamma\) in the boundary triangulation, viewed from \emph{outside} the manifold.
  We assign it the holonomy \(\hol(\gamma) = z_1 z_2^{-1}\).}
  \label{fig:cusp-triangulation-holonomy-example}
\end{marginfigure}

\begin{proposition}
  Consider an ideal triangulation of \(\comp L\) with with holonomy \(\rho : \pi_1(\comp L) \to \slg\).
  \(\hol\) defines a homomorphism
  \[
    \hol : \homol{\partial \extr L; \ZZ} \to \CC^{\times}
  \]
  and any decoration of \(\rho\) has
  \[
    \delta(x)^2 = \hol(x) \text{ for all } x \in \homol{\partial \extr L; \ZZ}.\qedhere
  \]
\end{proposition}
\begin{proof}
  The square root comes from the difference between the action of \(\CC^{\times}\) on \(\CC\)  by multiplication and the action of \(\slg\) by fractional linear transformations.
  Geometrically, the holonomy \(\hol(x)\) represents a scaling and rotation of \(\CC\) by multiplication by an element of \(\CC^{\times}\).
  On the other hand we can also compute this action directly from the matrix \(\rho(x) \in \slg\).
  Using the decoration \(\delta\), \(\rho(x)\) is conjugate to
  \[
    \tilde \rho(x)
    =
    \begin{pmatrix}
      \delta(x) & 0 \\
      b & \delta(x)^{-1}
    \end{pmatrix}
  \]
  for some \(b \in \CC\) and (following \cref{def:fractional-linear-action})
  \[
    z \cdot 
    \tilde \rho(x) = \frac{\delta(x) z + b}{\delta(x)^{-1}} = \delta(x)^{2} z + b \delta(x)^{-1}
  \]
  We are interested only in the scaling action \(\delta(x)^{2}\), which is the \emph{square} of the eigenvalue \(\delta(x)\) as claimed.
\end{proof}

For a \(\pslg\) representation the eigenvalues \(\delta(x)\) are only determined up to sign, so knowing \(\hol\) determines \(\delta\).
For \(\slg\) representations one needs to work out how to choose the sign, which is related to obstruction classes (\cref{rem:obstruction-classes}).

\begin{figure}
  \centering
  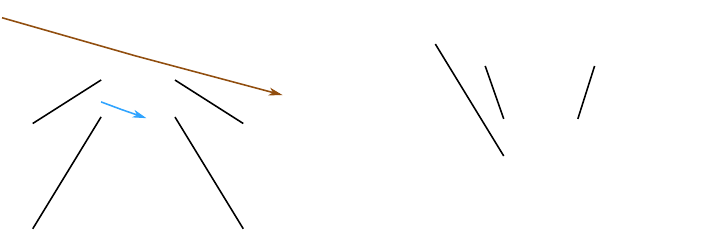
  \caption{Curves in the boundary of the octahedral decomposition near a positive crossing.}
  \label{fig:boundary-holonomy-positive}
\end{figure}

\begin{figure}
  \centering
  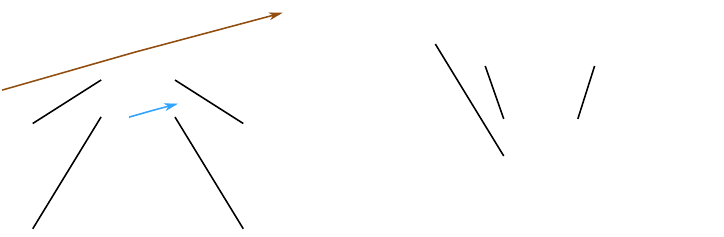
  \caption{Curves in the boundary of the octahedral decomposition near a negative crossing.}
  \label{fig:boundary-holonomy-negative}
\end{figure}

We first compute the holonomy of the meridian \(\mathfrak{m}\) of a segment of a \(\chi\)-colored diagram.
We already know that the answer should be \(m^{2}\) when the segment is assigned the shape  \((a,b,m)\), so our goal is to check this against the geometry.
We can express \(\mathfrak{m}\) as a composition of the curves \(\sigma^{\pm}, \tau^{\pm}\) shown in \cref{fig:boundary-holonomy-positive}.
The exact form depends on the type of segment:
\begin{figure}
  \centering
  %% Creator: Inkscape 1.2.2 (b0a8486541, 2022-12-01), www.inkscape.org
%% Accompanies image file 'segment-types.pdf' (pdf, eps, ps)
\begingroup%
  \makeatletter%
  \providecommand\color[2][]{%
    \errmessage{(Inkscape) Color is used for the text in Inkscape, but the package 'color.sty' is not loaded}%
    \renewcommand\color[2][]{}%
  }%
  \providecommand\transparent[1]{%
    \errmessage{(Inkscape) Transparency is used (non-zero) for the text in Inkscape, but the package 'transparent.sty' is not loaded}%
    \renewcommand\transparent[1]{}%
  }%
  \providecommand\rotatebox[2]{#2}%
  \newcommand*\fsize{\dimexpr\f@size pt\relax}%
  \newcommand*\lineheight[1]{\fontsize{\fsize}{#1\fsize}\selectfont}%
  \ifx\svgwidth\undefined%
    \setlength{\unitlength}{174.62414932bp}%
    \ifx\svgscale\undefined%
      \relax%
    \else%
      \setlength{\unitlength}{\unitlength * \real{\svgscale}}%
    \fi%
  \else%
    \setlength{\unitlength}{\svgwidth}%
  \fi%
  \global\let\svgwidth\undefined%
  \global\let\svgscale\undefined%
  \makeatother%
  \begin{picture}(1,0.308158)%
    \lineheight{1}%
    \setlength\tabcolsep{0pt}%
    \put(0,0){\includegraphics[width=\unitlength,page=1]{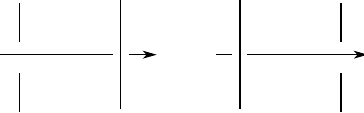}}%
    \put(0.0784108,0.09341107){\makebox(0,0)[lt]{\lineheight{1.25}\smash{\begin{tabular}[t]{l}$k$\end{tabular}}}}%
    \put(0.35758179,0.09341107){\makebox(0,0)[lt]{\lineheight{1.25}\smash{\begin{tabular}[t]{l}$k+1$\end{tabular}}}}%
    \put(0.68399726,0.09340828){\makebox(0,0)[lt]{\lineheight{1.25}\smash{\begin{tabular}[t]{l}$k$\end{tabular}}}}%
    \put(0.96316841,0.09340828){\makebox(0,0)[lt]{\lineheight{1.25}\smash{\begin{tabular}[t]{l}$k+1$\end{tabular}}}}%
  \end{picture}%
\endgroup%

  \caption{An over-under (left) and under-over (right) segment between crossings \(k\) and \(k+1\).}
  \label{fig:segment-types}
\end{figure}
\begin{lemma}
  Consider a segment between crossings labeled \(k\) and \(k+1\).
  Using the segment types in \cref{fig:segment-types},
  \begin{equation*}
    \begin{aligned}
      \tau^+_k  \sigma^{-}_{k+1}
      &=
      \mer
      &
      &\text{at an over-under segment,}
      \\
      \tau^-_k  \sigma^{+}_{k+1}
      &=
      \mer
      &
      &\text{at an under-over segment, and}
      \\
      \tau^\pm_k 
      &=
      \sigma^{\pm}_{k+1}
      &
      &\text{at an over-over or under-under segment.}
    \end{aligned}
  \end{equation*}
  where \(\mer\) is the meridian of the segment.
\end{lemma}
\begin{proof}
  We can see this directly by composing the curves in \cref{fig:boundary-holonomy-positive,fig:boundary-holonomy-negative} during the gluing.
  Alternately, it follows from the discussion in \cite[Section 4.1]{Kim2016} and in particular \cite[eq.\@ 10]{Kim2016}.
  Notice that their meridians are the inverse of ours.
\end{proof}

\begin{lemma}
  At a positive crossing,
  \begin{align*}
    \hol(\sigma^+)
    &=
    o_1^{-1}
    =
    \frac{m_1}{a_1}
    &
    \hol(\tau^+)
    &=
    o_{1'}
    =
    \frac{m_1}{a_{1'}}
    \\
    \hol(\sigma^-)
    &=
    o_2^{-1}
    =
    {m_2 a_{2}}
    &
    \hol(\tau^-)
    &=
    o_{2'}
    =
    {m_2 a_{2'}}
  \end{align*}
  while at a negative crossing
  \begin{align*}
    \hol(\sigma^+)
    &=
    o_2^{-1}
    =
    \frac{m_2}{a_2}
    &
    \hol(\tau^+)
    &=
    o_{2'}
    =
    \frac{m_2}{a_{2'}}
    \\
    \hol(\sigma^-)
    &=
    o_1^{-1}
    =
    {m_1 a_1}
    &
    \hol(\tau^-)
    &=
    o_{1'}
    =
    {m_1 a_{1'}}
  \end{align*}
  where the \(o_j\) are the shapes of the vertical edges given in \cref{eq:edge-shapes-vertical}.
\end{lemma}
\begin{proof}
  To apply \cref{def:cusp-holonomy} we divide the squares of \cref{fig:boundary-holonomy-positive-over-five-term} into triangles by dividing our octahedra into tetrahedra.
  Using the five-term decomposition
  this looks like \cref{fig:boundary-holonomy-positive-over-five-term},
  \begin{marginfigure}
    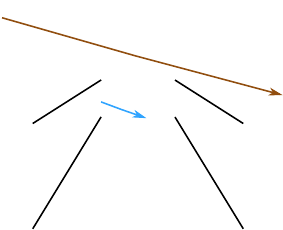
    \caption{Subdividing using the five-term decomposition to get a triangulation of the boundary.}
    \label{fig:boundary-holonomy-positive-over-five-term}
  \end{marginfigure}
  and then at a positive crossing we have
  \begin{equation*}
    \hol(\sigma^+)
    =
    (z_1^0)^{-1}
    \frac{m_1}{a_1}
  \end{equation*}
  and
  \[
    \hol(\sigma^-)
    =
    (z_W^1)^{-1}(z_S^{1})^{-1}
    =
    (z_1^{0})^{-1}
    =
    {m_2 a_2}
  \]
  because we view the boundary from outside the link exterior.
  The other cases follow from similar computations.
\end{proof}

\begin{proof}[Geometric proof of \cref{eq:meridian-eigenvalue}]
  Consider a segment labeled with \(\chi = (a, b, m)\) between crossings \(k\) and  \(k+1\).
  If it is an over-under segment, then
  \[
    \hol(\mer) = \hol(\tau_k^+  \sigma_{k+1}^{-}) = \frac{m}{a} m a = m^{2}
  \]
  as claimed.
  Notice that this computation does not rely on the signs of the crossings.
  Similarly at an under-over segment we have
  \[
    \hol(\mer) = \hol(\tau_k^-  \sigma_{k+1}^{+}) = {m a} \frac{m}{a} = m^2.
  \]
  Taking the square root gives \cref{eq:meridian-eigenvalue}.
  We know that \(m\) (and not \(-m\) ) is the right sign because we can explicitly check that it is an eigenvalue of \eqref{eq:meridian-factorization}.
  We only need to check \cref{eq:meridian-eigenvalue} for one segment of each link component, and at least one segment of any component of any link diagram is either over-under or under-over.
  (Actually, this is only true if the component has at least one crossing.  By adding kinks we can always assume this.)
\end{proof}

\begin{definition}
  \label{def:m-hyperbolicity}
  The \defemph{\(m\)-hyperbolicity equation} for a segment is
  \[
   \frac{o'}{o}
   =
    \begin{cases}
      m^{2} & \text{ if the segment is over-under or under-over, and }
      \\
      1 & \text{ otherwise.}
    \end{cases}
  \]
  where \(o\) is the shape of the vertical edge at the start of the segment and \(o'\) is the shape of the vertical edge at the end.
\end{definition}
We just showed that in any \(\chi\)-coloring the \(m\)-hyperbolicity equations automatically hold.
As discussed in the proof of \cref{thm:gluing-holds} they imply the gluing equations for the vertical edges.

Next we consider the longitudes.
Our convention is to obtain the blackboard-framed longitude \(\lonb\) by pushing off to the right, so it is given by
\begin{equation}
  \label{eq:blackboard-framed-longitude}
  \lonb
  =
  \prod_{k} \beta_k^{\eta_k}
\end{equation}
where \(\eta_k\) is \(+\) at an overcrossing and \(-\) at an undercrossing
and the product is over all intersections of our component with the rest of the diagram.
The zero-framed longitude is
\begin{equation}
  \label{eq:zero-framed-longitude}
  \lon = \lonb - w \mer
\end{equation}
where \(w\) is the writhe of the link component we are considering.

\begin{lemma}
  \label{lemma:longitude-holonomy}
  At a positive crossing
  \[
    \begin{aligned}
      \hol(\beta^+)
      &=
      a_2 \frac{b_{1'}}{b_1}
      \\
      \hol(\beta^-)
      &=
      \frac{1}{a_{1'}}  \frac{b_{2}}{b_{2'}}
    \end{aligned}
  \]
  and at a negative crossing
  \[
    \begin{aligned}
      \hol(\beta^+)
      &=
      \frac{1}{a_{1'}}  \frac{b_{2'}}{b_2} 
      \\
      \hol(\beta^-)
      &=
      {a_{2}}  \frac{b_{1}}{b_{1'}} 
    \end{aligned}
    \qedhere
  \]
\end{lemma}
\begin{proof}
  For \(\beta^+\) at a positive crossing, we can use the five-term decomposition as before to compute
  \begin{equation*}
    \hol(\beta^+)
    =
    z_1^1
    z_{1'}^1
    =
    \left(1 - \frac{a_1}{m_1}\right)^{-1}
    \left(1 - \frac{a_{1'} }{m_1}\right)
    =
    a_{2} \frac{b_{1'}}{b_1}
  \end{equation*}
  using \cref{eq:positive-a-relations}.
  The other computations follow similarly.
\end{proof}

\begin{marginfigure}
  \centering
  %% Creator: Inkscape 1.2.2 (b0a8486541, 2022-12-01), www.inkscape.org
%% Accompanies image file 'longitude-rules-b.pdf' (pdf, eps, ps)
\begingroup%
  \makeatletter%
  \providecommand\color[2][]{%
    \errmessage{(Inkscape) Color is used for the text in Inkscape, but the package 'color.sty' is not loaded}%
    \renewcommand\color[2][]{}%
  }%
  \providecommand\transparent[1]{%
    \errmessage{(Inkscape) Transparency is used (non-zero) for the text in Inkscape, but the package 'transparent.sty' is not loaded}%
    \renewcommand\transparent[1]{}%
  }%
  \providecommand\rotatebox[2]{#2}%
  \newcommand*\fsize{\dimexpr\f@size pt\relax}%
  \newcommand*\lineheight[1]{\fontsize{\fsize}{#1\fsize}\selectfont}%
  \ifx\svgwidth\undefined%
    \setlength{\unitlength}{78.141675bp}%
    \ifx\svgscale\undefined%
      \relax%
    \else%
      \setlength{\unitlength}{\unitlength * \real{\svgscale}}%
    \fi%
  \else%
    \setlength{\unitlength}{\svgwidth}%
  \fi%
  \global\let\svgwidth\undefined%
  \global\let\svgscale\undefined%
  \makeatother%
  \begin{picture}(1,0.64456235)%
    \lineheight{1}%
    \setlength\tabcolsep{0pt}%
    \put(0,0){\includegraphics[width=\unitlength,page=1]{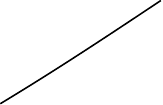}}%
    \put(0.16673431,0.37511578){\color[rgb]{0,0,0}\makebox(0,0)[lt]{\lineheight{1.25}\smash{\begin{tabular}[t]{l}$b$\end{tabular}}}}%
    \put(0.5647476,0.07920663){\color[rgb]{0,0,0}\makebox(0,0)[lt]{\lineheight{1.25}\smash{\begin{tabular}[t]{l}$b'$\end{tabular}}}}%
    \put(0,0){\includegraphics[width=\unitlength,page=2]{longitude-rules-b.pdf}}%
    \put(0.33852614,0.04798977){\color[rgb]{0,0,0}\makebox(0,0)[lt]{\lineheight{1.25}\smash{\begin{tabular}[t]{l}$r'$\end{tabular}}}}%
    \put(0.09672275,0.23354638){\color[rgb]{0,0,0}\makebox(0,0)[lt]{\lineheight{1.25}\smash{\begin{tabular}[t]{l}$r$\end{tabular}}}}%
  \end{picture}%
\endgroup%
 
  \caption{Following the gold strand northwest to southeast this crossing contributes a factor of \((r'/r)(b'/b)^{\eta}\) to the longitude holonomy \(\hol(\lon)\).
  }
  \label{fig:longitude-rules-b} 
\end{marginfigure}

\begin{lemma}
  \label{thm:longitude-eigenvalue-lemma}
  As claimed in \cref{eq:longitude-eigenvalue} the eigenvalue of the blackboard-framed longitude is given by 
  \[
    \delta(\lonb_j) = \prod_{k} b_k^{\eta_k}
    \text{ for }
    \eta_k
    \defeq
    \begin{cases}
      1 & \text{if segment \(k\) is over-under,}
      \\
      -1 & \text{if it is  under-over, and}
      \\
      0 & \text{otherwise.}
    \end{cases}
    \qedhere
  \]
\end{lemma}

\begin{proof}
  Fix a component \(j\) and abbreviate \(\lonb_{j} = \lonb\).
  We first show that \(\hol(\lonb) = \delta(\lonb)^{2}\) by taking the product over the crossings of our diagram.

  The expressions in \cref{lemma:longitude-holonomy} follow a simple pattern in terms of the \defemph{region variables} of \cref{sec:region-equations}.
  The idea is to assign variables \(r_k\) to the regions of the diagram so ratios of adjacent region variables give the \(a\)-variable of the strand between them, as in \cref{fig:region-variable-rule}.
  It is easy to see (\cref{lemma:shapings-give-region-vars}) we can always assign region variables to any \(\chi\)-colored diagram.

  Once we do this we can summarize \cref{lemma:longitude-holonomy} by saying that the contribution of the crossing in \cref{fig:longitude-rules-b} is
  \begin{equation}
    \label{eq:simple-crossing-contribution}
    \frac{r'}{r} \left( \frac{b'}{b} \right)^{\eta}
  \end{equation}
  where \(\eta\) is \(1\) if the gold strand passes over the black strand and \(-1\) if it passes under.
  The variables \(r,r'\) and \(b,b'\) correspond to the regions and segments adjacent to the crossing.
  We can prove \cref{eq:simple-crossing-contribution} by a trivial case-by-case check, as usual.

  Now we need to translate our product over crossings into a product over segments.
  First consider an over-under segment, like in \cref{fig:cancelling-a-variables}.
  \begin{marginfigure}
    %% Creator: Inkscape 1.2.2 (b0a8486541, 2022-12-01), www.inkscape.org
%% Accompanies image file 'cancelling-a-variables.pdf' (pdf, eps, ps)
\begingroup%
  \makeatletter%
  \providecommand\color[2][]{%
    \errmessage{(Inkscape) Color is used for the text in Inkscape, but the package 'color.sty' is not loaded}%
    \renewcommand\color[2][]{}%
  }%
  \providecommand\transparent[1]{%
    \errmessage{(Inkscape) Transparency is used (non-zero) for the text in Inkscape, but the package 'transparent.sty' is not loaded}%
    \renewcommand\transparent[1]{}%
  }%
  \providecommand\rotatebox[2]{#2}%
  \newcommand*\fsize{\dimexpr\f@size pt\relax}%
  \newcommand*\lineheight[1]{\fontsize{\fsize}{#1\fsize}\selectfont}%
  \ifx\svgwidth\undefined%
    \setlength{\unitlength}{96.05856514bp}%
    \ifx\svgscale\undefined%
      \relax%
    \else%
      \setlength{\unitlength}{\unitlength * \real{\svgscale}}%
    \fi%
  \else%
    \setlength{\unitlength}{\svgwidth}%
  \fi%
  \global\let\svgwidth\undefined%
  \global\let\svgscale\undefined%
  \makeatother%
  \begin{picture}(1,0.68201797)%
    \lineheight{1}%
    \setlength\tabcolsep{0pt}%
    \put(0,0){\includegraphics[width=\unitlength,page=1]{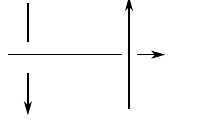}}%
    \put(0.10674642,0.01836035){\makebox(0,0)[lt]{\lineheight{1.25}\smash{\begin{tabular}[t]{l}$a$\end{tabular}}}}%
    \put(0.59082612,0.01836035){\makebox(0,0)[lt]{\lineheight{1.25}\smash{\begin{tabular}[t]{l}$a'$\end{tabular}}}}%
    \put(-0.04160055,0.29943892){\makebox(0,0)[lt]{\lineheight{1.25}\smash{\begin{tabular}[t]{l}$b_0$\end{tabular}}}}%
    \put(0.33317084,0.29943892){\makebox(0,0)[lt]{\lineheight{1.25}\smash{\begin{tabular}[t]{l}$b_1$\end{tabular}}}}%
    \put(0.76259628,0.29943892){\makebox(0,0)[lt]{\lineheight{1.25}\smash{\begin{tabular}[t]{l}$b_2$\end{tabular}}}}%
    \put(-0.00866841,0.12597966){\makebox(0,0)[lt]{\lineheight{1.25}\smash{\begin{tabular}[t]{l}$r_0$\end{tabular}}}}%
    \put(0.33487211,0.12597966){\makebox(0,0)[lt]{\lineheight{1.25}\smash{\begin{tabular}[t]{l}$r_1$\end{tabular}}}}%
    \put(0.73306657,0.12597966){\makebox(0,0)[lt]{\lineheight{1.25}\smash{\begin{tabular}[t]{l}$r_2$\end{tabular}}}}%
  \end{picture}%
\endgroup%

    \caption{Region and segment variables contributing to the longitude.}
    \label{fig:cancelling-a-variables}
  \end{marginfigure}
  The crossings at each end contribute a factor
  \[
    a \left( \frac{b_1}{b_0} \right)^{+1} \cdot \frac{1}{a'} \left( \frac{b_2}{b_1} \right)^{-1}
    =
    \frac{r_1}{r_0} \frac{b_1}{b_0} \frac{r_2}{r_1} \frac{b_1}{b_2}
    =
    \frac{r_2}{r_0} \frac{1}{b_0 b_2} b_1^{2}
  \]
  to \(\hol(\lonb)\).
  In particular, we see that \(r_1\) does not contribute, and the exponent of \(b_1\) is \(+2\).
  If instead we had an over-over segment, the contribution would be
  \[
    \frac{r_1}{r_0} \frac{b_1}{b_0} \frac{r_2}{r_1} \frac{b_2}{b_1}
    =
    \frac{r_2}{r_0} \frac{b_2}{b_0}
  \]
  and \(b_1\) does not appear at all.

  More generally, when following a component of a link diagram, the region variables appear as a telescoping product
  \[
    \frac{r_0}{r_1} \frac{r_1}{r_2} \cdots \frac{r_n}{r_0} = 1
  \]
  and the \(b\)-variables only show up with even exponents: \(+2\) if their segment is over-under, \(-2\) if it is under-over, and \(0\) otherwise.
  We have shown that when we write \(\hol(\lonb)\) as a product over segments,
  \[
    \hol(\lonb) = \prod_{k}  b_k^{2\eta_k}
  \]
  so
  \[
    \hol(\lon) = m^{-2w_j} \prod_{k} b_k^{2\eta_k} = \delta(\lon)^{2}.
  \]

It remains only to show that we have taken the correct sign of \(\sqrt {\hol(\lon)}\).
(We have already showed that we picked the right sign of \(\sqrt{\hol(\mer)}\) because \(m\) is an eigenvalue of \eqref{eq:meridian-factorization}, not \(-m\).)
One way to do this is to define a matrix-valued version of \(\hol\).
This is done in detail in the boundary-parabolic case in \cite[Section 4]{Kim2019}; the general argument follows by extending their work to \defemph{deformed Ptolemy varieties}, as in \cite[Section 2]{Yoon2018}.
There is also an elementary argument using the methods of \cite{McphailSnyder2024octahedralcoordinateswirtingerpresentation} so we do not give the details here.
\end{proof}

\begin{remark}
  \label{rem:obstruction-classes}
  When \(\rho\) is a boundary-parabolic \(\pslg\) representation that lifts to an \(\slg\) representation we can always choose a lift with \(\delta(\mer) = 1\).
  However, in general the sign of \(\delta(\lon)\) can be \(-1\), regardless of the choice of lift.
  In fact, for a hyperbolic knot complement we have \(\delta(\lon) = -1\) for any lift of the geometric representation \cite[Corollary 2.4]{Calegari2005}.

  We can think of this as an obstruction to lifting \(\rho\) to a boundary-unipotent \(\slg\)-representation, and in this context \(\delta(\lon) \in \set{1, -1}\) is called the \defemph{obstruction class} \cite{Garoufalidis2015,Kim2019,Cho2020} of the representation.
  Determining the sign of \(\sqrt{\hol(\lon)}\) is closely related to these obstruction classes.
\end{remark}

\section{Gluing equations}
\label{sec:gluing}
The conditions (\ref{eq:a-transf-positive}--\ref{eq:m-transf-positive}) on the \(\chi\)-colors at each crossing are somewhat complicated.
In practice, we can usually make some simplifications: we can either eliminate the \(a\)-variables or the \(b\)-variables, as long as we avoid certain geometrically degenerate solutions.
In this section we make this precise.

\begin{definition}
  \label{def:repvar-shapevar}
  Let \(L\) be a link in \(S^3\) with \(c\) components.
  The \defemph{representation variety} of \(L\) is the set \(\rvar L\) of representations \(\rho : \pi_1(S^3 \setminus L) \to \slg\) of the link complement into \(\slg\).

  Now suppose \(D\) is a diagram of \(L\) with \(s\) segments.
  We associate variables
  \[
    a_1, \dots, a_s, b_1, \dots, b_s \in \CC \setminus \{0\}
  \]
  to the segments of \(D\) and variables
  \[
    m_1, \dots, m_c \in \CC \setminus \{0\}
  \]
  to the components.
  Writing \(C(i)\) for the component of segment \(i\), we assign each segment the color \(\chi_i = (a_i, b_i, m_{C(i)})\).
  The \defemph{\(\chi\)-variety} of \(D\) is the set \(\svar D \subset (\CC \setminus \{0\})^{2s+c}\) of \(\chi\)-colorings satisfying the relations of \cref{def:shaping}.%
  \note{%
    Here \(\mathfrak{P}\) stands for ``Ptolemy'' because the \(\chi\) are Ptolemy coordinates.
    We would use \(\mathfrak{X}\) but this is the usual notation for the character variety, a different object.
  }
\end{definition}

The holonomy map of \cref{lemma:holonomy-is-defined} defines an inclusion \(\svar D \to \rvar L\).
It is not surjective, but it is up to conjugacy: \cref{thm:existence-blanchet}  says that the \(\slg\)-orbit of every \(\rho \in \rvar L\) intersects the image of \(\svar D\).

In this section we will define two sets \(\avar D\) and \(\bvar D\) defined by equations in significantly fewer variables.
\(\bvar D\) comes from eliminating the \(a\)-variables in terms of the \(b\)-variables, but it only detects non-pinched solutions in the sense of \cref{def:pinched-crossing}.
To check membership in \(\bvar D\) we only need to check one equation in the \(b_i\) and \(m_i\) for each segment of \(D\).
In parallel, \(\avar D\) comes from eliminating the \(b\)-variables in terms of the \(a\)-variables, and it only detects non-degenerate solutions.
The equations defining \(\avar D\) are instead associated to the regions of the diagram \(D\): for each region we check that a certain product involving \(a\)-variables and \(m\)-variables is \(1\).

There are examples of geometrically interesting points of \(\svar D\) that do not lie in the image of \(\bvar D\), but \cref{thm:nondegenerate-existence} says that every point of \(\svar D\) with nontrivial (not \(\pm 1\)) holonomy lies in the image of \(\avar D\).
The image of \(\bvar D\) in \(\rvar L\) is interesting in connection with the Volume Conjecture and is further studied in \cite[Section 5]{McphailSnyder2024octahedralcoordinateswirtingerpresentation}.

\subsection{The \texorpdfstring{\(b\)}{b}-variety \texorpdfstring{\(\bvar D\)}{B\_D} and the segment equations}
Let \(D\) be a diagram of \(L\) with \(c\) components and \(s\) segments.
Associate variables \(b_1, \dots, b_s\) to the segments and \(m_1, \dots, m_c\) to the components of \(D\); as before this gives a tuple \((b_i, m_{C(i)})\) for each segment \(i\) of \(D\).
Now consider a crossing of \(D\).
If the crossing is not pinched (which can be checked in terms of the \(b_i\) and \(m_i\) alone), \cref{prop:b-gluing-relations} determines the \(a\)-variables for each segment at the crossing.

\begin{definition}
  Because a segment is adjacent to two crossings, this procedure assigns two different \(a\)-variables to each segment.
  The \defemph{segment equation} of a segment says that the two \(a\)-variables agree.
  The \defemph{\(b\)-variety} \(\bvar D\) of the diagram \(D\) is the set of all \(b_1, \dots, b_s, m_1, \dots, m_c\) satisfying the segment equations.
\end{definition}

\begin{theorem}
  Write \(\svar{D}[\operatorname{np}]\) for the space of non-pinched \(\chi\)-colorings of \(D\).
  There is an invertible map \(\beta : \bvar D \to \svar{D}[\operatorname{np}]\).
\end{theorem}
\begin{proof}
  The segment equations are written in terms of two \(b_i, m_i\) of the three variables of the colors \(\chi_i = (a_i, b_i, m_i)\).
  When the segment equations are satisfied each segment is also assigned a well-defined \(a\)-variable \(a_i\), as discussed above.
  These assignments determine a \(\chi\)-coloring because the equations \eqref{eq:positive-a-relations} imply the braiding relations \(B(\chi_1, \chi_2) = (\chi_{2'}, \chi_{1'})\) at every positive crossing, and similarly for negative crossings.
  Thus \(\beta\) is well-defined.

  We can only compute non-pinched \(\chi\)-colorings in this manner because at a pinched crossing the expressions appearing in \eqref{eq:positive-a-relations} and \eqref{eq:negative-a-relations} are indeterminate.
  Conversely, any non-pinched \(\chi\)-coloring determines a solution of the \(b\)-gluing equations by forgetting the \(a\)-variables.
  This gives an inverse to \(\beta\).
\end{proof}

\begin{example}
  \label{ex:figure-eight-parabolic}
  Consider the diagram \(D\) of the figure-eight knot given in \cref{fig:figure-eight-labeled}.
  For simplicity, we restrict to the boundary-parabolic case where the meridian eigenvalue \(m\) is \(1\).
  We assign \(b\)-variables \(b_1, \dots, b_8\) to the segments of \(D\).
  The equation for segment \(1\) is
  \[
    \frac{b_{4}}{ b_{5}} \frac{ b_{1} -  b_{5}}{b_{1} -  b_{4}}
    =
    \frac{ b_{5} -  b_{1}}{ b_{6} - b_{1}}
  \]

  Following \cite[Example 4.6]{Kim2016} we see that there is a \(3\)-parameter family of solutions given in terms of \(p,q,r\) by
  \begin{align*}
    (b_1, \dots, b_8)
    &=
    \left(
      pr, pr(1 + q \Lambda), - \frac{pr \Lambda(1 + q\Lambda)}{1-p}, \frac{pqr}{1-p}
    \right.
    \\
    &\phantom{=}
    \left.
      -qr, r-qr, - \frac{pr(1-q)\Lambda^2}{1 + p\Lambda}, \frac{pr}{1 + p\Lambda}
    \right)
  \end{align*}
  where \(\Lambda\) satisfies  \(\Lambda^2 + \Lambda + 1\).
  The solution space is parametrized by one discrete parameter \(\Lambda\) and three continuous parameters \(p,q,r\), which can be freely chosen as long as we avoid pinched crossings and all the \(b_i\) are nonzero.
\end{example}

Solving for all of \(\bvar D\) and not just the part with \(m = 1\) is more complicated.
We discuss this further and compute more examples in \cref{sec:twist-region}.

\begin{remark}
  \label{rem:extra-parameters}
  In general our solutions have three extra parameters.
  \textcite[Example 4.6]{Kim2016} explain this as follows: one degree of freedom comes from the homogeneity of the segment equations in the \(b_i\), while the other two come from the arbitrary locations of the extra ideal points \(P_{\pm}\) of the octahedral decomposition (\cref{sec:geometry}).
  This can be seen explicitly using the gauge transformation formulas of \cite{McphailSnyder2024octahedralcoordinateswirtingerpresentation}.
\end{remark}

\subsection{The \texorpdfstring{\(a\)}{a}-variety \texorpdfstring{\(\avar D\)}{A\_D} and the region equations}
\label{sec:region-equations}
Instead of using \cref{prop:b-gluing-relations} to eliminate the \(a\)-variables, we can use \cref{lemma:a-gluing-relations} to eliminate the \(b\)-variables.
However, it is most convenient to do this using a slightly different set of variables.

\begin{definition}
  \label{def:region-variables}
  Let \(D\) be a diagram of a link \(L\) with \(s\) segments; because \(D\) is a \(4\)-valent planar graph there are \(f +1 = 2 + s/2\) regions, one of which is unbounded.
  To any \(\chi\)-coloring of \(D\) we associate \(f\) \defemph{region variables} \(r_0, \dots, r_{f}\) by the rule given in \cref{fig:region-variable-rule}:
  when passing from \(r_{i}\) to \(r_{i'}\) across a strand with color \(\chi = (a, b, m)\) we have \(r_{i'} = a r_{i}\).
\end{definition}

\begin{marginfigure}
  %% Creator: Inkscape 1.2.2 (b0a8486541, 2022-12-01), www.inkscape.org
%% Accompanies image file 'region-variable-rule.pdf' (pdf, eps, ps)
\begingroup%
  \makeatletter%
  \providecommand\color[2][]{%
    \errmessage{(Inkscape) Color is used for the text in Inkscape, but the package 'color.sty' is not loaded}%
    \renewcommand\color[2][]{}%
  }%
  \providecommand\transparent[1]{%
    \errmessage{(Inkscape) Transparency is used (non-zero) for the text in Inkscape, but the package 'transparent.sty' is not loaded}%
    \renewcommand\transparent[1]{}%
  }%
  \providecommand\rotatebox[2]{#2}%
  \newcommand*\fsize{\dimexpr\f@size pt\relax}%
  \newcommand*\lineheight[1]{\fontsize{\fsize}{#1\fsize}\selectfont}%
  \ifx\svgwidth\undefined%
    \setlength{\unitlength}{109.92816067bp}%
    \ifx\svgscale\undefined%
      \relax%
    \else%
      \setlength{\unitlength}{\unitlength * \real{\svgscale}}%
    \fi%
  \else%
    \setlength{\unitlength}{\svgwidth}%
  \fi%
  \global\let\svgwidth\undefined%
  \global\let\svgscale\undefined%
  \makeatother%
  \begin{picture}(1,0.3133619)%
    \lineheight{1}%
    \setlength\tabcolsep{0pt}%
    \put(0,0){\includegraphics[width=\unitlength,page=1]{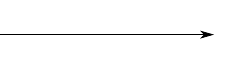}}%
    \put(0.341132,0.23042404){\makebox(0,0)[lt]{\lineheight{1.25}\smash{\begin{tabular}[t]{l}$r_i$\end{tabular}}}}%
    \put(0.17738868,0.02574484){\makebox(0,0)[lt]{\lineheight{1.25}\smash{\begin{tabular}[t]{l}$r_j = a r_i$\end{tabular}}}}%
    \put(0.95516958,0.16219765){\makebox(0,0)[lt]{\lineheight{1.25}\smash{\begin{tabular}[t]{l}$a_i$\end{tabular}}}}%
  \end{picture}%
\endgroup%

  \caption{The correspondence between region variables and \(a\)-variables.}
  \label{fig:region-variable-rule}
\end{marginfigure}

\begin{lemma}
  \label{lemma:shapings-give-region-vars}
  Any \(\chi\)-coloring of \(D\) gives well-defined region variables.
\end{lemma}
\begin{proof}
  The rules of \cref{def:region-variables} determine the \(r_i\) up to an overall constant from any \(\chi\)-coloring: we can walk from the unbounded region of \(D\) to any other region, picking up factors of \(a_i\) in the process.
  We need to make sure this assignment is well-defined.

  It's enough to check it's well-defined near each crossing.
  Given a choice of \(r_N\) at a crossing (labeled as in \cref{fig:crossing-regions}), we have both
  \[
    r_S = a_1 a_2 r_N
    \text{ and } 
    r_S = a_{1'} a_{2'} r_N.
  \]
  However, these give the same value for \(r_S\) because by \eqref{eq:a-transf-positive} and \eqref{eq:a-transf-negative} we have \(a_1 a_2 = a_{1'} a_{2'}\) at any crossing.
\end{proof}

We want to work the other way and use the region variables (and meridian eigenvalues) to determine a \(\chi\)-coloring.
Suppose \(D\) has  \(c\) components and we have chosen meridian eigenvalues \(m_1, \dots, m_c\) and region variables \(r_0, \dots, r_{f}\).
\cref{lemma:a-gluing-relations} gives the ratios of the \(b\)-variables near any crossing; given an arbitrary choice of one \(b_i\), this determines the rest of them and should give a \(\chi\)-coloring.
However, there is a consistency condition: for these ratios to come from an assignment of \(b\)-variables to each segment bounding a region, the product of the ratios (one from each corner) must be \(1\).

\begin{definition}
  Let \(D\) be a diagram with region variables \(\{r_i\}\) and meridian eigenvalues \(\{m_i\}\).
  At non-degenerate positive crossings, the \defemph{corner terms} are
  \begin{equation}
    \begin{aligned}
      k_N
      &=
      \frac{r_W r_E - r_N r_S}{(r_W - m_1 r_N)(r_E - r_N/m_2)}
      \\
      k_W
      &=
      \frac{(r_N - r_W/m_1)(r_S - r_W/m_2)}{r_N r_S - r_W r_E}
      \\
      k_S
      &=
      \frac{r_W r_E - r_N r_S}{(r_E - r_S/m_1)(r_W - m_2 r_S)}
      \\
      k_E
      &=
      \frac{(r_S - m_1 r_E)(r_N - m_2 r_E)}{r_N r_S - r_W r_E}
    \end{aligned}
  \end{equation}
  and at non-degenerate negative crossings they are
  \begin{equation}
    \begin{aligned}
      k_N
      &=
      \frac{(r_W - r_N/m_1)(r_E - m_2 r_N)}{r_W r_E - r_N r_S}
      \\
      k_W
      &=
      \frac{r_N r_S - r_W r_E}{(r_N - m_1 r_W)(r_S - m_2 r_W)}
      \\
      k_S
      &=
      \frac{(r_E - m_1 r_S)(r_W - r_S/m_2)}{r_W r_E - r_N r_S}
      \\
      k_E
      &=
      \frac{r_N r_S - r_W r_E}{(r_S - r_E/m_1)(r_N - r_E/m_2)}
    \end{aligned}
  \end{equation}
  Here by \(r_N\) we mean the region variable north of the crossing (viewed left-to-right) as in \cref{fig:crossing-regions} and similarly for \(r_E, r_S, r_W\).
  The \defemph{region equation} associated to any region of \(D\) says that the product of all the corner terms near a region is \(1\).
  We call the set \(\avar D\) consisting of solutions \((r_0, \dots, r_{s-2}, m_1, \dots, m_2)\) to the region  equations the \defemph{\(a\)-shape variety} of \(D\).
  We require that solutions in \(D\) satisfy the non-degeneracy conditions, which in terms of the region variables are
  \begin{equation}
    r_W \ne  m_1 r_N
  \end{equation}
\end{definition}

\begin{marginfigure}
  \centering
  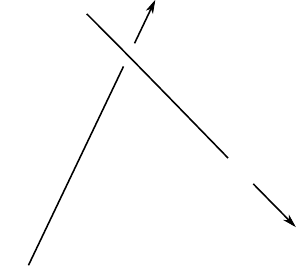
  \caption{A diagram region with three edges.}
  \label{fig:region-gluing-example-i}
\end{marginfigure}

\begin{example}
  In \cref{fig:region-gluing-example-i} the central region labeled \(7\) has three segments and three corners.
  The corner terms are
  \begin{align*}
    k_{13}
    &=
    \frac{r_1 r_3 - r_2 r_7}{(r_3 - r_7/m_2)(r_1 - m_1 r_7)}
    \\
    k_{35}
    &=
    \frac{r_3 r_5 - r_4 r_7}{(r_5 - r_7/m_3)(r_3 - m_2 r_7)}
    \\
    k_{51}
    &=
    \frac{r_1 r_5 - r_6 r_7}{(r_1 - r_7/m_1)(r_5 - m_3 r_7)}
  \end{align*}
  and the region equation is
  \[
    k_{13} k_{35} k_{51} = 1.
  \]
  The non-degeneracy relations at each crossing are
  \[
    r_1r_3 \ne r_2 r_5\text{, } r_3 r_5 \ne r_4r_7 \text{, and } r_1 r_5 \ne r_6 r_7 .
  \]
\end{example}

\begin{remark}
  Because the region equations are homogeneous in the region variables \(\set{r_i}\) they have an extra degree of freedom.
  We can remove this by fixing the variable for some region; an obvious choice is to fix the value for the unbounded region as \(r_0 = 1\).
\end{remark}

\begin{theorem}
  Write \(\svar{D}[\operatorname{nd}]\) for the space of non-degenerate \(\chi\)-colorings of \(D\).
  Then is a map \(\alpha : \avar D \to \svar{D}[\operatorname{nd}]\) with a right inverse.
\end{theorem}
\begin{proof}
  A point of \(\avar D\) is a choice of meridian eigenvalue \(m_i\) for each component and a choice of region variables, which uniquely determines the \(a\)-variables.
  The relations \eqref{eq:positive-b-relations} and \eqref{eq:negative-b-relations} determine the ratios between the \(b\)-variables of every segment, so if we pick the \(b\)-variable \(b_1\) of one segment arbitrarily we determine all of them.
  This defines the map \(\alpha\).

  To define a right inverse of \(\alpha\) let \(\chi\) be a non-degenerate \(\chi\)-coloring of \(D\).
  Choose the region variable of one region (say, the topmost one) to be  \(1\).
  Then by \cref{lemma:shapings-give-region-vars} the \(\chi\)-coloring gives well-defined region variables for the other regions, so we have defined a point \(A\) of \(\avar D\).
  It is clear that \(\alpha(A) = \chi\).
\end{proof}

\begin{example}
  Consider the diagram of the trefoil knot in \cref{fig:trefoil-regions} with labeled regions.
  For simplicity, assume that the meridian eigenvalue \(m\) is \(1\).
  Then the region equation for region \(2\) is
  \[
    \frac{(r_4 - r_2)(r_1 - r_2)}{r_1 r_4 - r_0 r_2} \frac{(r_1 -r_2)(r_4 - r_2)}{r_1 r_4 - r_2 r_3} = 1
  \]
  There are similar equations for the other regions.
  Two solutions to the region gluing equations are given by \cite[Example 4.11]{Kim2016}
  \begin{gather}
    \left(r_0, r_1, r_2, r_3, r_4\right)
    =
    \left(
      1
      ,
      \frac{(q - p)}{1+q - p}
      ,
      \frac{q - p + pq}{1+q - p}
      ,
      \frac{1 + 2q + pq}{1+q - p}
      ,
      \frac{1 + q + pq}{1 + q - p}
    \right)
    \intertext{and}
    \left(r_0, r_1, r_2, r_3, r_4\right)
   =
    \left(
      1,
      p,
      1,
      1,
      2-p
    \right)
  \end{gather}
  Here we can choose \(p,q\) arbitrarily as long as the non-degeneracy conditions are satisfied.
  The first family of solutions has nonabelian holonomy, while the second is abelian. 
  The abelian family does not correspond to a point of \(\bvar D\) because \(\chi\)-colorings with abelian holonomy are necessarily pinched.
\end{example}

\begin{marginfigure}
  \centering
  %% Creator: Inkscape 1.2.2 (b0a8486541, 2022-12-01), www.inkscape.org
%% Accompanies image file 'trefoil-regions.pdf' (pdf, eps, ps)
\begingroup%
  \makeatletter%
  \providecommand\color[2][]{%
    \errmessage{(Inkscape) Color is used for the text in Inkscape, but the package 'color.sty' is not loaded}%
    \renewcommand\color[2][]{}%
  }%
  \providecommand\transparent[1]{%
    \errmessage{(Inkscape) Transparency is used (non-zero) for the text in Inkscape, but the package 'transparent.sty' is not loaded}%
    \renewcommand\transparent[1]{}%
  }%
  \providecommand\rotatebox[2]{#2}%
  \newcommand*\fsize{\dimexpr\f@size pt\relax}%
  \newcommand*\lineheight[1]{\fontsize{\fsize}{#1\fsize}\selectfont}%
  \ifx\svgwidth\undefined%
    \setlength{\unitlength}{145.25805473bp}%
    \ifx\svgscale\undefined%
      \relax%
    \else%
      \setlength{\unitlength}{\unitlength * \real{\svgscale}}%
    \fi%
  \else%
    \setlength{\unitlength}{\svgwidth}%
  \fi%
  \global\let\svgwidth\undefined%
  \global\let\svgscale\undefined%
  \makeatother%
  \begin{picture}(1,0.79602985)%
    \lineheight{1}%
    \setlength\tabcolsep{0pt}%
    \put(0,0){\includegraphics[width=\unitlength,page=1]{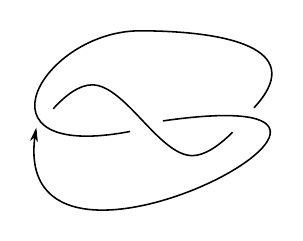}}%
    \put(0.34147504,0.74364961){\color[rgb]{0,0,0}\makebox(0,0)[lt]{\lineheight{1.25}\smash{\begin{tabular}[t]{l}$0$\end{tabular}}}}%
    \put(0.48857561,0.54105698){\color[rgb]{0,0,0}\makebox(0,0)[lt]{\lineheight{1.25}\smash{\begin{tabular}[t]{l}$1$\end{tabular}}}}%
    \put(0.27635904,0.39386475){\color[rgb]{0,0,0}\makebox(0,0)[lt]{\lineheight{1.25}\smash{\begin{tabular}[t]{l}$2$\end{tabular}}}}%
    \put(0.61717874,0.32891186){\color[rgb]{0,0,0}\makebox(0,0)[lt]{\lineheight{1.25}\smash{\begin{tabular}[t]{l}$3$\end{tabular}}}}%
    \put(0.33866738,0.2043193){\color[rgb]{0,0,0}\makebox(0,0)[lt]{\lineheight{1.25}\smash{\begin{tabular}[t]{l}$4$\end{tabular}}}}%
  \end{picture}%
\endgroup%

  \caption{A diagram of the trefoil knot with labeled regions.}
  \label{fig:trefoil-regions}
\end{marginfigure}

\subsection{The segment equations of a twist region}
\label{sec:twist-region}
\begin{figure}
  \centering
  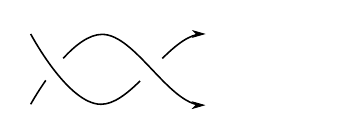
  \caption{
    A parallel twist region with $N$ positive crossings.
    Here we have labeled the \(b\)-variables of the top and bottom segments as \( x_i\) and \( y_i\).
  }
  \label{fig:twist-region}
\end{figure}
A \defemph{twist region} of a knot diagram with \(N\) positive twists is shown in \cref{fig:twist-region}.
We call it a \defemph{parallel} twist region because both strands are oriented in the same direction.
\textcite[Section 6]{Kim2016} showed how to solve the segment equations in a twist region in when the holonomy representation is boundary-parabolic (that is, when \(m = \pm 1\)).
By doing so, they can compute \(\slg\) representations on some infinite families of knots like \((2,N)\) torus knots and twist knots.
In this section we translate their computation to our conventions and explain how to extend it to all irreducible representations of \((2,N)\) torus knots.

\begin{definition}
  A sequence $\{F_i\}_{i \in \ZZ}$ is  \defemph{$W$-Fibonacci} for $W \in \CC$ if it satisfies
  \[
    F_{i+1} = W \cdot F_i + F_{i-1}.
  \]
  The sequence $\{B_i\}$ with $B_0 = 0$ and $B_1 = 1$ is called the \defemph{base} $W$-Fibonacci sequence.
\end{definition}

\begin{lemma}[\protect{ \cite[Lemma 6.2]{Kim2016} }]
  \label{lemma:fibonacci}
  Let \(\{F_i\}\) and \(\{G_i\}\) be \(W\)-Fibonacci sequences and \(\{B_i\}\) the base \(W\)-Fibonacci sequence.
  Then for all \(i\),
  \begin{enumerate}[(a)]
    \item \(F_i = F_0 B_{i-1} + F_1 B_i\),
    \item \(B_{i}^{2} - B_{i-1} B_{i+1} = (-1)^{i+1}\),
    \item and if \(i \ge 0\)
      \[
        B_{i+1} = \sum_{0 \le j \le i/2} \binom{i-j}{j} W^{i-2j}.\qedhere
      \]
  \end{enumerate}
\end{lemma}

\begin{lemma}
  \label{lemma:twist-rels}
  Consider a parallel twist region with $N$ positive twists labeled as in \cref{fig:twist-region} in which both strands have meridian eigenvalue \(m\).
  If the equations
  \begin{equation}
    \label{eq:twist-rels-positive}
    x_i = \frac{F_i}{G_i} \quad 
    y_i = \frac{F_{i-1}}{m G_{i+1}}
  \end{equation}
  hold for $i=1,2$ (i.e.\@ for segments $1, 2, 3, 4$) then they hold for all $1 \le i \le n+1$.
  In a region with $n$ negative twists, the same holds with \eqref{eq:twist-rels-positive} replaced with
  \begin{equation}
    \label{eq:twist-rels-negative}
    x_{i } = \frac{F_{i-1}}{G_{i+1}} \quad 
    y_{i} = \frac{F_{i}}{m G_{i}} \qedhere
  \end{equation}
\end{lemma}

\begin{proof}
  Consider the gluing equation for the segment labeled \(x_i\).
  Because we are using \(b\)-variables, we think of the four \(b\)-variables associated to the segments at each crossing as determining the \(a\)-variables via \eqref{eq:positive-a-relations}.
  The gluing equation of a segment is then checking that the \(a\)-variables for each side agree.
  In this case, it is
  \[
    \frac{x_{i-1}}{m y_{i}} \frac{m x_{i} - y_{i}}{x_{i} - x_{i-1}}
    =
    \frac{ y_i - m  x_i}{ x_{i+1}- x_i}
  \]
  Solving for \(x_{i+1}\) and repeating this argument for the segment \(y_{i}\) gives the recurrence relations
  \begin{align*}
     x_{i+1} &= m y_i -  \frac{m x_i  y_i}{ x_{i-1}} + x_i 
    \\
     \frac{1}{y_{i+1}} &= \frac{m}{x_i} - \frac{m y_{i-1}}{x_{i} y_{i}} + \frac{1}{y_i}
  \end{align*}
  (These are the relations of \cite[Lemma 6.3]{Kim2016} up to some factors of \(m\).)
  The result follows by induction, and the negative case works similarly.

  To give some details, substitute \eqref{eq:twist-rels-positive} into the recurrence for \(x_{i+1}\) to obtain
  \begin{align*}
    \frac{F_{i+1}}{G_{i+1}} &= \frac{F_{i-1}}{G_{i+1}} - \frac{F_{i} G_{i-1}}{G_{i} G_{i+1}} + \frac{F_{i}}{G_{i}}
    \intertext{equivalently}
    F_{i+1} &= G_{i}^{-1} \left[ G_{i} F_{i-1} + F_{i} (G_{i+1} -G_{i-1} ) \right]
  \end{align*}
  After applying the recurrence for \(G_{i+1}\) the right-hand side becomes
  \begin{align*}
    G_{i}^{-1} \left[ G_{i} F_{i-1} + F_{i} (W G_{i} + G_{i-1} -G_{i-1} ) \right]
    &=
    F_{i-1} +  W F_{i} \\
    &=
    F_{i+1}
  \end{align*}
  as required.
  Something similar works for the recurrence for \(y_i\).
\end{proof}

\begin{proposition}
  \label{prop:fib-ansatz-parallel}
  In a parallel twist region with positive twists in which both strands have meridian eigenvalue \(m\),
  the adjusted segment variables are given in terms of \( x_1,  x_2,  y_1,  y_2\) by
  \begin{align}
    \label{eq:parallel-twist-solution-x}
     x_i &=
    \frac{ m x_1  y_1 W \fib{i-1}(W) +  x_1 ( x_2 -  m y_1) \fib{i}(W) }{( x_2 -  m y_1) \fib{i-2}(W) +  x_1 W \fib{i-1}(W) }
    \\
    \label{eq:parallel-twist-solution-y}
     y_i &=
    m^{-1} \frac{ m x_1  y_1 W \fib{i-2}(W) +  x_1 ( x_2 -  m y_1) \fib{i-1}(W) }{( x_2 -  m y_1) \fib{i-1}(W) +  x_1 W \fib{i}(W) }
  \end{align}
  where
  \[
    W^2 = (m y_1 -  x_2) \left( \frac{1}{ x_1} - \frac{1}{m y_2} \right),
  \]
  \(\fib{i}\) is the polynomial \(\fib{i}(W) = B_i\), and \(B_i\) is the base \(W\)-Fibonnaci sequence determined by
  \[
    B_{i+1} = W B_i + B_{i-1}, B_1 = 1, B_0 = 0.\qedhere
  \]
\end{proposition}

\begin{remark}
  \label{rem:sqrt-W}
  The \(\fib i\) are sometimes called the \defemph{Fibonacci polynomials}.
  When \(i\) is odd \(\fib i(W)\) is a polynomial in \(W^2\), and when \(i\) is even \(\fib i(W)\) is \(W\) times a polynomial in \(W^2\).
  In particular, the solutions in (\ref{eq:parallel-twist-solution-x}--\ref{eq:parallel-twist-solution-y}) depend only on \(W^2\), not \(W\).
\end{remark}

\begin{proof}
  We need to pick the right initial conditions for \(F_i\) and \(G_i\).
  A convenient way to do this is to set
  \begin{align*}
    F_0 &=  m  W x_1  y_1
        &
    F_1 &=  x_1( x_2 -  m y_1)
    \\
    G_1 &=  x_2 -  m y_1
        &
    G_2 &=  W x_1 .
  \end{align*}
  Then we can check that
  \[
    x_i = \frac{F_i}{G_i} \text{ and } y_i = \frac{F_{i-1}}{m G_{i+1}}
  \]
  holds for \(i = 1,2\), so by \cref{lemma:twist-rels} they hold for all \(i\).
  We can now apply (a) of \cref{lemma:fibonacci}
\end{proof}

\begin{remark}
  \label{rem:twist-different-eigvals}
  To compute more examples, we would need to extend this computation to:
  \begin{enumerate}
    \item parallel twist regions where the meridian eigenvalues \(m_1, m_2\) on the two strands differ, and
    \item to antiparallel twist regions in the boundary non-parabolic (\(m \ne \pm 1\)) case.
  \end{enumerate}
  In fact, the segment equations in these two cases are closely related, because there is a simple formula \cite[Definition 4.2]{McPhailSnyderThesis} for reversing the orientation of a strand.
  Following \cite[Section 7]{Ham2018}, the right generalization is to consider sequences of the form
  \[
    A_{i+1} = W A_i + \frac{m_i}{m_{i+1}} A_{i-1}
  \]
  where the index on \(m_i\) is understood mod  \(2\).
\end{remark}

We can at use \cref{prop:fib-ansatz-parallel} to compute an infinite family of boundary non-parabolic examples.
A \defemph{\((2,N)\)-torus knot} is obtained by attaching segments \(x_1\) and \(x_{N+1}\) and segments \(y_1\) and \(y_{N+1}\) in \cref{fig:twist-region}.
We assume \(N = 2n+1\) is odd, so that we obtain a knot and not a link.
We can now solve for its segment variables.
We think of of \(x_1\), \(x_2\), and \(y_1\) as parameters and use the ansatz
\[
  x_i = \frac{F_i}{G_i} \text{ and } y_i = \frac{F_{i-1}}{m G_{i+1}}
\]
of \cref{prop:fib-ansatz-parallel}.
We need to choose \(\Lambda\) (equivalently, choose \(y_2\)) so that the gluing equations of the edges \(x_1 = x_{N+1}\) and \(y_1 = y_{N+1}\) are satisfied.
The former is
  \[
    \frac{x_{N}}{m y_{1}} \frac{m x_{1} - y_{1}}{x_{1} - x_{N}}
    =
    \frac{ y_{1} - m  x_{1}}{ x_{2}- x_{1}}
  \]
or
\[
  1 = \frac{x_N}{m y_{1}} \frac{x_1 - x_2}{x_1 - x_{N}}
\]
which is equivalent to 
\begin{equation}
  \label{eq:torus-poly-deriv}
  \frac{G_{N}}{F_{N}} = \frac{G_{0}}{F_{0}}.
\end{equation}
Now, the sequence \(H_i = F_0 G_i - F_i G_0\) is a \(\sqrt\Lambda\)-Fibonacci sequence satisfying \(H_0 = 0\) and \(H_1 \ne 0\), and \eqref{eq:torus-poly-deriv} holds if and only if \(H_N = 0\).
By part (a) of \cref{lemma:fibonacci} we have \(H_i = H_1 B_i\) for all \(i\), and we conclude that \eqref{eq:torus-poly-deriv} holds if and only if \(B_N = 0\).
Using the fact that \(N = 2n+1\) is odd, the condition is
\begin{align*}
  B_{2n+1}
  &= \sum_{0 \le j \le (2n+1)/2} \binom{2n-j}{j} \sqrt{\Lambda}^{2n-2j}
  \\
  &=
  \sum_{0 \le j \le n} \binom{2n - j}{j} {\Lambda}^{n - j}
\end{align*}
It turns out that this also implies the gluing relation for \(y_{1} = y_{N+1}\).
We have shown:

\begin{theorem}
  \label{thm:torus-knot-solutions}
  Taking the braid closure of \cref{fig:twist-region} for \(N = 2n+1\) gives a diagram \(D\) of a \((2, 2n+1)\)-torus knot.
  For any meridian eigenvalue \(m\) the \(b\)-variables of a \(\chi\)-coloring of \(D\) are given by
  \begin{align}
    \label{eq:torus-knot-solution-x}
     x_i &=
    \frac{ q \sqrt{\Lambda} \fib{i-1}(\sqrt{\Lambda}) +  pr \fib{i}(\sqrt{\Lambda}) }{r \fib{i-2}(\sqrt{\Lambda}) +  \sqrt{\Lambda} \fib{i-1}(\sqrt{\Lambda}) }
    \\
    \label{eq:torus-knot-solution-y}
     y_i &= \frac{1}{m}
    \frac{ q \sqrt{\Lambda} \fib{i-2}(\sqrt{\Lambda}) +  pr \fib{i-1}(\sqrt{\Lambda}) }{r \fib{i-1}(\sqrt{\Lambda}) +  \sqrt{\Lambda} \fib{i}(\sqrt{\Lambda}) }
  \end{align}
  where \(\Lambda\) satisfies
  \begin{equation}
    \label{eq:trefoil-riley}
    \sum_{0 \le j \le n} \binom{2n - j}{j} {\Lambda}^{n - j}  = 0
  \end{equation}
  and
  \[
    p =  x_1, \quad q = m  y_1 ,\quad r = ( x_2 - m y_1)/ x_1
  \]
  are arbitrary nonzero parameters chosen so that the first crossing is not pinched.\note{\(p,q,r\) are closely related to but not exactly the variables \(p,q,r\) in \cite[Theorem 6.5]{Kim2016}.}
  As discussed in \cref{rem:sqrt-W} the expressions for \( x_i\) and \( y_i\) depend only on \(\Lambda\), not \(\sqrt \Lambda\).
\end{theorem}

\begin{remark}
  The polynomial in \eqref{eq:trefoil-riley} is sometimes called the \defemph{Riley polynomial} of the knot, and this specific case is discussed in \cite[Section 5]{Riley1972}.
  Riley studied these polynomials for two-bridge knots, but they can be defined more generally and are a useful computational tool \cite{Cho2022}.

  In the present case \eqref{eq:trefoil-riley} is a monic polynomial with \(n\) distinct roots.
  The roots of \eqref{eq:trefoil-riley} can be given explicitly \cite[Theorem 5]{Riley1972}.
  In particular, when \(n =1\) the only root is \(\Lambda = -1\).
  This gives us the solutions \eqref{eq:trefoil-soln} for the trefoil knot.
\end{remark}

\begin{remark}
  The geometry of these solutions is not affected by the choice of \(p,q,r\), but is instead determined by \(m\) and the choice of root \(\Lambda\) of the Riley polynomial.
  In our conventions this is less obvious, but there is a different presentation \cite[Section 5]{Kim2016} of the holonomy of a shaped diagram that is manifestly independent of the choice of \(p,q,r\).

  We can still give an informal explanation of this independence.
  In general \cite[Theorem 3.1]{Munoz2009} the character variety of a  \((p,q)\) torus knot has \((p-1)(q-1)/2\) components coming from irreducible representations.\note{There is also a component describing the representations with abelian image. It is not detected by the \(b\)-gluing equations because they only detect non-pinched solutions, which always correspond to a nonabelian holonomy representation.}
  In our family of examples \(p = 2\) and \(q = 2n +1\) so there are \(n\) components, each corresponding to a root \(\Lambda\) of \eqref{eq:trefoil-riley}.
  Each component arising in this way is \(1\)-dimensional \cite[Theorem 3.1]{Munoz2009} and parametrized by a rational function of \(m\).
  (The components are isomorphic to \(\CC\) via the trace \(\mu + \mu^{-1}\) of a generator of the knot group, but this generator is not a meridian, so \(\mu\) is some rational function of \(m\).)
\end{remark}

\begin{example}
  \label{ex:trefoil}
  Consider the diagram of the trefoil in \cref{fig:trefoil}.
  For any meridian eigenvalue \(m \ne 0\), we can freely choose the variables \(b_1, b_2, b_3\), as long as
  \[
    \frac{b_2}{m b_1}, \frac{b_3}{b_1} \ne 1
  \]
  so that the solution is not pinched.
  The remaining segment variables are given by
  \begin{equation}
    \label{eq:trefoil-soln}
    \left(b_4, b_5, b_6\right)=
    \left(
      \frac{b_1(m b_2 - b_3)}{m(b_1 + m b_2 - b_3)},
      \frac{m b_1 b_2}{b_1 + m b_2 - b_3},
      -\frac{b_1 b_3}{m(m b_2 - b_3)}
    \right).
  \end{equation}
  This solution is the case \(n = 1\) of \cref{thm:torus-knot-solutions} after the substitutions \(p = b_1, q = m b_2, r = (b_3 - m b_2)/b_1\).

  We can use \cref{thm:decoration-octahedral} to compute the boundary eigenvalues of this representation.
  We have \(\delta(\mer) = m\) and
  \[
    \ell = \delta(\lon)
    =
    m^{-3} \frac{b_2 b_4 b_6}{b_1 b_3 b_5}
    =
    -1/m^{6}
  \]
  which matches the \(A\)-polynomial \(m^{6} \ell + 1\) of the trefoil knot.
\end{example}

As before, the choice of \(b_1, b_2, b_3\) does not affect the conjugacy class of the representation, which is uniquely determined by \(m\).
In this case, we can check this against the character variety of the trefoil knot.
Non-pinched solutions correspond to representations with nonabelian image, and the \(\slg\)-character variety of the trefoil has a single nonabelian component \cite[Theorem 3.1]{Munoz2009} cut out by the equation
\[
  m^{6} \ell +1 = 0.
\]
This shows explicitly that \eqref{eq:trefoil-soln} yields at least one representation every nonabelian conjugacy class.
We can do a similar computation for general \((2,2n+1)\)-torus knots:

\begin{example}
  Consider the \(\chi\)-coloring of the \((2,N)\)-torus knot given in \cref{thm:torus-knot-solutions}.
  The longitude is
  \begin{align*}
    \delta(\lon)
    &=
    \frac{1}{m^{N}} \prod_{k=1}^{N} \frac{x_i}{y_i}
    =
    \frac{1}{m^{N}} \prod_{k=1}^{N} \frac{F_{k-1}}{m G_{k+1}} \frac{G_k}{ F_k}
    \\
    &=
    \frac{1}{m^{2N}} \frac{F_0 G_1}{ F_{N} G_{N+1}}
    =
    \frac{1}{m^{2N}} \frac{G_0 G_1}{ G_{N} G_{N+1}},
  \end{align*}
  where in the last step we used \eqref{eq:torus-poly-deriv}.
  Because \(B_N = 0\), by \cref{lemma:fibonacci} we have
  \begin{align*}
    \frac{G_0 G_1}{ G_{N} G_{N+1}}
    &=
    \frac{G_0 G_1}{
      ( G_0 B_{N-1} + G_1 B_{N}) 
      ( G_0 B_{N} + G_1 B_{N + 1}) 
    }
    \\
    &=
    \frac{1}{ B_{N-1} B_{N}}
    =
    (-1)^{N} = -1
  \end{align*}
  because \(N = 2n+1\) is odd.
  We conclude that
  \begin{equation}
    \delta(\lon) = -m^{-2N}
  \end{equation}
  which recovers the \(A\)-polynomial of the right-handed \((2,N)\)-torus knot.
\end{example}
The discrete parameter \(\Lambda\) does not appear in this computation; this corresponds to the fact \cite[Theorem 3.1]{Munoz2009} that the \(n\) nonabelian components of the character variety are all isomorphic.
Choosing \(\Lambda\) picks out a component which is then continuously parametrized by \(m + m^{-1}\).

\appendix
\printbibliography
\end{document}